\documentclass[11pt]{article}
\usepackage[english]{babel}
\usepackage{amssymb}
\usepackage{amsmath}
\usepackage{indentfirst}
\usepackage[dvips]{graphicx}
\usepackage{epsfig}
\usepackage{hyperref}

\usepackage{amsthm}
\usepackage{comment}

\newtheorem{cor}{Corollary}[section]

\newtheorem{lemma}[cor]{Lemma}
\newtheorem{teo}[cor]{Theorem}

\newtheorem{rem}[cor]{Remark}

\newcommand{\R}{\mathbb{R}}

\numberwithin{equation}{section}

\begin{document}

\title{A Willmore-Helfrich $L^{2}$-flow of curves \\
with natural boundary conditions}

\author{Anna Dall'Acqua, Paola Pozzi}

%\date{8 October 2012}
\maketitle

\begin{abstract}
We consider regular open curves  in $\mathbb{R}^n$ with fixed boundary
points and moving according to
 the $L^2$-gradient flow for a generalisation of the Helfrich functional.
Natural boundary conditions are imposed along the evolution.
More precisely, at the boundary the curvature vector is
 equal to the normal projection of a fixed given vector. A long-time
 existence result together with  subconvergence to critical points is
proven.
\end{abstract}

 \bigskip
 

\bigskip

%%%%%%%%%%%%%%%%%%%%%%%%%%%%%%%%%%%%%%%%%%%%%%%%%%%%
%%%%%%% Introduzione %%%%%%%%%%%%%%%%%%%%%%%%%%%%%%%%%%
%%%%%%%%%%%%%%%%%%%%%%%%%%%%%%%%%%%%%%%%%%%%%%%%%%%%
\section{Introduction}
\label{Sez1}

In this paper we study the long-time evolution of regular open curves in $\R^{n}$ ($n \geq 2$) moving according to
 the $L^{2}$-gradient flow for a generalization of the Helfrich functional.
 
 The Helfrich energy of a \emph{closed} plane curve $f: \mathbb{S}^{1} \to \R^{2}$ is given by
 \begin{align}
 \label{helfrich}
 \mathcal{H}_{\lambda}(f)=\frac{1}{2} \int_{\mathbb{S}^{1}} (k-c_{0})^{2} ds + \lambda \mathcal{L}(f),
 \end{align}
 where $ds=|f_{x}|dx$ denotes the arc-length, $\vec{\nu}$ the unit normal of the  curve, $k= \langle f_{ss}, \vec{\nu} \rangle$ its scalar curvature and $\mathcal{L}(f)=\int_{\mathbb{S}^{1}} ds$ the length of  $f$.  The map $c_{0}:\mathbb{S}^{1} \to \R$ is called spontaneous curvature. The constant $\lambda \in \R$ is here taken to be positive, so that the growth in length of a curve is penalized.
The above functional is motivated by  the modeling of cell membranes, see \cite{Helfrich}.
Note that if $c_{0}$ is a constant, as we will assume henceforth,  then \eqref{helfrich} reduces to
\begin{equation*}
\mathcal{H}_{\lambda}(f)=\frac{1}{2} \int_{\mathbb{S}^{1}} |k|^{2} ds + \left(\lambda+ \frac{1}{2}c_{0}^{2} \right) \mathcal{L}(f) -2c_{0}\pi \omega,
\end{equation*}
where $\omega$ denotes the winding number of $f$. The special case where $c_{0}=0$ and $\lambda=0$ is sometimes known  as Willmore functional
 and it can also be historically motivated by the so-called Euler-Bernoulli model of elastic rods (see \cite{Truesdell}).
 
A possible generalisation of \eqref{helfrich} to $n$-dimensional closed curves for $n \geq 2$ is given by
\begin{equation}
\label{helfrichn}
\mathcal{H}_{\lambda}(f)=\frac{1}{2} \int_{\mathbb{S}^{1}} |\vec{\kappa}-\vec{c_{0}}|^{2} ds + \lambda \mathcal{L}(f),
\end{equation} 
where now $\vec{\kappa}=\partial_{ss} f$ is the curvature \emph{vector} and $\vec{c_{0}}$ is a given vector in $\R^{n}$. 
Note that since $\int_{\mathbb{S}^{1}} \langle \vec{\kappa}, \vec{c_{0}}\rangle \,ds = \int_{\mathbb{S}^{1}} \langle \partial_{ss}f, \vec{c_{0}}\rangle\, ds=0$  we can view \eqref{helfrichn} as a  natural extension of the classical Helfrich energy.

The Helfrich and Willmore energies are mathematically very interesting  and in particular the Willmore flow
is nowadays considered to be one of the most important models 
in which fourth order PDEs appear. 
Both functionals have 
been extensively investigated analytically and numerically in recent years and the literature is by now rather vast. 
Many of the references we cite  provide extensive information on the history and development of the research on Willmore/Helfrich functionals and related flows, thus we refrain from giving here a thorough account.

In \cite{DKS} the authors study analytically and numerically the long-time evolution of closed curves in $\R^{n}$ moving by the gradient flow of the elastic energy $E(f)=\frac{1}{2}\int_{\mathbb{S}^{1}} |\vec{\kappa}|^{2} ds$: the length of the curves is either a fixed constraint or added as a penalizing term as in \eqref{helfrichn}.
Their work extends  previous results of  \cite{Polden} and \cite{Wen} in the plane.
Further important related work in $\R^3$ can be found in \cite{LangerSinger1}, \cite{LangerSinger2}, and \cite{Koiso}.
In \cite{Wheeler} the author considers \eqref{helfrichn} for closed curves in $\R^n$ and for a specific class of spontaneous curvature vector fields $\vec{c_{0}}$ (in particular $\vec{c_{0}}$ is not required to be constant) and shows global existence of the related flow.
In the graph setting  the stationary problem for the elastic energy of \emph{open} curves subject to different boundary conditions was considered 
  in \cite{DG07}, \cite{DG08}, and  \cite{LinnerJerome}. Lin  investigated in \cite{Lin} the $L^{2}$- gradient flow of elastic curves in $\R^{n}$ with clamped boundary conditions.
In \cite{BGN} several interesting numerical simulations for the elastic flow of open and closed curves in $\R^{n}$ are presented.
An error analysis for a  FEM-approximation of the elastic flow for  curves in $\R^n$ can be found in \cite{DD09}.

Our investigation can be viewed as the next natural research step following the work of  \cite{DKS} and \cite{Lin}.

As already pointed out, here we are concerned with the study of \eqref{helfrichn} for open curves. More precisely we consider a time dependent family of regular curves $f: [0,T) \times \bar{I} \to \mathbb{R}^{n}$, $n \geq 2 $, $I=(0,1)$, with boundary points fixed in time, i.e.  
\begin{align}
\label{intro:bdryI}
f(t,0)=f_{-}, \qquad  f(t, 1)=f_{+} \qquad \forall \, t \in [0,T),
\end{align}
where $f_{-} \neq f_{+} \in \R^{n}$ are given.
For simplicity we write the energy \eqref{helfrichn} as follows
\begin{equation}
W_{\lambda}(f) = \int_{I} \left(\frac12 |\vec{\kappa}|^{2} - \langle \vec{\kappa}, \zeta \rangle \right) ds+ \lambda\int_{I} ds ,
\end{equation}
with $\zeta$ a given vector in $\mathbb{R}^n$ and $\lambda \geq 0$.

The associated $L^{2}$-gradient flow  for the one-parameter family of curves subject to \eqref{intro:bdryI} and to the \emph{natural boundary conditions}
\begin{align}
\label{intro:bdryII}
\vec{\kappa}(t,x)=\zeta- \langle \zeta, \tau(t,x) \rangle \tau(t,x) \qquad x \in \{ 0,1 \},
\end{align}
with $\tau=\partial_s f=\frac{f_{x}}{|f_{x}|}$ unit tangent, leads to the fourth order PDE
\begin{align}
\label{intro:eqflow}
 \partial_{t}f=-\nabla^{2}_{s} \vec{\kappa} -\frac{1}{2} |\vec{\kappa}|^{2} \vec{\kappa} + \lambda \vec{\kappa},  
 \end{align}
 where $\nabla_{s} \phi= \partial_{s} \phi - \langle \partial_{s} \phi, \tau \rangle \tau$ denotes the normal component of $\partial_{s} \phi$.
 
Our main result shows that for smooth initial data $f(0, \cdot)$ the flow exists for all time. 
\begin{teo}\label{mainTh}
Let $\lambda \geq 0$, and 
let vectors $f_{+}, f_{-}, \zeta \in \mathbb{R}^n$ with $f_{+} \ne f_{-}$ be given as well as a smooth regular curve $f_0: \bar{I} \rightarrow \mathbb{R}^n$ satisfying
\begin{align*}
& f_{0}(0)=f_{-} , \; f_{0}(1)=f_{+} \, ,\\
& \kappa[f_0](x) + \langle \zeta, \tau[f_0](x) \rangle \tau[f_0](x) =\zeta\mbox{ for }x\in \{0,1\} \, ,
\end{align*}
with $\vec{\kappa}[f_0]$ and $\tau[f_0]$ the curvature and tangent vector of $f_0$ respectively, together with suitable compatibility conditions (see Appendix \ref{AppendixCompa}). Then  a smooth solution $f: [0,T)\times [0,1] \rightarrow \mathbb{R}^n$ of the initial value problem 
 \begin{equation}\label{ivp}
\left\{\begin{array}{l} 
\partial_{t}f=-\nabla^{2}_{s} \vec{\kappa} -\frac{1}{2} |\vec{\kappa}|^{2} \vec{\kappa} + \lambda \vec{\kappa} \\
f( 0,x)=f_{0}(x) \mbox{ for }x \in [0,1]\\
f(t,0)=f_{-}, f(1,t)=f_{+} \mbox{ for %all 
}t \in [0,T)\\
 \vec{\kappa}(t,x) + \langle \zeta, \tau(t,x) \rangle \tau(t,x)=\zeta \mbox{
   for }x\in \{0,1\} \mbox{ and for %all 
}t \in [0,T),
\end{array} \right.
 \end{equation}
exists for all times, that is we may take $T=\infty$. 
Moreover if $\lambda>0$,  then as $t_{i} \to \infty$  the curves $f(t_{i, \cdot})$ subconverge, when reparametrized by arc-length, to a critical point of the Willmore-Helfrich functional with fixed endpoints, that is to a solution of 
\begin{equation}\label{ebvp}
\left\{\begin{array}{l} 
-\nabla^{2}_{s} \vec{\kappa} -\frac{1}{2} |\vec{\kappa}|^{2} \vec{\kappa} + \lambda \vec{\kappa}  =0 \, ,\\
f(0)=f_{-} \, , f(1)=f_{+} \, ,  \\
 \vec{\kappa}(x) + \langle \zeta, \tau(x) \rangle \tau(x)=\zeta \mbox{ for }x\in \{0,1\} \, .
\end{array} \right.
 \end{equation}
 \end{teo}

The method of proof borrows ideas from \cite{DKS} and \cite{Lin}. In order to motivate better the mathematical constructions that will follow, we recall here some of the most important arguments.

The main strategy is to assume that the flow exists only up to a finite time $T<\infty$ and to show that upper bounds for $\|\partial_{s}^{m} \vec{\kappa} \|_{L^{\infty}}$ hold for any $m \in \mathbb{N}_{0}$, so that we get a contradiction.
In order to obtain such bounds the key step is to look at the quantity (cf. Lemma \ref{lempartint})
\begin{align*}
\frac{d}{dt} \frac{1}{2}\int_{I} |\vec{\phi}|^{2} ds =
\int_{I} \langle \nabla_{t} \vec{\phi}, \vec{\phi} \rangle ds - \frac{1}{2} \int_{I} |\vec{\phi}|^{2} \langle \vec{\kappa}, \vec{V} \rangle ds,
\end{align*} 
where $\vec{V}=-\nabla^{2}_{s} \vec{\kappa} -\frac{1}{2} |\vec{\kappa}|^{2}\vec{\kappa} + \lambda \vec{\kappa}$ denotes the normal velocity of the flow (see \eqref{intro:eqflow}) and $\vec{\phi}$ is an appropriately chosen normal vector field, precisely $\vec{\phi}=\nabla_{s}^{m} \vec{\kappa}$ in \cite{DKS} and $\vec{\phi}=\nabla_{t}^{m} f$ in \cite{Lin} respectively.
In order to be able to bound the right-hand side of the above expression it is wise to add to both sides of the equation the carefully chosen term
\begin{equation}\label{sceltagiusta}
 \int_{I} \langle \nabla_{s}^{4} \vec{\phi}, \vec{\phi} \rangle ds
 \end{equation}
so that after integration by parts one obtains
\begin{align}
\label{alcompleto}
\frac{d}{dt} \frac{1}{2}\int_{I} |\vec{\phi}|^{2} ds + \int_{I}
|\nabla_{s}^{2} \vec{\phi}|^{2 } ds - & [ \langle \nabla_s \vec{\phi}, \nabla_s^2 \vec{\phi} \rangle ]_0^1 + [ \langle  \vec{\phi}, \nabla_s^3 \vec{\phi} \rangle ]_0^1  \notag \\
& =\int_{I} \langle Y, \vec{\phi} \rangle ds - \frac{1}{2} \int_{I} |\vec{\phi}|^{2} \langle \vec{\kappa}, \vec{V} \rangle ds ,
\end{align}  
where $Y=\nabla_{t}\vec{\phi}+ \nabla_{s}^{4} \vec{\phi}$.
The choice of \eqref{sceltagiusta} is dictated by the 
problem itself: indeed if one takes $\vec{\phi}=\vec{\kappa}$ (as in the setting of closed curves studied in \cite{DKS}) and looks at the parabolic equation \eqref{e} satisfied by the curvature, one recognizes that the sum $Y$ has now lower order terms than $\nabla_{t} \vec{\phi}$. The same happens also by taking $\vec{\phi}=\nabla_t f$ (as in \cite{Lin}) and using \eqref{intro:eqflow} and \eqref{E5}. Furthermore, with these choices it turns out that the right-hand side of \eqref{alcompleto} can be controlled by $\int_{I} |\nabla_{s}^{2} \vec{\phi}|^{2 } ds $ with the help of  suitable interpolation inequalities. 

We still have to comment on the boundary terms in \eqref{alcompleto}. In \cite{DKS} they did not actually come into play, because the authors deal with closed curves only. Lin on the contrary, who studied \eqref{intro:eqflow} subject to the \emph{clamped boundary conditions}, namely \eqref{intro:bdryI} together with
\begin{equation}
\label{intro:bdryLin}
\tau(t,0)= \tau_{-}, \qquad \tau(t,1)= \tau_{+} \quad \forall \, t \in [0,T),
\end{equation} 
(for given $\tau_{\pm} \in \R^n$) opted for choosing as $\vec{\phi}$ the only quantity which contains all relevant information about the curvature and which has zero boundary conditions, namely $\vec{\phi}=\nabla_{t}^{m} f$ (note that $\partial_{t}^{m} f$ is zero at the boundary). 
In the setting of Lin  it turns out that \emph{all} boundary terms in \eqref{alcompleto} are zero (see Remark \ref{remLin}).

In our setting the situation is definitely more complicated. Indeed due to the observations above it is still natural to work with $\vec{\phi}=\nabla_{t}^{m }f$ as in \cite{Lin}; however the boundary terms in \eqref{alcompleto} 
do \emph{not} disappear. The strategy here is to use again the 
structure of the equation  to infer that  the ``worst order''  terms are in fact of lower order as at first sight (see Lemma \ref{lembterms} for details) and to bound them with appropriate interpolation inequalities (see \S \ref{Sez3.1}).

The paper is organized as follows. In Section \ref{Sez2}  and \ref{Sez3} we fix the notation and collect a series of technical Lemmas, many of which are induced  by the geometry of the problem. We provide several comments to help the reader to understand both their motivation and derivation. Section \ref{Sez4} deals with interpolation inequalities and finally in Section \ref{Sez5} we give the proof of Theorem \ref{mainTh}.

Although some of the technical lemmas are adaptation to the present setting and notation of results given in \cite{DKS} and \cite{Lin} we would like the paper to be self-contained and therefore report full proofs. Some of them are collected in the Appendix for the sake of readability.
 
Finally, let us remark that, since the next relevant and natural question is to investigate  the evolution  of \eqref{intro:eqflow} subject to either natural or clamped boundary conditions but with a \emph{fixed length} constraint, we have decided to carefully keep track of the parameter $\lambda$ in all proofs.   
This problem will be treated elsewhere.

\bigskip

\textbf{Acknowledgements.} 
This work was partially supported by
DFG Transregional Collaborative Research Centre
SFB~TR~71. The first author gratefully acknowledges the \lq\lq  Deutsche  Forschungsgemeinschaft\rq\rq\ for the project \lq\lq Randwertprobleme f\"ur Willmorefl\"achen - Analysis, Numerik und Numerische Analysis\rq\rq \,(DE 611/5.1) and the grant no. HO4697/1-1. The authors thank Glen Wheeler for drawing their attention to
\cite{Lin}.

%%%%%%%%%%%%%%%%%%%%%%%%%%%%%%%%%%%%%%%%%%%%%%%%%%%%%
%%%%%%% Geometrical Lemmas %%%%%%%%%%%%%%%%%%%%%%%%%%
%%%%%%%%%%%%%%%%%%%%%%%%%%%%%%%%%%%%%%%%%%%%%%%%%%%%%

\section{Preliminaries and geometrical Lemmas}
\label{Sez2}
\subsection{Preliminaries and Notation}

We consider a time dependent curve $f: [0,T) \times \bar{I} \to
\mathbb{R}^{n}$, $f=f(t,x)$,  with $n \geq 2 $, $I=(0,1)$ and with endpoints fixed in time, that is $ f(t, 0) =f_{-}$, $f(t, 1)=f_{+} $ for given vectors $f_{-}, f_{+} \in \mathbb{R}^n$, $f_{-} \ne f_{+}$. 

As usual we denote by $s$ the arc-length parameter. Then $ds = |f_{x}| dx$, $\partial_s = \frac{1}{|f_{x}|} \partial_x$, $\tau = \partial_s f$ is the tangent unit vector and the curvature vector is given by $\vec{\kappa} = \partial_{s s} f$. In the following, vector fields with an arrow on top are normal vector fields. The standard scalar product in $\mathbb{R}^n$ is denoted by $\langle \cdot , \cdot \rangle $, while $\nabla_{ s} \phi$  (resp. $\nabla_{t} \phi$) is the normal component of $\partial_{s} \phi$  (resp. $\partial_{t} \phi$) for  a vector field $\phi$. That is, 
$$\nabla_{ s} \phi = \partial_{ s} \phi - \langle \partial_{ s} \phi, \tau \rangle \tau \, .$$   
The Willmore-Helfrich energy for the curve $f$ is given by
\begin{equation}\label{Wh}
W_{\lambda}(f) = \int_{I} \left(\frac12 |\vec{\kappa}|^{2} - \langle \vec{\kappa}, \zeta \rangle \right) ds+ \lambda\int_{I} ds ,
\end{equation}
where $\zeta$ is a given vector in $\mathbb{R}^n$ and $\lambda \geq 0$ a second parameter.
In this paper we study
\begin{align}\label{eqh}
 \partial_{t}f=-\nabla^{2}_{s} \vec{\kappa} -\frac{1}{2} |\vec{\kappa}|^{2} \vec{\kappa} + \lambda \vec{\kappa} \, ,
 \end{align}
for a smooth regular curve $f$  subject to the boundary conditions
\begin{align}\nonumber
f(t, 0) &=f_{-}\ , \quad f(t, 1)=f_{+} \, , \\  \label{bch}
\vec{\kappa}(t, 0)&= \zeta - \langle \zeta, \tau(t, 0) \rangle \tau(t, 0)\, , \quad \quad \mbox{ for all }t \in (0,T) \, \\ \nonumber
\vec{\kappa}(t, 1) &= \zeta - \langle \zeta, \tau(t, 1) \rangle \tau(t, 1) \, , 
\end{align} 
and for some smooth initial data $f_{0}$. Notice that the second boundary condition gives that the curvature at the boundary is equal to the normal component of the vector $\zeta$.

Lemma \ref{endecr} and Lemma \ref{lemfv} show that equation \eqref{eqh} corresponds to the $L^{2}$-gradient flow for $W_{\lambda}$ and that the boundary conditions considered are natural in the usual sense of calculus of variation.

Aim of this paper is to show the results formulated in Theorem \ref{mainTh}.

%%%%%%%%%%%%%%%%%%%%%LEMMA FORMULAS %%%%%%%%%%%%%%%%%%%%%%%%%%%%%%%%%%%%%%%%%%%%%%%%%%%%%%%%%%%%%
\subsection{Geometrical Lemmas}
We start by studying the variation of some geometrical quantities considering smooth solutions $f:[0, T)\times \bar{I} \rightarrow \mathbb{R}^{n}$ of the more general flow 
\begin{equation*}
\partial_{t } f =\vec{V} + \varphi \tau
\end{equation*}
with $\vec{V}$ the normal velocity and $\varphi= \langle \partial_{t } f, \tau \rangle $ the tangential component of the velocity.
\begin{lemma}\label{lemform} 
Let $f:[0, T)\times \bar{I} \rightarrow \mathbb{R}^{n}$, $f=f(t,x)$, be a smooth solution of $\partial_{t} f = \vec{V} + \varphi \tau$ for $t \in (0, T)$, $x \in I$, and  with $\vec{V}$ the normal velocity. Given $\vec{\phi}$ any smooth normal field along $f$, the following formulas hold.
\begin{align}
\label{a}
\partial_{t}(ds)&=(\partial_{s} \varphi - \langle \vec{\kappa}, \vec{V} \rangle) ds \, , \\
\label{b}
\partial_{t} \partial_{s}- \partial_{s}\partial_{t} &= (\langle \vec{\kappa}, \vec{V} \rangle -\partial_{s} \varphi) \partial_{s} \, ,\\
\label{c}
\partial_{t} \tau &= \nabla_{s} \vec{V} + \varphi \vec{\kappa} \, ,\\
\label{d}
  \partial_{t} \vec{\phi}&= \nabla_{t} \vec{\phi}- \langle \nabla_s \vec{V} + \varphi \vec{\kappa}, \vec{\phi} \rangle \tau \, ,\\ 
\label{e0}
\partial_{t} \vec{\kappa} & = \partial_{s} \nabla_{s} \vec{V} + \langle \vec{\kappa}, \vec{V}\rangle  \vec{\kappa} + \varphi \partial_{s} \vec{\kappa} \, ,\\
\label{e}
\nabla_{t} \vec{\kappa}& = \nabla_{s}^{2} \vec{V} + \langle \vec{\kappa}, \vec{V}\rangle  \vec{\kappa} + \varphi \nabla_{s} \vec{\kappa} \, ,\\
\label{f}
(\nabla_{t}\nabla_{s}- \nabla_{s}\nabla_{t}) \vec{\phi} &=(\langle
\vec{\kappa},\vec{V} \rangle -\partial_{s }\varphi) \nabla_{s}\vec{\phi}
+[\langle \vec{\kappa}, \vec{\phi}\rangle \nabla_{s}\vec{V} -\langle
\nabla_{s} \vec{V}, \vec{\phi} \rangle \vec{\kappa}] 
\end{align}
and
\begin{align}
\label{g}
& (\nabla_{t}\nabla_{s}^2- \nabla_{s}^2\nabla_{t}) \vec{\phi} \\ \nonumber
& \quad =2 (\langle \vec{\kappa},\vec{V} \rangle -\partial_{s }\varphi) \nabla_{s}^2\vec{\phi} - (\partial^2_{s }\varphi) \nabla_{s}\vec{\phi}\\ \nonumber
&\qquad  +[\langle \nabla_s\vec{\kappa}, \vec{\phi}\rangle \nabla_{s}\vec{V} + \langle \vec{\kappa}, \vec{\phi}\rangle \nabla_{s}^2\vec{V} -\langle \nabla_{s}^2 \vec{V}, \vec{\phi} \rangle \vec{\kappa} -\langle \nabla_{s} \vec{V}, \vec{\phi} \rangle \nabla_s \vec{\kappa}] \\ \nonumber
&\qquad +[\langle \vec{\kappa}, \nabla_s \vec{V} \rangle \nabla_{s} \vec{\phi}+ \langle \nabla_s \vec{\kappa}, \vec{V} \rangle \nabla_{s} \vec{\phi} + 2 \langle \vec{\kappa}, \nabla_s \vec{\phi}\rangle \nabla_{s} \vec{V} - 2 \langle \nabla_{s} \vec{V}, \nabla_s \vec{\phi} \rangle \vec{\kappa} ] \, .
\end{align}
Furthermore, if $\varphi \equiv 0$, and if $\vec{\phi} = 0 =\vec{V}$ at the boundary, we have that at the boundary
\begin{align} \nonumber
\partial_{t} \partial_{s} & = \partial_{s}\partial_{t} \, ,\\
\label{f1}
\nabla_{t}\nabla_{s} \vec{\phi} & = \nabla_{s}\nabla_{t} \vec{\phi} \, ,\\
\label{g1}
\nabla_{t}\nabla_{s}^2 \vec{\phi} & =  \nabla_{s}^2\nabla_{t} \vec{\phi} \\
& \quad +[\langle \vec{\kappa}, \nabla_s \vec{V} \rangle \nabla_{s} \vec{\phi} + 2 \langle \vec{\kappa}, \nabla_s \vec{\phi}\rangle \nabla_{s} \vec{V} - 2 \langle \nabla_{s} \vec{V}, \nabla_s \vec{\phi} \rangle \vec{\kappa} ] .\nonumber
\end{align}
\end{lemma}

\begin{proof}
See \cite[Lemma 2.1]{DKS} or \cite[Lemma 1]{Lin} for formulas \eqref{a} to \eqref{f}. The last formula follows from \eqref{f} as follows
\allowdisplaybreaks{\begin{align*}
\nabla_{t}\nabla_{s}^2 \vec{\phi} & = (\nabla_s \nabla_t) \nabla_s \vec{\phi} + (\langle \vec{\kappa},\vec{V} \rangle -\partial_{s }\varphi) \nabla_{s}^2\vec{\phi} +[\langle \vec{\kappa}, \nabla_s \vec{\phi}\rangle \nabla_{s}\vec{V} -\langle \nabla_{s} \vec{V}, \nabla_s\vec{\phi} \rangle \vec{\kappa}]\\
& = \nabla_s [ \nabla_s \nabla_t \vec{\phi}+ (\langle \vec{\kappa},\vec{V} \rangle -\partial_{s }\varphi) \nabla_{s}\vec{\phi} +[\langle \vec{\kappa}, \vec{\phi}\rangle \nabla_{s}\vec{V} -\langle \nabla_{s} \vec{V}, \vec{\phi} \rangle \vec{\kappa}] ] \\
&\quad + (\langle \vec{\kappa},\vec{V} \rangle -\partial_{s }\varphi) \nabla_{s}^2\vec{\phi} +[\langle \vec{\kappa}, \nabla_s \vec{\phi}\rangle \nabla_{s}\vec{V} -\langle \nabla_{s} \vec{V}, \nabla_s\vec{\phi} \rangle \vec{\kappa}]\\
& =  \nabla_s^2 \nabla_t \vec{\phi}+ (\langle \vec{\kappa},\vec{V} \rangle -\partial_{s }\varphi) \nabla_{s}^2\vec{\phi}+ (\langle \nabla_s \vec{\kappa},\vec{V} \rangle + \langle  \vec{\kappa},\nabla_s \vec{V} \rangle - \partial^2_{s }\varphi) \nabla_{s}\vec{\phi} \\
&\quad  +[\langle \vec{\kappa}, \vec{\phi}\rangle \nabla_{s}^2\vec{V} -\langle \nabla_{s} \vec{V}, \vec{\phi} \rangle \nabla_s \vec{\kappa}]  \\
& \quad +[\langle \nabla_{s} \vec{\kappa}, \vec{\phi}\rangle \nabla_{s}\vec{V}  +\langle  \vec{\kappa}, \nabla_{s} \vec{\phi}\rangle \nabla_{s}\vec{V} -\langle \nabla^2_{s} \vec{V}, \vec{\phi} \rangle \vec{\kappa} -\langle \nabla_{s} \vec{V},  \nabla_{s} \vec{\phi} \rangle \vec{\kappa}]  \\
& \quad + (\langle \vec{\kappa},\vec{V} \rangle -\partial_{s }\varphi) \nabla_{s}^2\vec{\phi} +[\langle \vec{\kappa}, \nabla_s \vec{\phi}\rangle \nabla_{s}\vec{V} -\langle \nabla_{s} \vec{V}, \nabla_s\vec{\phi} \rangle \vec{\kappa}] \\
& =  \nabla_s^2 \nabla_t \vec{\phi}+ 2 (\langle \vec{\kappa},\vec{V} \rangle -\partial_{s }\varphi) \nabla_{s}^2\vec{\phi} - (\partial^2_{s }\varphi) \nabla_{s}\vec{\phi} \\
& \quad +[\langle \nabla_{s} \vec{\kappa}, \vec{\phi}\rangle \nabla_{s}\vec{V} +\langle \vec{\kappa}, \vec{\phi}\rangle \nabla_{s}^2\vec{V}  -\langle \nabla^2_{s} \vec{V}, \vec{\phi} \rangle \vec{\kappa}-\langle \nabla_{s} \vec{V}, \vec{\phi} \rangle \nabla_s \vec{\kappa}]  \\
& \quad +[ \langle  \vec{\kappa},\nabla_s \vec{V} \rangle \nabla_{s} \vec{\phi} +\langle \nabla_s \vec{\kappa},\vec{V} \rangle \nabla_{s} \vec{\phi} + 2\langle  \vec{\kappa}, \nabla_{s} \vec{\phi}\rangle \nabla_{s}\vec{V}  -2 \langle \nabla_{s} \vec{V},  \nabla_{s} \vec{\phi} \rangle \vec{\kappa}] .
\end{align*}}
\end{proof}

In the previous lemma it is made evident that $\nabla_s$ and $\nabla_t$ commutes at the boundary when certain quantities vanish. In the next lemma we see which terms are zero at the boundary when $f$ satisfies the boundary conditions~\eqref{bch}. 
\begin{lemma}\label{lembdybeh}
Under the assumption that $f$ solves $\partial_{t} f = \vec{V}$ on $(0,T) \times I$ with boundary conditions $f(t,0)=f_{-}$ and $f(t,1)=f_{+}$ for all $t$, we have that for $m \in \mathbb{N}_{0}$
\begin{align*}
\partial_{t} f= \nabla_{t} f  =0, \qquad \nabla_{t}^{m+1}f =0 \quad \mbox{ and } \quad 
\nabla_t^m \vec{V} =0 \quad \mbox{ for }x \in \{0,1\} \, .
\end{align*}
\end{lemma}
\begin{proof}
The statements follow directly from the assumptions. 
\end{proof}

%%%%%%%%%%%%%%%%%%%%%%%%%%%%%%CRUCIAL LEMMA %%%%%%%%%%%%%%%%%%%%%%%%%%%%%%%%%%%%%%%%%%%%%%%%

The general idea to prove Theorem \ref{mainTh} is to differentiate repeatedly Equation \eqref{eqh} and then, by integration, to derive estimates in appropriate Sobolev spaces. In this approach the following lemma is crucial.
\begin{lemma}\label{lempartint}
Suppose $\partial_{t}f =\vec{V}$ on $(0,T) \times I$. Let $\vec{\phi}$ be a normal vector field along $f$ and $Y=\nabla_{t} \vec{\phi} + \nabla_{s}^{4} \vec{\phi}$.
Then
\begin{align}\label{eqgen0}
\frac{d}{dt} \frac{1}{2}\int_{I} |\vec{\phi}|^{2} ds + \int_{I}
|\nabla_{s}^{2} \vec{\phi}|^{2 } ds & = - [ \langle \vec{\phi}, \nabla_s^3 \vec{\phi} \rangle ]_0^1+ [ \langle \nabla_s \vec{\phi}, \nabla_s^2 \vec{\phi} \rangle ]_0^1
\\ \nonumber
& \qquad +\int_{I} \langle Y, \vec{\phi} \rangle ds - \frac{1}{2} \int_{I} |\vec{\phi}|^{2} \langle \vec{\kappa}, \vec{V} \rangle ds ,
\end{align}
and if furthermore $\vec{\phi}=0$ on $\partial I$ then
\begin{equation}\label{eqgen}
\frac{d}{dt} \frac{1}{2}\int_{I} |\vec{\phi}|^{2} ds + \int_{I} |\nabla_{s}^{2} \vec{\phi}|^{2 } ds = [ \langle \nabla_s \vec{\phi}, \nabla_s^2 \vec{\phi} \rangle ]_0^1 +
\int_{I} \langle Y, \vec{\phi} \rangle ds - \frac{1}{2} \int_{I} |\vec{\phi}|^{2} \langle \vec{\kappa}, \vec{V} \rangle ds.
\end{equation}
\end{lemma}
\begin{proof}
See \cite[Lemma 2.2]{DKS}, \cite[Lemma 3]{Lin} for similar statements. The claim follows  using \eqref{a} and integration by parts.
\end{proof}

Typically the previous lemma is used to get an estimate for the $L^2$-norm of $\vec{\phi}$ squared using Gronwall's Lemma. To this end one first adds $\int_I |\vec{\phi}|^2 ds$ to both sides of the equation \eqref{eqgen0}/\eqref{eqgen}. Then it is necessary to show that
\begin{equation}\label{order}
\int_{I} |\nabla_{s}^{2} \vec{\phi}|^{2 } ds \,  
\end{equation}
together with the energy bound give us means to controll all  terms in the right-hand side of \eqref{eqgen0}/ \eqref{eqgen}. This is achieved by using  interpolation inequalities and the fact that for an appropriate choice of $\vec{\phi}$ the terms in the right-hand side are  in fact of lower order (see the discussion in the Introduction). 
Finally the obtained bounds yield the long-time existence result.

\begin{rem}\emph{As mentioned in the Introduction, in \cite{DKS} the authors consider closed curves and hence there are no boundary terms in \eqref{eqgen0}. In view of the parabolic equation \eqref{e} for the curvature vector, a good and natural choice for $\vec{\phi}$ is  $\nabla_s^{m} \vec{\kappa}$ for $m \in \mathbb{N}$. }
\end{rem}
\begin{rem}
\label{remLin}
\emph{In the case of curves with boundary one needs to take care of the boundary terms. 
In \cite{Lin} the author studies the evolution of \eqref{eqh} subject to the clamped boundary conditions 
\begin{equation}\label{haupt}
 f(t,0)= f_{-}, \, f(t,1)=f_{+}, \quad  \partial_s f(t,0)= \tau_{-}, \, \partial_s f(t,1)= \tau_{+} \, \mbox{ for all }t \, , 
\end{equation}
with $f_{-}, f_{+}, \tau_{-}, \tau_{+}$ fixed given vectors in $\mathbb{R}^n$.
Also in this case Lemma \ref{lembdybeh} holds yielding that $\nabla_t^{m}f=0$, $m \in \mathbb{N}$, at the boundary. Furthermore, the fact that the tangent vectors are given and fixed in time implies that also $$\nabla_s \nabla_t^{m}f =0  \qquad (m \in \mathbb{N})$$ at the boundary. 
Indeed, using \eqref{c}, the fact that the flow has no tangential component and the fact that $\partial_{t} \tau=0$ at the boundary we get that $\nabla_{s}\vec{V} =0$ and hence $\nabla_s \nabla_t f =0$.
Next from \eqref{f} one even infers that at the boundary 
\begin{align}
\label{f**}
\nabla_{s}\nabla_{t} \vec{\phi}=\nabla_{t} \nabla_{s} \vec{\phi}
\end{align}
for \emph{any} normal vector field $\vec{\phi}$. 
Thus we have
$0=\nabla_{t}\nabla_{s}\vec{V}=\nabla_{t}\nabla_{s}\nabla_{t}f = \nabla_{s} \nabla_{t}^{2} f$
at the boundary 
and applying \eqref{f**} repeatedly we obtain the claim.}

\emph{
The idea in \cite{Lin} is  to choose $\vec{\phi}=\nabla_t^m f$ since  both boundary terms in \eqref{eqgen0} disappear.}
\end{rem}

Following the idea of Lin in \cite{Lin} we take $\vec{\phi}=\nabla_t^m f$, $m \geq 1$, in Lemma \ref{lempartint}. Then $\vec{\phi}$ is zero at the boundary by \eqref{bch} and Lemma \ref{lembdybeh}. On the other hand in general none of the derivatives with respect to $s$ of $\vec{\phi}$ vanishes at the boundary. As a consequence, we have to work with Equation \eqref{eqgen}. The fact that the boundary term on the right-hand side of \eqref{eqgen} can also be controlled by \eqref{order} is a consequence of the boundary conditions \eqref{bch}. 
In the next section we present  computations that yield
this result. 

%%%%%%%%%%%%%%%%%%%%%%%%%%BOUNDARY TERMS%%%%%%%%%%%%%%%%%%%%%%%%%%%%%%%%%%%%%%%%%%%%%%%%%%%%%%%%%%%%%%

\subsubsection{Boundary term}

In this section we use the following notation 
\begin{equation}\label{psil}
\vec{\psi}^{m}:= \nabla_{t}^{m} f \quad \mbox{ for }\quad m \in \mathbb{N} \,  . 
\end{equation}
As already pointed out  we are going to take $\vec{\phi}=\vec{\psi}^{m}$ in \eqref{eqgen}. Therefore the boundary term reads
\begin{equation}\label{bterm}
[\langle \nabla_s \vec{\psi}^{m}, \nabla_s^2 \vec{\psi}^{m} \rangle ]_0^1\, .
\end{equation}
Due to the boundary condition \eqref{bch} relating the curvature to $\zeta$ and the tangent vector, we will be able to show that
\begin{equation*}
\nabla_s^2 \vec{\psi}^{m} = - \langle \zeta, \tau \rangle \nabla_s \vec{\psi}^{m} + \mbox{  lower order terms} \, ,
\end{equation*}
This observation is crucial to achieve  a control of  \eqref{bterm}  by \eqref{order}.

\smallskip
\underline{Notation for $\vec{R}^{m}_{n}$ and $\vec{S}^{m}_{n}$: } 
it is convenient to introduce two new vector fields.
\begin{itemize}
\item[-]  For $n$ odd: $\vec{R}^{m}_{n}$ denotes any linear combination of terms of the form
\begin{equation*}
\langle \nabla_{s} \vec{\psi}^{i_1}, \nabla_{s} \vec{\psi}^{i_2} \rangle \dots  \langle \nabla_{s} \vec{\psi}^{i_{n-2}}, \nabla_{s} \vec{\psi}^{i_{n-1}} \rangle \nabla_{s} \vec{\psi}^{i_{n}}
\end{equation*}
with $i_1+ \dots +i_{n}=m$, $i_{j} \geq 1$, and coefficients bounded by some universal constants.\newline
\item[-] For $n$ even: $\vec{S}^{m}_{n}$ denotes any linear combination of terms of the form
\begin{equation*}
\langle \zeta, \nabla_{s} \vec{\psi}^{i_1} \rangle \dots  \langle \nabla_{s} \vec{\psi}^{i_{n-2}}, \nabla_{s} \vec{\psi}^{i_{n-1}} \rangle \nabla_{s} \vec{\psi}^{i_{n}}
\end{equation*}
with $i_1+ \dots +i_{n}=m$, $i_{j} \geq 1$, and coefficients bounded by some universal constants. Here  $\zeta$ is as in \eqref{bch}.
\end{itemize}
We start by collecting some relations.

\begin{lemma}\label{lemsteps}
Suppose $\partial_{t}f =\vec{V}$ on $(0,T) \times I$. 
Then for any $m,n,i \in \mathbb{N}$ and $t \in (0,T)$ we have that at the boundary $\vec{\psi}^{i}=0$ and 
\begin{enumerate}
\item[i.] $\partial_{t} \langle \zeta, \tau \rangle = \langle \zeta, \nabla_{s} \vec{\psi}^{1} \rangle$ ;
\item[ii.] $\nabla_{t} [\langle \zeta, \tau \rangle \tau ] =\langle \zeta, \tau \rangle \nabla_{s} \vec{\psi}^{1} $;
\item[iii.] $\partial_{t} \langle \zeta, \nabla_{s} \vec{\psi}^{i} \rangle = \langle \zeta, \nabla_{s} \vec{\psi}^{i+1} \rangle - \langle \zeta, \tau \rangle \langle  \nabla_{s} \vec{\psi}^{1} , \nabla_{s} \vec{\psi}^{i} \rangle$;
\item[iv.]  for $n$ odd, $\nabla_{t} \vec{R}^{m}_{n} = \vec{R}^{m+1}_{n}$;
\item[v.] for $n$ even, $\nabla_{t} \vec{S}^{m}_{n}= \vec{S}^{m+1}_{n} + \langle \zeta, \tau \rangle \vec{R}^{m+1}_{n+1}$;
\item[vi.] for $n$ odd, $\nabla_{t} [\langle \zeta, \tau \rangle \vec{R}_{n}^{m}]= \langle \zeta, \tau \rangle \vec{R}_{n}^{m+1}+ \vec{S}_{n+1}^{m+1}$.
\end{enumerate}
\end{lemma}
\begin{proof}
By Lemma \ref{lembdybeh}, we know that $\vec{\psi}^i=0$ at the boundary for all $i \in \mathbb{N}$. This in particular implies $\vec{V}=0$ at the boundary. First of all recall that for a vector field $\phi$ and scalar function $f$ we have that
\begin{align*}
\nabla_{t}(f \phi)= \partial_{t} f  (\phi-\langle \phi, \tau \rangle \tau) + f \nabla_{t} \phi.
\end{align*}
Equations \textit{i.} and \textit{ii.} follow from \eqref{c}, the formula above and the equalities $\vec{V}=\partial_{t} f = \vec{\psi}^{1}$. Similarly, using \eqref{f1} and \eqref{c}
\begin{align*}
\partial_{t} \langle \zeta, \nabla_{s} \vec{\psi}^{i} \rangle & = \langle \zeta,  \nabla_{t} \nabla_{s} \vec{\psi}^{i} \rangle + \langle \zeta, \tau \rangle \langle \tau, \partial_{t} \nabla_{s} \vec{\psi}^{i} \rangle \\
 & = \langle \zeta, \nabla_{s} \vec{\psi}^{i+1} \rangle - \langle \zeta, \tau \rangle \langle \nabla_s \vec{\psi}^{1}, \nabla_{s} \vec{\psi}^{i} \rangle
\end{align*}
that is \textit{iii.}. The other claims follow similarly.
\end{proof}

\begin{lemma}[The boundary term]\label{lembterms}
Suppose $\partial_{t}f =\vec{V}$ on $(0,T) \times I$. Then for any $m \in \mathbb{N}$ and $t \in (0,T)$ we have that at the boundary $\vec{\psi}^{m}=0$ and 
\begin{align} \label{cl}
\nabla_{s}^2 \vec{\psi}^{m} & = -\langle \zeta, \tau \rangle \nabla_s \vec{\psi}^{m} + (\zeta -\langle \zeta, \tau \rangle \tau)  \sum_{\substack{i+j=m\\i,j \geq 1}} c^{m}_{i,j}\langle \nabla_{s} \vec{\psi}^{i}, \nabla_{s} \vec{\psi}^{j} \rangle \\ \nonumber
& \quad  + \langle \zeta, \tau \rangle  \sum_{\substack{3 \leq n \leq m\\n \mbox{ \tiny odd}}} \vec{R}_{n}^{m} + \sum_{\substack{2 \leq n \leq m\\n \mbox{ \tiny even}}} \vec{S}_{n}^{m} \, ,
\end{align}
with $c^{m}_{i,j}$ absolute constants.
\end{lemma}
\begin{proof}
By Lemma \ref{lembdybeh}, we know that $\vec{\psi}^m=0$ at the boundary for all $m \in \mathbb{N}$. This in particular implies $\vec{V}=0$ at the boundary. In this case and at the boundary, we may write \eqref{g1} with the notation just introduced as follows
\begin{align}\nonumber
\nabla_{s}^2 \vec{\psi}^{m+1}  &  = \nabla_{t}\nabla_{s}^2 \vec{\psi}^{m} - \langle \zeta, \nabla_s \vec{\psi}^{1} \rangle \nabla_{s} \vec{\psi}^{m}\\ \nonumber
& \quad  -  2 \langle \zeta, \nabla_s \vec{\psi}^{m}\rangle \nabla_{s} \vec{\psi}^{1} + 2 \langle \nabla_{s} \vec{\psi}^{1}, \nabla_s \vec{\psi}^{m} \rangle (\zeta-\langle \zeta, \tau \rangle \tau) \\ \label{compm}
& =  \nabla_{t}\nabla_{s}^2 \vec{\psi}^{m} + \vec{S}^{m+1}_{2} - 2   (\zeta-\langle \zeta, \tau \rangle \tau)  \langle \nabla_{s} \vec{\psi}^{1}, \nabla_s \vec{\psi}^{m} \rangle .
\end{align}
Here we have used the boundary conditions \eqref{bch} and the fact that $\vec{V}=\partial_{t} f = \vec{\psi}^{1}$.

We prove \eqref{cl} by induction. Since $\vec{V}=0$ at the boundary, for $m=1$ we have with $\vec{V}= \vec{\psi}^{1}$, \eqref{e}, \eqref{bch} and \textit{ii.} in Lemma \ref{lemsteps}
\begin{align*}
\nabla_s^2 \vec{\psi}^1 & = \nabla_s^2 \vec{V}= \nabla_{t} \vec{\kappa} = \nabla_{t} \left(\zeta -\langle \zeta, \tau \rangle \tau \right) =    -\langle \zeta, \tau \rangle \nabla_{s} \vec{\psi}^1 \, ,
\end{align*}
that is \eqref{cl} in the special case $m=1$. 

Assuming that \eqref{cl} is valid for $m \geq 1$, we find using $\vec{V}=0$ at the boundary, \eqref{compm}, \eqref{f1} and Lemma \ref{lemsteps}
\begin{align*}
& \nabla_{s}^2 \vec{\psi}^{m+1} \\
& =  \nabla_{t} [ -\langle \zeta, \tau \rangle \nabla_s \vec{\psi}^{m} + (\zeta -\langle \zeta, \tau \rangle \tau)  \sum_{\substack{i+j=m\\i,j \geq 1}} c^{m}_{i,j}\langle \nabla_{s} \vec{\psi}^{i}, \nabla_{s} \vec{\psi}^{j} \rangle \\
&\quad  + \langle \zeta, \tau \rangle  \sum_{\substack{3 \leq n \leq m\\n \mbox{ \tiny odd}}} \vec{R}_{n}^{m} + \sum_{\substack{2 \leq n \leq m\\n \mbox{ \tiny even}}} \vec{S}_{n}^{m} ] + \vec{S}^{m+1}_{2} - 2   (\zeta-\langle \zeta, \tau \rangle \tau)  \langle \nabla_{s} \vec{\psi}^{1}, \nabla_s \vec{\psi}^{m} \rangle \\
& = -\langle \zeta, \tau \rangle \nabla_s \vec{\psi}^{m+1} -\langle \zeta, \nabla_s \vec{\psi}^{1} \rangle \nabla_s \vec{\psi}^{m} + (\zeta -\langle \zeta, \tau \rangle \tau)  \sum_{\substack{i+j=m\\i,j \geq 1}} c^{m}_{i,j} \partial_{t} \langle \nabla_{s} \vec{\psi}^{i}, \nabla_{s} \vec{\psi}^{j} \rangle \\
&\quad -\langle \zeta, \tau \rangle \nabla_s \vec{\psi}^{1}
\sum_{\substack{i+j=m\\i,j \geq 1}} c^{m}_{i,j}\langle \nabla_{s}
\vec{\psi}^{i}, \nabla_{s} \vec{\psi}^{j} \rangle  + \langle \zeta, \tau
\rangle  \sum_{\substack{3 \leq n \leq m\\n \mbox{ \tiny odd}}}
\vec{R}_{n}^{m+1} \\
& \quad  + \langle \zeta, \nabla_s \vec{\psi}^{1} \rangle  \sum_{\substack{3 \leq n \leq m\\n \mbox{ \tiny odd}}} \vec{R}_{n}^{m}+ \sum_{\substack{2 \leq n \leq m\\n \mbox{ \tiny even}}} \vec{S}_{n}^{m+1}  + \langle \zeta, \tau \rangle \sum_{\substack{2 \leq n \leq m\\n \mbox{ \tiny even}}} \vec{R}_{n+1}^{m+1} \\
& \quad  + \vec{S}^{m+1}_{2} - 2   (\zeta-\langle \zeta, \tau \rangle \tau)  \langle \nabla_{s} \vec{\psi}^{1}, \nabla_s \vec{\psi}^{m} \rangle\\
& = -\langle \zeta, \tau \rangle \nabla_s \vec{\psi}^{m+1}  + (\zeta -\langle \zeta, \tau \rangle \tau)  \sum_{\substack{i+j=m+1\\i,j \geq 1}} c^{m+1}_{i,j} \langle \nabla_{s} \vec{\psi}^{i}, \nabla_{s} \vec{\psi}^{j} \rangle \\
& \quad  + \sum_{\substack{2 \leq n \leq m+1\\n \mbox{ \tiny even}}} \vec{S}_{n}^{m+1}  + \langle \zeta, \tau \rangle \sum_{\substack{3 \leq n \leq m+1\\n \mbox{ \tiny odd}}} \vec{R}_{n}^{m+1} \, .
\end{align*}
\end{proof}

%%%%%%%%%%%%%%%%%%%%%%%%%%%%%%%%%%%%%%%%%%%%%%%%%%%%%%%%%%%%%
%%%%%%%%%%%%%%%% SEZIONE 3 %%%%%%%%%%%%%%%%%%%%%%%%%%%%%%%%%%
%%%%%%%%%%%%%%%%%%%%%%%%%%%%%%%%%%%%%%%%%%%%%%%%%%%%%%%%%%%%%

\section{A technical Lemma}
\label{Sez3}

In this section we derive the equation satisfied by $\nabla_t^{m} f$, $m \in \mathbb{N}$, so that we will be able to infer that  the term
$$\nabla_t (\nabla_t^{m} f)+ \nabla_s^4 \nabla_t^{m} f = Y \, , $$
in Lemma \ref{lempartint} (with $\vec{\phi}=\nabla^{m}_{t}f$) contains lower order terms than $\nabla_t(\nabla_t^m f)$. 

The results presented in the following Lemma \ref{lemLin} can essentially be found in \cite[Lemma 8]{Lin}. However we present here full proofs 
 for sake of completeness and also because we use a different notation that provides  more information than the one used in \cite{Lin}. This extra information is also crucial for the clarity and transparency  of some steps in the final proof of long-time existence.
 
The equation satisfied by $\nabla_t^{m} f$  can be derived by repeatedly differentiating equation \eqref{eqh} and by interchanging the operators $\nabla_s$ and $\nabla_t$. This generates extremely many terms (recall \eqref{f}). Similarly to \cite{Lin}, our strategy in the representation of the equations is  to single out the most singular term and to introduce a notation that takes care of all remaining ones. 
In addition it should be 
immediately clear: the number of derivatives present, the number of factors present and the maximal number of derivatives falling on one factor.

As in \cite{DKS}, for normal vector fields $\vec{\phi}_1, \dots, \vec{\phi}_k$, the product $\vec{\phi}_1 * \dots * \vec{\phi}_k$ defines for even $k$ a function  given by
$$\langle \vec{\phi}_1,  \vec{\phi}_2  \rangle \dots \langle \vec{\phi}_{k-1}, \vec{\phi}_{k} \rangle \, ,$$
while for $k $ odd it defines a normal vector field 
$$\langle \vec{\phi}_1,  \vec{\phi}_2  \rangle \dots \langle \vec{\phi}_{k-2}, \vec{\phi}_{k-1} \rangle \vec{\phi}_{k} \, .$$

For $\vec{\phi}$ a normal vector field, $P_{b}^{a,c}(\vec{\phi})$ denotes any linear combination of terms of  type 
$$\nabla_{s}^{i_1} \vec{\phi} * \dots * \nabla_{s}^{i_{b}} \vec{\phi} \mbox{ with }i_1 + \dots +i_{b}= a \mbox{ and } \max i_{j} \leq c \, ,$$
with coefficients bounded by some universal constant. Notice that $a$ gives the total number of derivatives, $b$ gives the number of factors and $c$ gives the highest number of derivatives falling on one factor. Comparing our notation with the one in \cite{Lin} one notices that we have added the parameter $c$. Furthermore,  for  sums over $a$, $b$ and $c$ we set
\begin{equation*}
\sum_{\substack{[[a,b]] \leq [[A,B]]\\c\leq C}} P^{a,c}_{b} (\vec{\phi}) : = \sum_{a=0}^{A} \sum_{b=1}^{2A+B-2a} \sum_{c=0}^{C}  P^{a,c}_{b}(\vec{\phi}) \, .
\end{equation*}
(The range of the $b$'s will also be often specified at the bottom of the  sum sign.) 

It is important to understand the relation between $a$ and $b$ in the sum: the more derivatives we take the less factors are present. In the other direction: if we take one derivative less we may allow for two factors more. This relation has its origin in the equation that $f$ satisfies. Indeed \eqref{eqh} may be written as
\begin{eqnarray*}
\partial_t f = - \nabla_s^2 \vec{\kappa}- \frac12 |\vec{\kappa}|^2 \vec{\kappa} + \lambda \vec{\kappa} = \sum_{\substack{[[a,b]] \leq [[2,1]]\\c\leq 2}} P^{a,c}_{b} (\vec{\kappa}) \, +\lambda P^{0,0}_{1} (\vec{\kappa}).
\end{eqnarray*}
This structure is maintained in the equations obtained by differentiation. Moreover it is important to keep track of this relation for the application of the interpolation inequalities. In particular notice that for all the terms in the sum, one has
\begin{equation}
\label{rel-ab}
a+ \frac12 b \leq a + \frac12 (2A +B -2 a) = A + \frac12 B . 
\end{equation}

In the following lemma we collect the formulas needed. 
\begin{lemma}\label{lemLin}
Suppose $f: [0,T) \times \bar{I} \rightarrow \mathbb{R}^n$ is a smooth regular solution to
\begin{equation*}
\partial_{t} f= -\nabla_s^2 \vec{\kappa} -\frac12 |\vec{\kappa}|^2 \vec{\kappa} + \lambda \vec{\kappa} = \vec{V} \, 
\end{equation*}
in $(0,T) \times I$. Then, the following formulas hold on $(0,T) \times I$. 
\begin{enumerate}
\item  For any $l \in \mathbb{N}_{0}$ and $k \in \mathbb{N}$
\begin{align}\label{E1}
\left[ \nabla_{t} \nabla_{s}^{k}  -\nabla_{s}^{k} \nabla_{t} \right] \nabla^{l}_{s}\vec{\kappa} & = \sum_{\substack{[[a,b]] \leq [[l+k+2,3]]\\c\leq \max\{l,2\}+k\\b \in [3,5], odd}} P^{a,c}_{b} (\vec{\kappa})  + \lambda
\sum_{\substack{[[a,b]] \leq [[l+k,3]]\\ c \leq l+k\\b=3}} P^{a,c}_{b} (\vec{\kappa}) \, .
\end{align}
\item For any $m, \nu \in \mathbb{N}$, $\nu$ odd, and $\mu, d \in \mathbb{N}_0$
\begin{align}\label{E2}
\nabla_{t}^{m} P^{\mu,d}_{\nu} (\vec{\kappa}) = \sum_{i=0}^{m} \lambda^{i} \sum_{\substack{[[a,b]] \leq [[4m + \mu -2 i,\nu]]\\c\leq 4 m -2i + d\\b \in [\nu,\nu+4m-2i], odd}} P^{a,c}_{b} (\vec{\kappa})  \, .
\end{align}
\item For any $A, C \in \mathbb{N}_0$, $B, N, M \in \mathbb{N}$, $B$ odd,
\begin{align}\label{E2sum}
\nabla_{t} \sum_{\substack{[[a,b]]\leq [[A,B]]\\c\leq C\\b\in [N,M], odd}}P^{a,c}_{b} (\vec{\kappa}) = \sum_{\substack{[[a,b]] \leq [[A +4,B]]\\c\leq C+4\\b \in [N,M+4], odd}} P^{a,c}_{b} (\vec{\kappa}) + \lambda \sum_{\substack{[[a,b]] \leq [[A +2,B]]\\c\leq C+2\\b \in [N,M+2], odd}} P^{a,c}_{b} (\vec{\kappa})  \, .
\end{align}
\item For any $m \in \mathbb{N}$
\begin{align}\label{E3}
& \nabla_{t}^{m} \vec{\kappa}- (-1)^m \nabla_s^{4m} \vec{\kappa} \\ \nonumber 
& \quad = \sum_{\substack{[[a,b]] \leq [[4m -2,3]]\\c\leq 4 m -2\\b\in[3,4m+1],odd}}
P^{a,c}_{b} (\vec{\kappa}) + \sum_{i=1}^{m} \lambda^{i} \sum_{\substack{[[a,b]] \leq [[4m -2i,1]]\\c\leq 4 m-2i \\b\in[1,4m+1-2i],odd}} P^{a,c}_{b} (\vec{\kappa}) \, .
\end{align}
\item For any $m,k \in \mathbb{N}$ and $l \in \mathbb{N}_0$
\begin{align}\label{E4}
& [\nabla_{t}^{m} \nabla_s^k -  \nabla_s^k \nabla_{t}^{m}] \nabla_{s}^{l}
\vec{\kappa} \\ \nonumber 
& \quad =  \sum_{\substack{[[a,b]] \leq [[4m+k+l -2,3]]\\c\leq 4 m
    +l+k-2\\b\in[3,4m+1],odd}} P^{a,c}_{b} (\vec{\kappa})  + \sum_{i=1}^{m} \lambda^{i} \sum_{\substack{[[a,b]] \leq [[4m+k+l-2i,1]]\\c\leq 4 m +l+k-2i\\b\in[1,4m-2i+1],odd}} P^{a,c}_{b} (\vec{\kappa}) \, .
\end{align}
\item For any $m \in \mathbb{N}$
\begin{align}\label{E5}
& \nabla_{t}^{m} f -  (-1)^{m} \nabla_s^{4m-2} \vec{\kappa}  \\ \nonumber
& \quad = \sum_{\substack{[[a,b]] \leq [[4m -4,3]]\\c\leq 4 m -4\\b\in[3, 4m-1],odd}}
P^{a,c}_{b} (\vec{\kappa})
+ \quad \sum_{i=1}^{m} \lambda^{i} \sum_{\substack{[[a,b]] \leq [[4m -2-2i,1]]\\c\leq 4 m -2-2i\\b\in[1, 4m-1-2i],odd}} P^{a,c}_{b} (\vec{\kappa}) \, .
\end{align}
\end{enumerate}
\end{lemma}
\begin{proof}
See Appendix \ref{AppLin}.
\end{proof}

In the previous lemma we have chosen to express explicitely the dependence in $\lambda$ in the equations. This was not done in \cite[Lemma 8]{Lin} and could be useful for studying the flow  with a fixed length constraint.

%%%%%%%%%%%%%%%%%%%%%%%%%%%%%%%%%%%%%Boundary Terms - Section Lin %%%%%%%%%%%%%%%%%%%%%%%%%%%%%%%%%%%%%%%%%%%%%%

\subsection{Estimates for some boundary terms}
\label{Sez3.1}

It is convenient here to collect some estimates on some boundary terms. In the following, $|P^{a,c}_b(\vec{\kappa})|$ denotes any linear combination with non-negative coefficients of terms of  type 
$$|\nabla_{s}^{i_1} \vec{\phi}|\cdot |\nabla_{s}^{i_2} \vec{\phi}| \cdot ... \cdot | \nabla_{s}^{i_{b}} \vec{\phi}| \mbox{ with }i_1 + \dots +i_{b}= a \mbox{ and } \max i_{j} \leq c \, .$$

\begin{lemma}
\label{hilfslemmabdry}
For $m \geq 1 $ we have that at the boundary ($x \in \partial I$)
\begin{align}\label{epsbdyht0}
|\nabla_s \nabla_t^m f|^2 (x) \leq \int_I \sum_{i=0}^{2m} \lambda^{i} \sum_{\substack{[[a,b]] \leq [[8m-1-2i,2]]\\c\leq 4m\\b\in[2,8m+1-2i]}} |P^{a,c}_{b} (\vec{\kappa})| ds,
\end{align}
and also for all $\epsilon>0$ 
\begin{align}\label{epsbdyht}
|\nabla_s \nabla_t^m f|^2 (x) & \leq \epsilon \int_I |\nabla_s^2 \nabla_t^{m}
f|^2 ds \\ \nonumber
& \quad + C(\epsilon) \int_I \sum_{\substack{i=0\\i \ even}}^{2m} \lambda^{i} \sum_{\substack{[[a,b]] \leq [[8m-2-2i,4]]\\c\leq 4m-1-i\\b\in[2,8m+2-2i]}} |P^{a,c}_{b} (\vec{\kappa})| ds.
\end{align}
\end{lemma}
\begin{proof}
Since $\nabla_{t}^{m}f=0$ at the boundary, for each space component $[\nabla_{t}^{m}f]^{i}$ there exists some $x_i^{*} \in I$ such that $\partial_s [\nabla_{t}^{m}f]^{i} (x_i^{*})=0$. Hence
\begin{align}
|\nabla_s \nabla_{t}^{m}f|^2 (x) &  \leq |\partial_s \nabla_{t}^{m}f|^2 (x)= \sum_{i=1}^{n} [(\partial_s [\nabla_{t}^{m}f]^{i})^2 (x) - (\partial_s [\nabla_{t}^{m}f]^{i})^2 (x_i^{*}) ] \nonumber \\
&\leq 2 \int_{I} |\partial_s  \nabla_{t}^{m}f| | \partial_s^2  \nabla_{t}^{m}f| \, ds \, . \label{smily1}
\end{align}
Using
\begin{align}
\label{hl1}
\nabla^{j}_s \nabla_t^m f=\sum_{i=0}^{m} \lambda^{i} \sum_{\substack{[[a,b]] \leq [[4m-2+j-2i,1]]\\c\leq 4m-2+j-2i\\b\in[1,4m-1-2i], \ odd}} P^{a,c}_{b} (\vec{\kappa}) \quad \mbox{  for }j=0,1,2 \, ,
\end{align}
and the fact that for a normal vector field $g$ its full derivatives can be written as $\partial_s g  = \nabla_s g -\langle g, \vec{\kappa}\rangle \tau $, and $
\partial_s^2 g  = \nabla_s^2 g -2 \langle \nabla_s g, \vec{\kappa}\rangle \tau -\langle  g, \nabla_s \vec{\kappa}\rangle \tau-\langle g,  \vec{\kappa}\rangle \vec{\kappa} $, it follows that
\begin{align}\label{smily2}
|\partial_s \nabla_t^m f|  
 & \leq \sum_{i=0}^{m} \lambda^{i} \sum_{\substack{[[a,b]] \leq [[4m-1-2i,1]]\\c\leq 4m-1-2i\\b\in[1,4m-2i]}} |P^{a,c}_{b} (\vec{\kappa})| \, ,
\end{align}
and
\begin{align}\label{smily3}
|\partial_s^2 \nabla_t^m f| & \leq   |\nabla_s^2 \nabla_t^m f|
  + \sum_{i=0}^{m} \lambda^{i} \sum_{\substack{[[a,b]] \leq [[4m-1-2i,2]]\\c\leq 4m-1-2i\\b\in[2,4m+1-2i] }} | P^{a,c}_{b} (\vec{\kappa})| \\ \nonumber
& \leq  \sum_{i=0}^{m} \lambda^{i} \sum_{\substack{[[a,b]] \leq [[4m-2i,1]]\\c\leq 4m-2i\\b\in[1,4m+1-2i]}} |P^{a,c}_{b} (\vec{\kappa}) |\, .
\end{align}

The first claim follows directly from \eqref{smily1}, \eqref{smily2} and \eqref{smily3}. 
Finally, for any $\epsilon \in (0,1)$ 
\begin{align*}
|\nabla_s \nabla_t^m f |^2 (x) & \leq \frac{\epsilon}{2} \int_{I} |\partial_s^2 \nabla_t^m f|^2 + C(\epsilon) \int_{I} |\partial_s \nabla_t^{m} f|^2. 
\end{align*}
Using the previous estimates and the fact that $(\sum_{i=1}^{n} a_i)^2 \leq n \sum_{i=1}^n a_i^2$ we infer the second claim.
\end{proof}

\begin{lemma}\label{hilfslemmabdry2}
 We have that for $x \in \partial I$
\begin{align*}
|\nabla_{s}\nabla_{t} f|^{4}(x) \leq \int_{I}
\sum_{i=0}^{4} \lambda^{i} \sum_{\substack{[[a,b]] \leq [[13-2i,4]]\\c\leq 4\\b\in[4,17-2i]}} |P^{a,c}_{b} (\vec{\kappa})|.
\end{align*}
\end{lemma}
\begin{proof}
Proceeding as in the proof of the previous lemma we find
\begin{align*}
|\nabla_{s}\nabla_{t} f|^{4}(x) \leq |\partial_s \nabla_{t}f|^4 (x) \leq  4n \int_{I} |\partial_s \nabla_{t}f|^3 | \partial_s^2 \nabla_{t}f| \, ds \, . 
\end{align*}
Using \eqref{smily2}, \eqref{smily3} (with $m=1$) 
and 
\begin{align*}
|\partial_s \nabla_t f|^3 & \leq \sum_{i=0}^{3} \lambda^{i} \sum_{\substack{[[a,b]] \leq [[9-2i,3]]\\c\leq 3\\b\in[3,12-2i]}} |P^{a,c}_{b} (\vec{\kappa})| \, ,
\end{align*}
the claim follows.
\end{proof}

%%%%%%%%%%%%%%%%%%%%%%%%%%%%%%%%%%%%%%%%%%%%%%%%%%%%%%%%%%%%%%%%%%%%
%%%%%%%%%%%% SEZIONE 4 %%%%%%%%%%%%%%%%%%%%%%%%%%%%%%%%%%%%%%%%%%%%%
%%%%%%%%%%%%%%%%%%%%%%%%%%%%%%%%%%%%%%%%%%%%%%%%%%%%%%%%%%%%%%%%%%%%

\section{Interpolation inequalities}
\label{Sez4}

The main result in this section is the following inequality (see Lemma \ref{lemineq} for more details): one has that for any $\epsilon \in (0,1)$
\begin{equation*}
\sum_{\substack{[[a,b]]\leq[[A,B]]\\c\leq k-1\\ b \in [2, M]}} \int_I |P^{a,c}_{b}(\vec{\kappa})| \, ds  \leq \epsilon \int_I |\nabla_s^{k} \vec{\kappa}|^2 ds +C(\epsilon, \mathcal{L}[f], \| \vec{\kappa}\|_{L^2}, A,B,k,n,M) \, ,
\end{equation*}
whenever $A + \frac12 B < 2k+1$. This is the key ingredient to control the terms in the right-hand side of \eqref{eqgen} in Lemma \ref{lempartint}. 

The inequality stated above follows from suitable interpolation inequalities for which 
it is useful to introduce the following norms
\begin{equation*}
\| \vec{\kappa}\|_{k,p} := \sum_{i=0}^{k} \| \nabla_s^{i} \vec{\kappa}\|_{p} \quad \mbox{ with } \quad \| \nabla_s^{i} \vec{\kappa} \|_{p} := \mathcal{L}[f]^{i+1-1/p} \Big( \int_I |\nabla_s^{i} \vec{\kappa}|^p ds \Big)^{1/p} \, ,
\end{equation*}
as opposed to
\begin{equation*}
\| \nabla_s^{i} \vec{\kappa}\|_{L^{p}} := \Big( \int_I |\nabla_s^{i} \vec{\kappa}|^p ds \Big)^{1/p} \, .
\end{equation*}
These norms  are motivated by suitable  scaling properties (see Appendix \ref{AppendixInter}). 

The following Lemma \ref{leminter}, Lemma \ref{lemineq}, and Lemma \ref{Qlemma} are adaptations to the present setting and notation of those used in \cite{DKS} and \cite{Lin}. We choose to state the results in details for sake of completeness. 
Moreover we indicate the precise dependence of the appearing constants. 

\begin{lemma}\label{leminter}
Let $f: I \rightarrow \mathbb{R}^n$ be a smooth regular curve. Then for all $k \in \mathbb{N}$, $p \geq 2$ and $0 \leq i<k$ we have
\begin{equation*}
\| \nabla_s^{i} \vec{\kappa}\|_{p} \leq C \|\vec{\kappa}\|_{2}^{1-\alpha} \|\vec{\kappa}\|_{k,2}^{\alpha} \, ,  
\end{equation*}
with $\alpha= (i+\frac12 -\frac{1}{p})/k$ and $C=C(n,k,p)$.
\end{lemma}
\begin{proof}
A proof of this fact is hinted at in \cite[Lemma 2.4]{DKS} and \cite[Lemma~5]{Lin}. We give all details in Appendix \ref{AppendixInter}.
\end{proof}

\begin{cor}\label{corinter}
Let $f: I \rightarrow \mathbb{R}^n$ be a smooth regular curve. Then for all $k \in \mathbb{N}$ we have
\begin{equation*}
\|\vec{\kappa}\|_{k,2} \leq C ( \|\nabla^{k}_{s}\vec{\kappa}\|_{2} + \|\vec{\kappa}\|_{2}) \, ,
\end{equation*}
with $C=C(n,k)$.
\end{cor}
\begin{proof}
We proceed by induction on $k$. The claim for $k=1$ follows directly from the definition of the norm. Let us assume the claim is true up to some $k \geq 1$. Then by Lemma \ref{leminter}
\begin{align*}
 & \| \vec{\kappa}\|_{k+1,2}  = \| \vec{\kappa}\|_{k,2} + \| \nabla_s^{k+1} \vec{\kappa}\|_{2} \leq C (\|\nabla^{k}_{s}\vec{\kappa}\|_{2} + \|\vec{\kappa}\|_{2}) + \| \nabla_s^{k+1} \vec{\kappa}\|_{2} \\
&  \leq C \|\vec{\kappa}\|_{k+1,2}^{\frac{k}{k+1}}
\|\vec{\kappa}\|_{2}^{1-\frac{k}{k+1}}\!\! + \! C \|\vec{\kappa}\|_{2} + \! \|
\nabla_s^{k+1} \vec{\kappa}\|_{2}  \leq \frac{1}{2} \|\vec{\kappa}\|_{k+1,2} + \!C \|\vec{\kappa}\|_{2} + \! \| \nabla_s^{k+1} \vec{\kappa}\|_{2}  
\end{align*}
from which the claim follows.
\end{proof}

\begin{lemma}\label{lemineq}
For any $a,c \in \mathbb{N}_0$, $b \in \mathbb{N}$, $b\geq 2$, $c \leq k-1$ we find
\begin{equation*}
\int_I |P^{a,c}_{b} (\vec{\kappa})| \, ds \leq C \mathcal{L}[f]^{1-a-b} \|\vec{\kappa}\|_{2}^{b-\gamma} \| \vec{\kappa}\|_{k,2}^{\gamma} \, ,
\end{equation*}
with $\gamma=(a+\frac12 b-1)/k$ and $C=C(n,k, b)$. Further if $A,B,M\in \mathbb{N}$, $M \geq 2$ with $A+\frac12 B<2k+1$, then for any $\epsilon \in (0,1)$
\begin{align*}
\sum_{\substack{[[a,b]]\leq[[A,B]]\\c\leq k-1\\b \in [2,M]}} \int_I |P^{a,c}_{b}(\vec{\kappa})| & \leq \epsilon \int_I |\nabla_s^{k} \vec{\kappa}|^2 ds +C \epsilon^{-\frac{\overline{\gamma}}{2-\overline{\gamma}}} \max\{1, \|\vec{\kappa}\|^2_{L^{2}}\}^{\frac{M-\overline{\gamma}}{2-\overline{\gamma}}}\\
& \quad  + C \min\{1, \mathcal{L}[f]\}^{1-A-\frac{B}{2}} \max\{1, \|\vec{\kappa}\|_{L^{2}}\}^{M}  + C \| \vec{\kappa} \|_{L^{2}}^{2}\, ,
\end{align*}
with $\overline{\gamma}= (A + \frac12 B-1)/k$ and $C=C(n,k,A,B)$.
\end{lemma}
It is interesting to note that the right-hand side of the second inequality depends only on the \emph{lower bound} of the length of the curve.
\begin{proof}
First of all note that $\gamma=0$ if and only if $a=0$ and $b=2$. In this case the first claim follows immediately using the definition of the norm.
Next let $0<\gamma$.
Each of the terms in $|P^{a,c}_{b}(\vec{\kappa})|$ is of the form $|\nabla_s^{i_1} \vec{\kappa} |\cdot \ldots \cdot| \nabla_s^{i_{b}}\vec{\kappa}|$ with $i_1+ ...+i_{b}=a$ and $i_{j} \leq c \leq k-1$. Then by H\"older's inequality and Lemma~\ref{leminter}
\begin{align*}
\int_I |\nabla_s^{i_1} \vec{\kappa}|\cdot \ldots \cdot| \nabla_s^{i_{b}}\vec{\kappa} | \, ds & \leq \prod_{j=1}^{b} \| \nabla_s^{i_j} \vec{\kappa} \|_{L^b} = \mathcal{L}[f]^{1-a-b} \prod_{j=1}^{b} \| \nabla_s^{i_j} \vec{\kappa} \|_{b}\\
& \leq \mathcal{L}[f]^{1-a-b} \prod_{j=1}^{b} C \|\vec{\kappa} \|_{2}^{1-\gamma_j} \|\vec{\kappa} \|_{k,2}^{\gamma_j} \, 
\end{align*}
with $\gamma_j=(i_{j}+\frac12 -\frac{1}{b})/k$, from which the first claim follows directly.

For the second claim, notice that each element of the sum is of  type $|\nabla_s^{i_1} \vec{\kappa} |\cdot \ldots \cdot| \nabla_s^{i_{b}}\vec{\kappa}|$ with $i_1+ ...+i_{b}=a \leq A$, $2 \leq b \leq \min\{M, 2A+B-2a \}$ and $i_{j} \leq c \leq k-1$. In particular, $a+\frac12 b \leq A + \frac12 B$ (see \eqref{rel-ab}). Then by the first claim and Corollary \ref{corinter} we obtain
\begin{align*}
\int_I |\nabla_s^{i_1} \vec{\kappa} |\cdot \ldots \cdot| \nabla_s^{i_{b}}\vec{\kappa} | \, ds & \leq C \mathcal{L}[f]^{1-a-b} \|\vec{\kappa} \|_{2}^{b-\gamma} \|\vec{\kappa} \|_{k,2}^{\gamma} \\
& \leq  C \mathcal{L}[f]^{1-a-b} \|\vec{\kappa} \|_{2}^{b-\gamma} ( \|\nabla_s^{k} \vec{\kappa} \|_{2}^{\gamma} +\|\vec{\kappa} \|_{2}^{\gamma} )\\
& \leq C  \|\vec{\kappa} \|_{L^2}^{b-\gamma} \|\nabla_s^{k} \vec{\kappa} \|_{L^2}^{\gamma} +C \mathcal{L}[f]^{1-a-\frac12 b}  \|\vec{\kappa} \|_{L^2}^{b} \, ,
\end{align*}
with $0<\gamma=(a+\frac12 b-1)/k \leq (A+\frac12 B-1)/k<2$ (the last inequality being true by assumption). Then by Young's inequality
\begin{align*}
\int_I |\nabla_s^{i_1} \vec{\kappa} |\cdot \ldots \cdot| \nabla_s^{i_{b}}\vec{\kappa} | \, ds
& \leq \epsilon \|\nabla_s^{k} \vec{\kappa} \|_{L^2}^2 +C \epsilon^{-\frac{\gamma}{2-\gamma}} \| \vec{\kappa} \|_{L^2}^{2 \frac{b-\gamma}{2-\gamma}}  +C\mathcal{L}[f]^{1-a-\frac12 b}  \|\vec{\kappa} \|_{L^2}^{b} ,
\end{align*}
that gives the claim using that $0<\gamma \leq \overline{\gamma}$ and $2 \leq b \leq M$. The term with $\gamma=0$ is taken care of by $C \| \vec{\kappa}\|_{L^{2}}^{2}$.
\end{proof}

The following estimates will also be useful in the proof of long-time existence.
\begin{lemma}
\label{bfm}
Assume that $\|\vec{\kappa}\|_{L^{2}} \leq C$. If $\|\nabla_{t}^{m}f\|_{L^{2}} \leq C$, for some  $1 \leq m$, then it follows that $$\| \nabla_{s}^{i} \vec{\kappa} \|_{L^{2}} \leq C, \qquad \text{ for all }\,  0\leq i \leq 4m-2.$$
The constant $C$ depends on $\lambda$, $n$, $m$ and on the lower bound on $\mathcal{L}[f]$.
\end{lemma}
\begin{proof}
The result follows using \eqref{E5} in Lemma \ref{lemLin}, Lemma
\ref{lemineq}, and the bound $\|\vec{\kappa}\|_{L^{2}} \leq C$.  
Indeed from \eqref{E5} we derive
{\allowdisplaybreaks 
\begin{align*}
\|\nabla_{s}^{4m-2}\vec{\kappa}\|_{L^{2}}^{2} & \leq C \|\nabla_{t}^{m}f\|^{2}_{L^{2}} +
\int_{I} \sum_{\substack{[[a,b]] \leq [[8m-8,6]]\\c\leq 4m-4\\b\in[6, 8m-2],even}} | P^{a,c}_{b} (\vec{\kappa}) | \; ds  +\\
& \quad +  \sum_{\substack{i=2\\ even}}^{2m} \lambda^{i} \int_{I} \sum_{\substack{[[a,b]] \leq [[8m-4-2i,2]]\\c\leq 4m-2-i\\b\in[2, 8m-2-2i],even}} | P^{a,c}_{b} (\vec{\kappa}) | \; ds.
\end{align*}
}
Applying Lemma \ref{lemineq} to the second and third term (with $k=4m-2$) we can bound $\|\nabla_{s}^{4m-2}\vec{\kappa}\|_{L^{2}}$.
The bound on $\|\nabla_{s}^{j}\vec{\kappa}\|_{L^{2}}^{2}$ for $j \leq 4m-3$ follows
again by using Lemma \ref{lemineq} and $\|\nabla_{s}^{4m-2}\vec{\kappa}\|_{L^{2}} \leq C$. 
\end{proof}

So far we have derived bounds for the normal component of the derivatives of the curvature. The following lemmas indicate how to gain control over the whole derivative.

\begin{lemma}
\label{Qlemma}
We have the identities
\begin{align*}
& \partial_{s} \vec{\kappa}= \nabla_{s}\vec{\kappa} -|\vec{\kappa}|^{2}\tau,\\
& \partial_{s}^{m} \vec{\kappa} = \nabla_{s}^{m} \vec{\kappa} +\tau \sum_{\substack{[[a,b]] \leq [[m-1,2]]\\c \leq m-1\\b \in [2,2[\frac{m+1}{2}]], even}} P^{a,c}_{b} 
(\vec{\kappa}) + \sum_{\substack{[[a,b]] \leq [[m-2,3]]\\c \leq m-2\\b \in [3,2[\frac{m}{2}]+1], odd}} P^{a,c}_{b} (\vec{\kappa})  \quad \text{ for } m\geq 2 \, .
\end{align*}
\end{lemma}
\begin{proof}
The first claim is obtained directly using that 
$$\partial_s \vec{\kappa}= \nabla_s \vec{\kappa} + \langle \partial_s \vec{\kappa}, \tau \rangle \tau  = \nabla_s \vec{\kappa} - |\vec{\kappa}|^2 \tau \, .$$ 
The second claim follows by induction using that 
\begin{align*}
\partial_s \Big( \tau \sum_{\substack{[[a,b]] \leq [[m-1,2]]\\c \leq m-1\\b
    \in [2,2[\frac{m+1}{2}]], even}}  \hspace{-.2cm} P^{a,c}_{b} (\vec{\kappa}) \Big) & = \tau
\hspace{-.2cm}\sum_{\substack{[[a,b]] \leq [[m,2]]\\c \leq m\\b \in
    [2,2[\frac{m+1}{2}]], even}}  \hspace{-.2cm} P^{a,c}_{b} (\vec{\kappa}) +
\hspace{-.2cm}\sum_{\substack{[[a,b]] \leq [[m-1,3]]\\c \leq m-1\\b \in
    [3,2[\frac{m+1}{2}]+1], odd}}  \hspace{-.2cm} P^{a,c}_{b} (\vec{\kappa}) \, ,\\
\partial_s \Big( \sum_{\substack{[[a,b]] \leq [[m-2,3]]\\c \leq m-2\\b \in
    [3,2[\frac{m}{2}]+1], odd}}  \hspace{-.2cm} P^{a,c}_{b} (\vec{\kappa}) \Big) & =
 \hspace{-.2cm} \sum_{\substack{[[a,b]] \leq [[m-1,3]]\\c \leq m-1\\b \in
    [3,2[\frac{m}{2}]+1], odd}} P^{a,c}_{b} (\vec{\kappa}) + \tau
\hspace{-.2cm} \sum_{\substack{[[a,b]] \leq [[m-2,4]]\\c \leq m-2\\b \in [4,2[\frac{m}{2}]+2], even}} P^{a,c}_{b} (\vec{\kappa})  \, .
\end{align*}
\end{proof}

\begin{lemma}
\label{bvollder}
Given $m\geq 1$,
assume that $\| \nabla_{s}^{m} \vec{\kappa} \|_{L^{2}} \leq C$ and $\| \vec{\kappa}\|_{L^{2}} \leq C$. Then we have that
\begin{align*}
\| \partial_{s}^{l} \vec{\kappa} \|_{L^{2}} \leq C \; \text{ for } 0\leq l \leq m.
\end{align*}
The constant $C$ depends on $n$, $m$ and on the lower bound on $\mathcal{L}[f]$.
\end{lemma}
\begin{proof}
It follows directly from Lemma \ref{Qlemma} and Lemma \ref{lemineq}. First of all note that a bound on  $\| \vec{\kappa}\|_{L^{2}} \leq C$ and on $\| \nabla_{s}^{m} \vec{\kappa} \|_{L^{2}} \leq C$ imply that 
\begin{equation}\label{m1}
\| \nabla_{s}^{l} \vec{\kappa} \|_{L^{2}} \leq C \quad \mbox{ for all }0 \leq l\leq m \, ,
\end{equation} 
by a direct application of Lemma \ref{lemineq}. Morever Lemma \ref{Qlemma} yields for any $l \in \mathbb{N}$
\begin{align}\label{m2}
\| \partial_{s}^{l} \vec{\kappa} \|^{2}_{L^{2}} & \leq C \| \nabla_{s}^{l} \vec{\kappa} \|^{2}_{L^{2}} + \int_{I}
\sum_{\substack{[[a,b]] \leq [[2l-2,4]]\\c \leq l-1\\b \in [4,2l+2], even}} |P^{a,c}_{b} 
(\vec{\kappa})| \, ds \, .
\end{align}
Then the claim follows using \eqref{m2}, \eqref{m1} and applying Lemma \ref{lemineq} with $k=l$.
\end{proof}

%%%%%%%%%%%%%%%%%%%%%%%%%%%%%%%%%%%%%%%%%%%%%%%%%%%%%%%%%%%%%%%%%%
%%%%%%%%%%% SEZIONE 5 %%%%%%%%%%%%%%%%%%%%%%%%%%%%%%%%%%%%%%%%%%%
%%%%%%%%%%%%%%%%%%%%%%%%%%%%%%%%%%%%%%%%%%%%%%%%%%%%%%%%%%%%%%%%%%

\section{Long-time existence}
\label{Sez5}
This entire Section is dedicated to the proof of the main result, Theorem~\ref{mainTh} (see Section \ref{Sez1} for a precise statement), which states global existence of the flow.

A detailed proof of short-time existence is outside the scope of this paper:
for main ideas and useful arguments we refer to \cite{MM}, \cite{Poppin}, \cite{LPS}, and \cite{EZ}.

\begin{proof}[Proof of Theorem \ref{mainTh}] In the following $C$ denotes a generic constant that may vary from line to line. We will explicitly write down  what the constant depends on.

We assume by contradiction that the solution of \eqref{ivp} does not exist globally. Then there exists a maximal time  $0<T<\infty$,  such that the solution exists only for $t \in [0,T)$.\smallskip

\textit{First Step:} \underline{Bounds on the length $\mathcal{L}[f]$ and on $\int_{I} |\vec{\kappa}|^2 \, ds$}.

By Lemma \ref{endecr} we know that the energy is decreasing in $t$. Hence for all $t \in [0,T)$, 
\begin{equation}\label{mm1}
W_{\lambda} (f(t)) \leq W_{\lambda}(f_0) \, ,
\end{equation} 
which directly implies a upper bound on the length when $\lambda>0$: indeed
\begin{equation}\label{upblength}
\mathcal{L}[f(t)] \leq \frac{1}{\lambda} W_{\lambda}(f(t)) \leq \frac{1}{\lambda} W_{\lambda}(f_0) \, . 
\end{equation}
(The case $\lambda=0$ will be dealt with later on, see \eqref{luplength2}.)
On the other hand, the lower bound on the length follows directly from the boundary conditions as follows
\begin{equation}\label{loblength}
\mathcal{L}[f(t)] \geq  |f_{+}-f_{-}|\, . 
\end{equation}
We also find
\begin{align*}
\frac12 \int_{I} |\vec{\kappa}|^2 \, ds & \leq \frac12 \int_{I} |\vec{\kappa}|^2 \, ds - \int_{I} \langle \vec{\kappa}, \zeta \rangle \, ds  + \Big| \int_{I} \langle \vec{\kappa}, \zeta \rangle \, ds \Big| %\\
& \leq W_{\lambda}(f_0) +  \Big| [ \langle \tau, \zeta \rangle ]_0^1 \Big|  ,
\end{align*}
 hence
\begin{align}
\label{upbcurv}
\int_{I} |\vec{\kappa}|^2 \, ds \leq C(\zeta, W_{\lambda}(f_0))\,  .
\end{align}
Note that the above inequality is independent of any control of the length of the curve.
 
\newpage

\textit{Second Step:} \underline{Expression for $\frac12 \frac{d}{dt} \int_{I}
  |\nabla_t^{m} f|^2 \, ds + \int_{I} |\nabla_t^{m} f|^2  \, ds$, $m \in \mathbb{N}$.}

As in \cite{Lin} we get using \eqref{E5} in Lemma \ref{lemLin}
\begin{align*}
& \nabla_t \nabla_t^{m} f + \nabla_s^4 \nabla_t^m f  = \nabla_t^{m+1} f + \nabla_s^4 \nabla_t^m f \\
& \qquad = \sum_{\substack{[[a,b]]\leq [4m,3]\\c\leq 4m\\b\in[3, 4m+3], odd}} P^{a,c}_{b} (\vec{\kappa}) + \sum_{i=1}^{m+1} \lambda^{i} \sum_{\substack{[[a,b]]\leq [4m+2-2i,1]\\c\leq 4m+2-2i\\b\in[1, 4m+3-2i], odd}} P^{a,c}_{b}(\vec{\kappa}) =: Y \, ,
\end{align*}
since there is a cancellation on the highest order terms. 
Note that for some universal constants $C_1,C_2,C_3$
\begin{align}
\label{Yexpand}
Y &= C_1\nabla_{s}^{4m} \vec{\kappa} * \vec{\kappa}*\vec{\kappa} + C_2 \lambda \nabla_{s}^{4m} \vec{\kappa}  +C_3 \nabla_{s}^{4m} \vec{\kappa} \\ \nonumber
& \quad + \hspace{-.3cm}\sum_{\substack{[[a,b]]\leq [4m,3]\\c\leq 4m-1\\b\in[3, 4m+3], odd}}\hspace{-.2cm}
P^{a,c}_{b} (\vec{\kappa}) + \lambda \hspace{-.2cm} \sum_{\substack{[[a,b]]\leq
    [4m-1,3]\\c\leq 4m-1\\b\in[3, 4m+1], odd}}\hspace{-.3cm} P^{a,c}_{b}(\vec{\kappa}) + \sum_{i=2}^{m+1} \lambda^{i}\hspace{-.2cm} \sum_{\substack{[[a,b]]\leq [4m+2-2i,1]\\c\leq 4m+2-2i\\b\in[1, 4m+3-2i], odd}}\hspace{-.3cm} P^{a,c}_{b}(\vec{\kappa}).
\end{align}
In this way we have singled out the most critical terms in $Y$, namely those on which $4m$ derivatives fall all on one  factor. 

By Equation \eqref{eqgen} in Lemma \ref{lempartint} with $\vec{\phi}=\nabla_t^{m}f$ (recall that $\nabla_t^{m}f=0$ at the boundary) we get
\begin{align} \label{pente1}
 \frac{d}{dt} \frac{1}{2}\int_{I} |\nabla_t^{m}f|^{2} ds  + \int_{I}
 |\nabla_{s}^{2} \nabla_t^{m} f|^{2 }& ds - [ \langle \nabla_s \nabla_t^{m} f,
 \nabla_s^2 \nabla_t^{m} f \rangle ]_0^1 \\  \nonumber
& = \int_{I} \langle Y, \nabla_t^{m}f \rangle ds - \frac{1}{2} \int_{I} |\nabla_t^{m}f|^{2} \langle \vec{\kappa}, \vec{V} \rangle ds \, ,
\end{align}
with $\vec{V}= - \nabla^2_{s} \vec{\kappa}-\frac12 |\vec{\kappa}|^2 \vec{\kappa}+ \lambda \vec{\kappa}$. 
We first look at the order of the term $\nabla_s^2 \nabla_t^m f$. Again by \eqref{E5} in Lemma \ref{lemLin} (see also \eqref{hl1}) we infer
\begin{equation*}
\nabla_s^2\nabla_{t}^{m} f =  (-1)^{m} \nabla_s^{4m} \vec{\kappa} + \sum_{\substack{[[a,b]] \leq [[4m -2,3]]\\c\leq 4 m -2\\b\in[3, 4m-1],odd}} P^{a,c}_{b} (\vec{\kappa}) + \sum_{i=1}^{m} \lambda^{i} \sum_{\substack{[[a,b]] \leq [[4m-2i,1]]\\c\leq 4 m-2i\\b\in[1, 4m-1-2i],odd}} P^{a,c}_{b} (\vec{\kappa})\,. 
\end{equation*}
Using the fact that $(\sum_{i=1}^{q} a_{i})^{2} \leq q \sum_{i=1}^{q} a_{i}^{2}$ 
we derive
\begin{align*}
& |\nabla_s^2\nabla_{t}^{m} f |^2  \\
& = | \nabla_s^{4m} \vec{\kappa}|^2 \\
& \qquad + 2 (-1)^m  \nabla_s^{4m} \vec{\kappa} \Big[ \sum_{\substack{[[a,b]] \leq [[4m -2,3]]\\c\leq 4 m -2\\b\in[3, 4m-1],odd}} P^{a,c}_{b} (\vec{\kappa}) + \sum_{i=1}^{m} \lambda^{i} \sum_{\substack{[[a,b]] \leq [[4m-2i,1]]\\c\leq 4 m-2i\\b\in[1, 4m-1-2i],odd}} P^{a,c}_{b} (\vec{\kappa}) \Big]\\
&\qquad + \Big[ \sum_{\substack{[[a,b]] \leq [[4m -2,3]]\\c\leq 4 m -2\\b\in[3, 4m-1],odd}} P^{a,c}_{b} (\vec{\kappa}) + \sum_{i=1}^{m} \lambda^{i} \sum_{\substack{[[a,b]] \leq [[4m-2i,1]]\\c\leq 4 m-2i\\b\in[1, 4m-1-2i],odd}} P^{a,c}_{b} (\vec{\kappa}) \Big]^{2}\\
&  \geq (1-\epsilon_1) |\nabla_s^{4m} \vec{\kappa}|^2 \\
& \qquad - C(\epsilon_1) \Big( \sum_{\substack{[[a,b]] \leq [[8m -4,6]]\\c\leq 4 m -2\\b\in[6, 8m-2],even}} |P^{a,c}_{b} (\vec{\kappa})| + \sum_{\substack{i=2\\   even}}^{2m} \lambda^{i} \sum_{\substack{[[a,b]] \leq [[8m-2i,2]]\\c\leq 4 m-i\\b\in[2, 8m-2-2i],even}} |P^{a,c}_{b} (\vec{\kappa})| \Big).
\end{align*} 
In \eqref{pente1} we write
$|\nabla_s^2\nabla_{t}^{m} f |^2=\varepsilon_{2} |\nabla_s^2\nabla_{t}^{m} f |^2 +(1-\varepsilon_{2}) |\nabla_s^2\nabla_{t}^{m} f |^2$ 
 and we apply the inequality above to the second term. 
 Thus we obtain for \linebreak 
 any~$\varepsilon_1, \varepsilon_2 \in (0,1)$
{\allowdisplaybreaks\begin{align*}
& \frac{d}{dt} \frac12 \int_{I} |\nabla_t^{m}f|^{2} ds + \epsilon_2 \int_{I} |\nabla_{s}^{2} \nabla_t^{m} f|^{2 } ds   \\
&\quad + (1-\epsilon_2)(1-\epsilon_1) \int_{I} |\nabla_s^{4m} \vec{\kappa}|^2 ds  - [ \langle \nabla_s \nabla_t^{m} f, \nabla_s^2 \nabla_t^{m} f \rangle ]_0^1 \\
& \leq \int_{I} \langle Y, \nabla_t^{m}f \rangle ds - \frac{1}{2} \int_{I} |\nabla_t^{m}f|^{2} \langle \vec{\kappa}, \vec{V} \rangle ds \\
& \quad + C(\epsilon_1, \epsilon_2)\int_{I} \Big(\hspace{-.2cm} \sum_{\substack{[[a,b]] \leq
    [[8m -4,6]]\\c\leq 4 m -2\\b\in[6, 8m-2],even}}\hspace{-.2cm} | P^{a,c}_{b}
(\vec{\kappa})| + \sum_{\substack{i=2,\\ even }}^{2m} \lambda^{i}
\hspace{-.2cm} \sum_{\substack{[[a,b]] \leq [[8m-2i,2]]\\c\leq 4 m-i\\b\in[2, 8m-2-2i],even}} \hspace{-.2cm}| P^{a,c}_{b} (\vec{\kappa}) | \Big) ds .
\end{align*}}

Next we aim at writing the first two integrals on the right-hand side of the previous inequality in the form of  the last one. We cannot do it directly since in $Y$ there are terms where $4m$ derivatives fall on one factor (see \eqref{Yexpand}) and these would be too singular when interpolating. But using \eqref{E5} in Lemma \ref{lemLin}, the fact that $\nabla_{t}^{m}f=0$ at the boundary, and integrating by parts once the highest order terms, we obtain
\begin{align*}
\int_{I} \langle Y, \nabla_t^{m}f \rangle ds & = \int_{I}\Big(
\hspace{-.2cm}\sum_{\substack{[[a,b]] \leq [[8m -2,4]]\\c\leq 4 m -1\\b\in[4,
    8m+2],even}}  \hspace{-.2cm} P^{a,c}_{b} (\vec{\kappa}) +
\sum_{i=1}^{2m+1} \lambda^{i}  \hspace{-.2cm} \sum_{\substack{[[a,b]] \leq
    [[8m-2i,2]]\\c\leq 4 m-1\\b\in[2, 8m+2-2i],even}}  \hspace{-.2cm} P^{a,c}_{b} (\vec{\kappa}) \Big) ds \,. 
\end{align*}
Next note that since
\begin{align*}
|\nabla_{t}^{m}f|^{2} &\leq C  | \nabla_s^{4m-2} \vec{\kappa}|^{2} \\
& \quad  + C( \sum_{\substack{[[a,b]] \leq [[8m -8,6]]\\c\leq 4 m -4\\b\in[6, 8m-2],even}} |P^{a,c}_{b} (\vec{\kappa}) |+ \sum_{\substack{i=2, \\even}}^{2m} \lambda^{i} \sum_{\substack{[[a,b]] \leq [[8m -4-2i,2]]\\c\leq 4 m -2-i\\b\in[2, 8m-2-2i],even}} |P^{a,c}_{b} (\vec{\kappa}) | )\\
&\leq\sum_{\substack{[[a,b]] \leq [[8m -4,2]]\\c\leq 4 m -2\\b\in[2, 8m-2],even}} |P^{a,c}_{b} (\vec{\kappa}) |+ \sum_{\substack{i=2, \\even}}^{2m} \lambda^{i} \sum_{\substack{[[a,b]] \leq [[8m -4-2i,2]]\\c\leq 4 m -2-i\\b\in[2, 8m-2-2i],even}} |P^{a,c}_{b} (\vec{\kappa}) | )\,,
\end{align*}
we derive 
\begin{align*}
 |\int_{I} &|\nabla_t^{m}f|^2 \langle \vec{\kappa}, \vec{V} \rangle ds | \\
& \leq \int_{I}\Big( \sum_{\substack{[[a,b]] \leq [[8m -2,4]]\\c\leq 4 m -2\\b\in[4, 8m+2],even}} |P^{a,c}_{b} (\vec{\kappa})| + \sum_{i=1}^{2m+1} \lambda^{i} \sum_{\substack{[[a,b]] \leq [[8m-2-2i,4]]\\c\leq 4 m-2\\b\in[4, 8m+2-2i],even}} P^{a,c}_{b} (\vec{\kappa}) \Big) ds\,,
\end{align*}
yielding (add $\int_{I} |\nabla_t^{m}f|^{2} ds$ to both sides of \eqref{pente1})
\begin{align}\label{emlr}
& \frac{d}{dt} \frac12 \int_{I} |\nabla_t^{m}f|^{2} ds + \int_{I} |\nabla_t^{m}f|^{2} ds + \epsilon_2 \int_{I} |\nabla_{s}^{2} \nabla_t^{m} f|^{2 } ds \\ \nonumber
& \quad + (1-\epsilon_2)(1-\epsilon_1) \int_{I} |\nabla_s^{4m} \vec{\kappa}|^2 ds  - [ \langle \nabla_s \nabla_t^{m} f, \nabla_s^2 \nabla_t^{m} f \rangle ]_0^1 \\  \nonumber
& \leq C(\epsilon_1, \epsilon_2) \int_{I} \Big( \hspace{-.2cm}
\sum_{\substack{[[a,b]] \leq [[8m -2,4]]\\c\leq 4 m -1\\b\in[2, 8m+2],even}}
\hspace{-.2cm} | P^{a,c}_{b} (\vec{\kappa})| +\hspace{-.2cm} \sum_{i=1}^{2m +1} \lambda^{i} \sum_{\substack{[[a,b]] \leq [[8m-2i,2]]\\c\leq 4 m-1\\b\in[2, 8m+2-2i],even}}\hspace{-.2cm} | P^{a,c}_{b} (\vec{\kappa}) | \Big) ds .
\end{align}
\newpage

\textit{Third Step:} \underline{The boundary terms.}

By Lemma \ref{lembterms} (recall \eqref{psil}) we may write the boundary terms as
\begin{align*}
 [\langle \nabla_s & \nabla_t^{m} f,  \nabla_s^2 \nabla_t^{m} f \rangle ]_0^1 \\
& =[ \, -   \langle \zeta, \tau \rangle |\nabla_s \nabla_t^{m} f|^2  +  \langle \zeta , \nabla_s \nabla_t^{m} f \rangle \sum_{\substack{i+j=m\\i,j \geq 1}} c^{m}_{i,j}\langle \nabla_{s} \nabla_t^{i} f, \nabla_{s} \nabla^{j}_t f \rangle  \\ 
& \qquad  + \langle \zeta, \tau \rangle  \sum_{\substack{3 \leq n \leq m\\n \mbox{ \tiny odd}}} \langle\vec{R}_{n}^{m}, \nabla_s \nabla^{m}_t f \rangle  + \sum_{\substack{2 \leq n \leq m\\n \mbox{ \tiny even}}} \langle \vec{S}_{n}^{m},\nabla_s \nabla^{m}_t f \rangle  \, ]_0^1\\
& =[ I +II +III+IV ]_0^1\, .
\end{align*}
We need to bound these terms in absolute value from above.

a) We start by looking at the first boundary term $I:=-  \langle \zeta, \tau \rangle |\nabla_s \nabla_t^{m} f|^2 (x) $, for $x \in \partial I$.
Using \eqref{epsbdyht} in Lemma \ref{hilfslemmabdry} and the fact that $\zeta$ is fixed we  find for any $\epsilon>0$
\begin{align*}
 & \Big| \langle \zeta,  \tau \rangle |\nabla_s \nabla_t^{m} f|^2  \Big| \\
%& \quad \leq \tilde{\epsilon} |\zeta| \int_I |\nabla_s^2 \nabla_t^{m} f|^2 ds 
% + C(\tilde{\epsilon},\zeta) \int_I \sum_{\substack{i=0,\\ i \ even}}^{2m} \lambda^{i} \sum_{\substack{[[a,b]] \leq [[8m-2-2i,2]]\\c\leq 4m-1-i\\b\in[2,8m-2-2i], even}} |P^{a,c}_{b} (\vec{\kappa})| ds\, \\
& \quad \leq  \epsilon \int_I |\nabla_s^2 \nabla_t^{m} f|^2 ds 
+ C(\epsilon,\zeta) \int_I \sum_{\substack{i=0,\\ i \ even}}^{2m} \lambda^{i}
\sum_{\substack{[[a,b]] \leq [[8m-2-2i,4]]\\c\leq 4m-1-i\\b\in[2,8m+2-2i]}} |P^{a,c}_{b} (\vec{\kappa})| ds\, . %,
\end{align*} 
%where we have chosen $\tilde{\epsilon}=\epsilon / |\zeta|$ (recall that $\zeta$ is fixed).
%%%%%%%%%%%%

b) Next we consider the second boundary term
 $$II:= \langle \zeta , \nabla_s \nabla_t^{m} f \rangle \sum_{\substack{i+j=m\\i,j \geq 1}} c^{m}_{i,j}\langle \nabla_{s} \nabla_t^{i} f, \nabla_{s} \nabla^{j}_t f \rangle \, .$$ 
At $x \in \partial I$ it  can be estimated as follows
\begin{align*}
|II|\leq |\nabla_s \nabla_t^m f|^2 (x)    + C(\zeta)\sum_{\substack{i+j=m\\i,j \geq 1}}| \nabla_{s} \nabla_t^{i} f|^2 | \nabla_{s} \nabla^{j}_t f |^2 (x) \, .
\end{align*}

%%%%%%%%%%%%
c) Similarly for 
\begin{align*}
III & := \langle \zeta, \tau \rangle  \sum_{\substack{3 \leq n \leq m\\n \mbox{ \tiny odd}}} \langle\vec{R}_{n}^{m}, \nabla_s \nabla^{m}_t f \rangle  \, , \qquad 
IV  := \sum_{\substack{2 \leq n \leq m\\n \mbox{ \tiny even}}} \langle \vec{S}_{n}^{m},\nabla_s \nabla^{m}_t f \rangle  \, ,
\end{align*}
we have that for $x \in \partial I$
{\allowdisplaybreaks 
\begin{align*}
|III| & \leq |\nabla_s \nabla_t^m f|^2 (x) + C(\zeta) \sum_{\substack{i_{1}+ \ldots+ i_{n}=m, i_{j} \geq1 \\ n \in [3,m], odd}}
|\nabla_{s } \nabla_{t}^{i_{1}} f|^{2} \cdot \ldots \cdot |\nabla_{s }
\nabla_{t}^{i_{n}} f|^{2} (x)\,, % \quad \mbox{ and }
\\
|IV| & \leq |\nabla_s \nabla_t^m f|^2 (x) + C( \zeta) \sum_{\substack{i_{1}+ \ldots+ i_{n}=m, i_{j} \geq1 \\ n \in [2,m], even}}
|\nabla_{s } \nabla_{t}^{i_{1}} f|^{2} \cdot \ldots \cdot |\nabla_{s } \nabla_{t}^{i_{n}} f|^{2} (x) \, .
\end{align*}
}

From a), b), c), and treating the first term in the right-hand side of b) and c) as in \eqref{epsbdyht} in Lemma \ref{hilfslemmabdry}, we get that for any $\epsilon_3>0$
\begin{align*}
 - [ \langle \nabla_s & \nabla_t^{m} f, \nabla_s^2 \nabla_t^{m} f \rangle ]_0^1 \\
& \geq  - \epsilon_3 \int_I |\nabla_s^2 \nabla_t^{m} f|^2 ds -
C(\epsilon_3,\zeta) \int_I \hspace{-.2cm} \sum_{\substack{i=0,\\ i \ even}}^{2m} \lambda^{i}
\sum_{\substack{[[a,b]] \leq [[8m-2-2i,4]]\\c\leq 4m-1-i\\b\in[2,8m+2-2i]
    }} \hspace{-.2cm} |P^{a,c}_{b} (\vec{\kappa})| ds  \\
& \qquad - C(\epsilon_3,\zeta) \sum_{\substack{i_{1}+ \ldots+ i_{n}=m, i_{j} \geq1 \\ n \in [2,m]}} \max_{x \in \partial I}
|\nabla_{s } \nabla_{t}^{i_{1}} f|^{2} \cdot \ldots \cdot |\nabla_{s } \nabla_{t}^{i_{n}} f|^{2}(x) \, .
\end{align*}
Choosing $\epsilon_3=\epsilon_2$ we get from \eqref{emlr}
\begin{align}\label{emlr1}
& \frac{d}{dt} \frac12 \int_{I} |\nabla_t^{m}f|^{2} ds + \int_{I} |\nabla_t^{m}f|^{2} ds + (1-\epsilon_2)(1-\epsilon_1) \int_{I} |\nabla_s^{4m} \vec{\kappa}|^2 ds \\ \nonumber
& \quad - C(\epsilon_2,\zeta) \sum_{\substack{i_{1}+ \ldots+ i_{n}=m, i_{j} \geq1 \\ n \in [2,m]}} \max_{x \in \partial I}
|\nabla_{s } \nabla_{t}^{i_{1}} f|^{2} \cdot \ldots \cdot |\nabla_{s } \nabla_{t}^{i_{n}} f|^{2}(x)  \\  \nonumber
& \leq C(\epsilon_1, \epsilon_2) \int_{I} \Big(\hspace{-.2cm} \sum_{\substack{[[a,b]] \leq [[8m -2,4]]\\c\leq 4 m -1\\b\in[4, 8m+2]}}\hspace{-.2cm} | P^{a,c}_{b} (\vec{\kappa})| + \sum_{i=1}^{2m +1} \lambda^{i}\hspace{-.2cm} \sum_{\substack{[[a,b]] \leq [[8m-2i,2]]\\c\leq 4 m-1\\b\in[2, 8m+2-2i]}}\hspace{-.2cm} | P^{a,c}_{b} (\vec{\kappa}) | \Big) ds .
\end{align}
\smallskip

\emph{Fourth Step:} \underline{Bound on $\| \nabla_{t}^{m} f\|_{L^{2}}$, for $m =1,2$.}

We start from inequality \eqref{emlr1}. For $m=1$ it becomes
\begin{align}\label{emlr2}
& \frac{d}{dt} \frac12 \int_{I} |\nabla_t f|^{2} ds + \int_{I} |\nabla_t f|^{2} ds + (1-\epsilon_2)(1-\epsilon_1) \int_{I} |\nabla_s^{4} \vec{\kappa}|^2 ds \\  \nonumber
& \quad \leq C(\epsilon_1, \epsilon_2) \int_{I} \Big( \sum_{\substack{[[a,b]] \leq [[6,4]]\\c\leq 3\\b\in[4, 10]}} | P^{a,c}_{b} (\vec{\kappa})| + \sum_{i=1}^{3} \lambda^{i} \sum_{\substack{[[a,b]] \leq [[8-2i,2]]\\c\leq 3\\b\in[2, 10-2i]}} | P^{a,c}_{b} (\vec{\kappa}) | \Big) ds .
\end{align}
By \eqref{loblength}, \eqref{upbcurv}, and Lemma \ref{lemineq} (with $k=4$) we  find for any $\epsilon_4>0$
\begin{align*}
& \mbox{RHS in \eqref{emlr2} } \leq  \epsilon_4 \int_{I} | \nabla_s^{4} \vec{\kappa}|^2 ds + C(n,\epsilon_1,\epsilon_2,\epsilon_4,\lambda, W_{\lambda}(f_0),\zeta,f_{+},f_{-}) \, .
\end{align*}
It is interesting to note that so far we have needed only a lower bound on the length of the curve.
The above estimate in \eqref{emlr2} yields
\begin{align*}
\frac{d}{dt} \frac12 \int_{I} |\nabla_tf|^{2} ds  + \int_{I} | \nabla_t f|^{2 } ds  \leq C(n, \lambda, W_{\lambda}(f_0),\zeta,f_{+},f_{-}) \, ,
\end{align*}
which implies $\| \nabla_{t} f\|^{2}_{L^{2}}(t) \leq \| \nabla_t f
\|^{2}_{L^2} (0)+ C(n, \lambda,W_{\lambda}(f_0),\zeta,f_{+},f_{-} )$. Since
$\| \nabla_t f \|_{L^2} (0) \leq C(f_0)$ ($f_0$ is attained
                                smoothly and 
we may take the limit $t \searrow 0$ in \eqref{eqh}) we will
simply write $\| \nabla_{t} f\|_{L^{2}} \leq C(n,
\lambda,W_{\lambda}(f_0),\zeta,f_{+},f_{-}, f_{0} )$.

For $m=2$ inequality \eqref{emlr1} becomes
\begin{align*}
& \frac{d}{dt} \frac12 \int_{I} |\nabla_t^{2}f|^{2} ds + \int_{I}
|\nabla_t^{2}f|^{2} ds \\
& + (1-\epsilon_2)(1-\epsilon_1) \int_{I} |\nabla_s^{8} \vec{\kappa}|^2 ds  - C(\epsilon_2,\zeta) \max_{x \in \partial I}|\nabla_{s } \nabla_{t} f|^{4} (x) \\  
& \quad \leq C(\epsilon_1, \epsilon_2) \int_{I} \Big( \sum_{\substack{[[a,b]] \leq [[14,4]]\\c\leq 7\\b\in[4, 18]}} | P^{a,c}_{b} (\vec{\kappa})| + \sum_{i=1}^{5} \lambda^{i} \sum_{\substack{[[a,b]] \leq [[16-2i,2]]\\c\leq 7\\b\in[2, 18-2i]}} | P^{a,c}_{b} (\vec{\kappa}) | \Big) ds .
\end{align*}
Lemma \ref{hilfslemmabdry2} yields
\begin{align*}
& \frac{d}{dt} \frac12 \int_{I} |\nabla_t^{2}f|^{2} ds + \int_{I} |\nabla_t^{2}f|^{2} ds + (1-\epsilon_2)(1-\epsilon_1) \int_{I} |\nabla_s^{8} \vec{\kappa}|^2 ds   \\  
& \quad \leq C(\epsilon_1, \epsilon_2,\zeta) \int_{I} \Big( \sum_{\substack{[[a,b]] \leq [[14,4]]\\c\leq 7\\b\in[4, 18]}} | P^{a,c}_{b} (\vec{\kappa})| + \sum_{i=1}^{5} \lambda^{i} \sum_{\substack{[[a,b]] \leq [[16-2i,2]]\\c\leq 7\\b\in[2, 18-2i]}} | P^{a,c}_{b} (\vec{\kappa}) | \Big) ds ,
\end{align*}
and then using \eqref{loblength}, \eqref{upbcurv} and Lemma \ref{lemineq} (with $k=8$) we get
\begin{align*}
\frac{d}{dt} \frac12 \int_{I} |\nabla_t^{2}f|^{2} ds  + \int_{I} | \nabla_t^{2} f|^{2 } ds  \leq C(n,\lambda, W_{\lambda}(f_0),\zeta,f_{+},f_{-}).
\end{align*}
Thus $\| \nabla_{t}^{2} f\|_{L^{2}}^{2} (t)\leq \| \nabla_t^2 f\|^{2}_{L^2}(0)+C(n,\lambda, W_{\lambda}(f_0),\zeta,f_{+},f_{-})$ and as before we simply write
$\| \nabla_{t}^{2} f\|_{L^{2}} \leq C(n,\lambda, W_{\lambda}(f_0),\zeta,f_{+},f_{-},f_{0})$.
\smallskip

\emph{Fifth Step:} \underline{Bound on $\| \nabla_{t}^{m} f\|_{L^{2}}$, for $m \geq 3$.}

In order to bound  $\| \nabla_{t}^{m} f\|_{L^{2}}$ for $m \geq 3$ we proceed by induction. Let us assume that for some $m \in \mathbb{N}$, $m\geq 2$,
\begin{equation}\label{eqemlr3}
\| \nabla_{t}^{i} f\|_{L^{2}} \leq C(n,\lambda,W_{\lambda}(f_0),\zeta,f_{+},f_{-},f_0,m) \quad \mbox{ for all }1\leq i \leq m \, .
\end{equation} 
We need to show that the bound holds also for $m+1$. We first observe that \eqref{eqemlr3} implies the following estimate
\begin{align}
\label{inftybdry}
| \nabla_{s} \nabla_t^{i} f|^2(x) \leq C(n,\lambda,W_{\lambda}(f_0),\zeta,f_{+},f_{-},f_0,m), 
\end{align}
for $1\leq i\leq m -1$ and $x \in \partial I$. Indeed, if we know that $\|\nabla_{t}^{m}f\|_{L^{2}} \leq C$, ($m \geq 2$), then by Lemma \ref{bfm} we infer that $\| \nabla_{s}^{l} \vec{\kappa} \|_{L^{2}} \leq C$
 for $0 \leq l \leq 4m-2$. By \eqref{epsbdyht0}, Lemma \ref{lemineq} (take the second inequality with $k=4i+1$ and $\epsilon =1$), \eqref{upbcurv}, and \eqref{loblength}, we have for $x \in \partial I$
\begin{align*}
| \nabla_{s} \nabla_t^{i} f(x)|^2 & \leq \int_I \sum_{j=0}^{2i} \lambda^{j} \sum_{\substack{[[a,b]] \leq [[8i-2j-1,2]]\\c\leq 4i\\b\in[2, 8i+1-2j]}} | P^{a,c}_{b} (\vec{\kappa}) | ds \\
& \leq \int_I |\nabla_s^{4i+1} \vec{\kappa}|^2 ds +C (n,\lambda,W_{\lambda}(f_0),\zeta,f_{+},f_{-},m).
\end{align*}
which implies \eqref{inftybdry} directly.

Therefore using \eqref{inftybdry} formula \eqref{emlr1} with $m+1$ becomes
\begin{align} \nonumber
& \frac{d}{dt} \frac12 \int_{I} |\nabla_t^{m+1}f|^{2} ds + \int_{I} |\nabla_t^{m+1}f|^{2} ds + (1-\epsilon_2)(1-\epsilon_1) \int_{I} |\nabla_s^{4(m+1)} \vec{\kappa}|^2 ds \\ \nonumber
& \quad - C(\epsilon_2,\zeta) \max_{x \in \partial I}
|\nabla_{s } \nabla_{t}^{m} f|^{2} \cdot |\nabla_{s } \nabla_{t} f|^{2}(x)  -C\\  \nonumber
& \leq C(\epsilon_1, \epsilon_2) \int_{I} \sum_{\substack{[[a,b]] \leq [[8(m+1) -2,4]]\\c\leq 4 (m+1) -1\\b\in[4, 8(m+1)+2]}} | P^{a,c}_{b} (\vec{\kappa})| ds\\ \label{emlr4}
& \quad +  C(\epsilon_1, \epsilon_2) \int_{I}  \sum_{i=1}^{2(m+1) +1} \lambda^{i} \sum_{\substack{[[a,b]] \leq [[8(m+1)-2i,2]]\\c\leq 4 (m+1)-1\\b\in[2, 8(m+1)+2-2i]}} | P^{a,c}_{b} (\vec{\kappa}) | ds .
\end{align}
By \eqref{epsbdyht0} in Lemma \ref{hilfslemmabdry} and \eqref{inftybdry} the boundary term can be absorbed in the right-hand side of \eqref{emlr4}, namely
{\allowdisplaybreaks
\begin{align*}
& \frac{d}{dt} \frac12 \int_{I} |\nabla_t^{m+1}f|^{2} ds + \int_{I} |\nabla_t^{m+1}f|^{2} ds + (1-\epsilon_2)(1-\epsilon_1) \int_{I} |\nabla_s^{4(m+1)} \vec{\kappa}|^2 ds \\   
&\leq C+ C(\epsilon_1, \epsilon_2, \zeta) \int_{I}
\sum_{\substack{[[a,b]] \leq [[8(m+1) -2,4]]\\c\leq 4 (m+1) -1\\b\in[2,
    8(m+1)+2]}} | P^{a,c}_{b} (\vec{\kappa})|\,
ds \, +\\
&\quad  + C(\epsilon_1, \epsilon_2, \zeta) \int_{I}  \sum_{i=1}^{2(m+1) +1} \lambda^{i} \sum_{\substack{[[a,b]] \leq [[8(m+1)-2i,2]]\\c\leq 4 (m+1)-1\\b\in[2, 8(m+1)+2-2i]}} | P^{a,c}_{b} (\vec{\kappa}) |  ds .
\end{align*}
}
The right-hand side of the above inequality can be estimated using \eqref{loblength}, \eqref{upbcurv} and Lemma \ref{lemineq} (with $k=4(m+1)$). Finally we get
\begin{align*}
\frac{d}{dt} \frac12 \int_{I} |\nabla_t^{m+1}f|^{2} ds  + \int_{I} | \nabla_t^{m+1} f|^{2 } ds  \leq C(n,\lambda,W_{\lambda}(f_0),\zeta,f_{+},f_{-},f_0,m) \, ,
\end{align*}
and hence $\| \nabla_t^{m+1}f\|_{L^{2}} \leq C(n,\lambda,W_{\lambda}(f_0),\zeta,f_{+},f_{-},f_0,m+1)$.
\smallskip

\emph{Sixth Step:} \underline{Bound on $\| \partial_{s}^{l} \vec{\kappa}\|_{L^{\infty}}$ for $l \in \mathbb{N}_0$.}

By the result in the previous step, Lemma \ref{bfm}, Lemma \ref{bvollder} and \eqref{upbcurv} we can find bounds 
\begin{equation}\label{kdertotal}
\| \nabla_{s}^{l} \vec{\kappa}\|_{L^{2}} \, , \| \partial_{s}^{l} \vec{\kappa}\|_{L^{2}} \leq C(n,\lambda,W_{\lambda}(f_0),\zeta,f_{+},f_{-},f_0,l) \, ,
\end{equation}
for any $l  \in \mathbb{N}_0$. 

We prove now that the length remains bounded in $[0,T)$ when $T<\infty$ for any  $\lambda \geq 0$ (recall that \eqref{upblength} holds for positive $\lambda$ only). 
As we will see below a control of the length (from below \emph{and} above) is needed when applying embedding theory.
Using \eqref{a}, \eqref{eqh}, \eqref{kdertotal} and Lemma \ref{lemineq} (with $A=0$, $B=4$, $k=1$) we get
\begin{equation*}
\frac{d}{dt} \mathcal{L}[f] + \lambda \int_I |\vec{\kappa}|^2 \, ds \leq C(n,\lambda,W_{\lambda}(f_0),\zeta,f_{+},f_{-},f_0) \, ,
\end{equation*}
from which we infer that
\begin{equation}\label{luplength2}
\mathcal{L}[f] \leq C(n,\lambda,W_{\lambda}(f_0),\zeta,f_{+},f_{-},f_0,T)  \, .
\end{equation}
By Lemma \ref{lemBGH} we find for any $l \in \mathbb{N}_0$
\begin{align*}
\| \partial_s^{l} \vec{\kappa}\|_{L^{\infty}} & \leq c(n) \|  \partial_s^{l+1} \vec{\kappa}\|_{L^1} + \frac{c(n)}{\mathcal{L}[f]} \|\partial_s^{l} \vec{\kappa}\|_{L^{1}} \\ 
& \leq c(n) \mathcal{L}[f]^{\frac12} \|  \partial_s^{l+1} \vec{\kappa}\|_{L^2} + \frac{c(n)}{\mathcal{L}[f]^{\frac12}} \|\partial_s^{l} \vec{\kappa}\|_{L^{2}} \, ,
\end{align*}
which together with \eqref{kdertotal}, \eqref{loblength} and \eqref{luplength2} yields
\begin{equation}\label{sept1}
\| \partial_s^{l} \vec{\kappa}\|_{L^{\infty}}  \leq  C(n,\lambda,W_{\lambda}(f_0),\zeta,f_{+},f_{-},f_0,T) \, .
\end{equation}
From \eqref{kdertotal} and \eqref{sept1} we also easily derive 
\begin{align}
\label{sept1b}
\|\partial_{s}^{l} V \|_{L^{2}}, \|\partial_{s}^{l} V \|_{L^{\infty}} \leq C(n,\lambda,W_{\lambda}(f_0),\zeta,f_{+},f_{-},f_0,T) 
\end{align}
for any $l  \in \mathbb{N}_0$.
\smallskip

\emph{Seventh Step:} \underline{Bound on $\| \partial_{x}^{l} \vec{\kappa}\|_{L^{\infty}}$ for $l \in \mathbb{N}_0$.}

Here we follow the reasoning presented in \cite[page 1234]{DKS}. For simplicity of notation let $\gamma:= |\partial_x f|$. Then, $\partial_x = \gamma \,\partial_s$. By induction it can be proven that for any function $h:\bar{I} \to \R$ or vector field $h:\bar{I} \to \R^{n}$, and for any $m \in \mathbb{N}$
\begin{equation}\label{sept4}
\partial_x^m h= \gamma^m \partial_s^m h + \sum_{j=1}^{m-1} P_{m-1}(\gamma, .. , \partial_x^{m-j} \gamma) \partial_s^j h \, ,
\end{equation}
with $P_{m-1}$ a polynomial of degree at most $m-1$. A bound on $\| \partial_{x}^{l} \vec{\kappa}\|_{L^{\infty}}$ follows from \eqref{sept4} taking $h= \vec{\kappa}$ and from bounds on  $\| \partial_{s}^{l} \vec{\kappa}\|_{L^{\infty}}$ (see \eqref{sept1}) and on $\|\partial_x^l \gamma\|_{L^{\infty}}$. 

Thus it remains  to estimate $\|\partial_x^l \gamma\|_{L^{\infty}}$ for $l \in \mathbb{N}_0$.
We start by showing that $\gamma = |\partial_x f|$ is uniformly bounded from above and below. This fact is also important because we want the flow to be regular over time. The function $\gamma$ satisfies the following parabolic equation
\begin{equation}\label{sept2}
\partial_t \gamma = \langle \tau , \partial_x \vec{V} \rangle =  - \langle \vec{\kappa}, \vec{V} \rangle \gamma \, .
\end{equation}
Moreover by assumption on the initial datum we know that $1/c_{0} \leq \gamma (0) \leq c_{0} $ for some positive $c_{0}$.
From the estimates \eqref{sept1} and \eqref{sept1b} it follows that  the coefficient $\| \langle \vec{\kappa}, \vec{V} \rangle\|_{L^{\infty}}$  in \eqref{sept2} is uniformly bounded and hence we  infer that $1/C \leq \gamma \leq C$, with $C$ having the same dependencies as the constant in~\eqref{sept1}.

In order to prove bounds on $\partial_x^m \gamma$ we proceed by induction. Let us assume that 
\begin{equation}\label{sept3}
\| \partial_x^m \gamma \|_{L^{\infty}} \leq C(n,\lambda,W_{\lambda}(f_0),\zeta,f_{+},f_{-},f_0,T,m) \, \mbox{ for some }m\geq 0 \, .
\end{equation}  
Choosing $h=\langle \vec{\kappa}, \vec{V} \rangle$ in \eqref{sept4}, the induction assumption and \eqref{sept1} yield that 
\begin{equation}\label{sept5}
\| \partial_x^{i} \langle \vec{\kappa}, \vec{V} \rangle \|_{L^{\infty}} \leq C(n,\lambda,W_{\lambda}(f_0),\zeta,f_{+},f_{-},f_0,T,m)  
\end{equation}
for all $ 0 \leq i \leq m+1 $. 
Differentiating \eqref{sept2} $(m+1)$-times with respect to $x$, we find
\begin{equation*}
\partial_t \partial_x^{m+1} \gamma = - \langle \vec{\kappa}, \vec{V} \rangle \partial_x^{m+1} \gamma - \sum_{\substack{i+j=m+1\\j \leq m}} c(i,j,m) \partial_x^{i}(\langle \vec{\kappa}, \vec{V} \rangle) \partial_x^j \gamma\, , 
\end{equation*}
for some coefficients $c(i,j,m)$. Together
 with \eqref{sept3}, \eqref{sept5} we derive
\begin{equation*}
\partial_t \partial_x^{m+1} \gamma \leq  - \langle \vec{\kappa}, \vec{V} \rangle \partial_x^{m+1} \gamma +  C(n,\lambda,W_{\lambda}(f_0),\zeta,f_{+},f_{-},f_0,T,m)\, , 
\end{equation*}
which implies
\begin{equation*}
\| \partial_x^{m+1} \gamma \|_{L^{\infty}} \leq C(n,\lambda,W_{\lambda}(f_0),\zeta,f_{+},f_{-},f_0,T,m+1) \, .
\end{equation*}

Finally note that from \eqref{sept1b}  and \eqref{sept4} we obtain  also uniform estimates for $\|\partial_{x}^{m}\vec{V}\|_{L^{\infty}}$.
\smallskip

\emph{Eighth Step: }\underline{Long-time existence.}

The uniform $L^{\infty}$-bounds on the curvature $\vec{\kappa}$, the velocity $\vec{V}$, $\gamma$, and all their derivatives, allow for a smooth  extension of $f$ up  to $t=T$ and then by the short-time existence result even beyond. In view of this contradiction, the flow must exist globally. 
\smallskip   

\emph{Ninth Step: }\underline{Subconvergence to a critical point for $\lambda >0$.}

Here we follow the reasoning given in \cite[Page 1235]{DKS}. Since $\lambda >0$ we can use \eqref{upblength} instead of \eqref{luplength2} to estimate the length from above, so that together with \eqref{loblength} we obtain
\begin{equation}
\label{fullboundlength}
|f_{+}-f_{-}| \leq \mathcal{L}[f] \leq C(\lambda, W_{\lambda}(f_{0})) \qquad\text{ for all } t \in [0,\infty).
\end{equation}
 In this way  we get for \eqref{sept1} and \eqref{sept1b}  estimates independent of $T$, thus 
\begin{equation}\label{sept7}
\| \partial_s^{l} \vec{\kappa}\|_{L^{\infty}}, \| \partial_s^{l}
\vec{V}\|_{L^{\infty}}   \leq
C(n,\lambda,W_{\lambda}(f_0),\zeta,f_{+},f_{-},f_0) 
\end{equation}
for any $l \in \mathbb{N}_0$, for all $t \in [0,\infty)$. Next we observe that $\| f\|_{L^{\infty}} \leq C$ for all $t \in [0,\infty)$ due to the upper bound on the length and the fixed end-points of the curve. Hence one naturally expects some sort of convergence of subsequences. Using \eqref{fullboundlength} and reparametrizing $f$ by arc-length in order to have a control on the parametrization (which otherwise could become non regular at $T=\infty$) one can show  that there exist sequences of times $t_i \rightarrow \infty$ such that the curves $f(t_i, \cdot)$ converges smoothly to a smooth curve $f_{\infty}$.

 It remains to show that $f_{\infty}$ is a critical point for the Willmore-Helfrich functional, that is, a solution to $\vec{V}=0$. We prove this by considering the function  $u(t) := \| \vec{V}\|_{L^2}^2(t)$ and showing that $\lim_{t \rightarrow \infty} u(t)=0$. First observe that
\begin{equation*} 
 \frac{d}{dt} u(t) = - \int_{I} |\vec{V}|^2 \langle \vec{\kappa}, \vec{V} \rangle \, ds + \int_{I} \langle \vec{V}, \nabla_t \vec{V} \rangle \, ds . 
\end{equation*}
Since $\nabla_t \vec{V} = \nabla_t^2 f$ we infer from \eqref{upblength},
\eqref{sept7}  and the bounds derived in the Fourth Step that
\begin{equation*} 
\Big|\frac{d}{dt} u(t)\Big| \leq C(n,\lambda,W_{\lambda}(f_0),\zeta,f_{+},f_{-},f_0) \, . 
\end{equation*}
On the other hand from 
\begin{equation*}
\frac{d}{dt} W_{\lambda}(t) = - \int_{I} |\vec{V}|^2\, ds \, ,
\end{equation*}
(see proof of Lemma \ref{endecr}) it follows that $u \in L^1((0,\infty))$ and hence necessarily $u(t) \rightarrow 0$ for $t \rightarrow \infty$. The limit curve $f_{\infty}$ is therefore a critical point of the Willmore-Helfrich functional.
\end{proof}

%%%%%%%%%%%%%%%%%%%%%%%%%%%%%%%%%%%%%%%%%%%%%%%%%%
%%%%%%%%%%%%%%%%%%%%%%% APPENDIX  %%%%%%%%%%%%%%%%
%%%%%%%%%%%%%%%%%%%%%%%%%%%%%%%%%%%%%%%%%%%%%%%%%%

\renewcommand{\thesection}{}

\appendix\renewcommand{\thesection}{\Alph{section}}
\setcounter{equation}{0}
\renewcommand{\theequation}{\Alph{section}\arabic{equation}}

%%%%%%%%%%%%%%%%%%%%%%%%%%%%%%%%%%%%%%%%%%%%%%%%
%%%%%%%%%%%% Appendix Variation %%%%%%%%%%%%%%%%
%%%%%%%%%%%%%%%%%%%%%%%%%%%%%%%%%%%%%%%%%%%%%%%%

\section{First variation and decrease of the energy}
\label{AppVar}

Let $f:\bar{I} \to \R^{n}$
be a regular parametrization of a smooth curve in $\R^{n}$.
Define the following functionals
\begin{align*}
\mathcal{L} (f) & := \int_{I}  ds = \int_{I} |\partial_{x} f| dx \,,\\ % \qquad \mbox{ Length functional}\\
\mathcal{E} (f) &:=  \frac12 \int_{I} |\vec{\kappa}|^2 \, ds \,, \qquad \qquad 
%\qquad \mbox{ Elastic energy}\\
\mathcal{K}_{\zeta} (f) := \int_{I} \langle \vec{\kappa}, \zeta \rangle \, ds  \, ,
\end{align*}
with $\zeta \in \mathbb{R}^n$ a fixed vector.

\begin{lemma}[The first variation]\label{lemfv}
Suppose $f:\bar{I}=[0,1] \rightarrow \mathbb{R}^n$ is a smooth regular curve in $\mathbb{R}^n$. Then for any  perturbation of $f$ of the kind $f_{\epsilon}=f+\epsilon \eta$ with $\eta \in C^{\infty}(\bar{I}; \mathbb{R}^n)$  and satisfying $\eta(0)=\eta(1)=0$, one has the following formulas
\begin{align*}
\left. \frac{d}{d \epsilon} \mathcal{L}[f_{\epsilon}] \right|_{\epsilon=0} & = - \int_{I} \langle \vec{\kappa}, \eta \rangle \; ds , \quad \quad \left. \frac{d}{d \epsilon} \mathcal{K}_{\zeta}[f_{\epsilon}] \right|_{\epsilon=0}  = [\langle \nabla_s \eta, \zeta \rangle  ]_{0}^{1} \, ,\\
\left. \frac{d}{d \epsilon} \mathcal{E}[f_{\epsilon}] \right|_{\epsilon=0} & = \int_{I} \langle \nabla_s^2\vec{\kappa}+\frac12 |\vec{\kappa}|^2 \vec{\kappa}, \eta \rangle \; ds + [ \langle \nabla_s \eta, \vec{\kappa} \rangle  ]_{0}^{1} .\end{align*}
In particular, $f$ is a critical point for the Willmore-Helfrich functional given in \eqref{Wh} among all curves with fixed endpoints $f_{-}, f_{+} \in \mathbb{R}^n$ if $f$ satisfies~\eqref{ebvp}.
\end{lemma}
\begin{proof}
Since 
$\left. \frac{d}{d \epsilon} (ds_{\epsilon})  \right|_{\epsilon=0} =\langle \tau,
\partial_{s} \eta \rangle ds $,
we have
\begin{equation*}
\left. \frac{d}{d \epsilon} \mathcal{L}[f_{\epsilon}]\right|_{\epsilon=0}  = \int_{I} \langle \tau , \partial_{s} \eta \rangle ds = - \int_{I} \langle \vec{\kappa}, \eta \rangle ds ,
\end{equation*}
since $\eta$ is zero on the boundary. Using that
$
\left. \frac{d}{d \epsilon} \vec{\kappa}_{\epsilon} \right|_{\epsilon=0}  =
\partial_{s}\nabla_{s} \eta  - \langle \tau, \partial_{s} \eta \rangle \vec{\kappa} 
$
the expression for the first variation of $\mathcal{K}_{\zeta}$ follows immediately.
Finally, for the elastic energy we derive
\begin{align*}
\left. \frac{d}{d \epsilon} \mathcal{E}[f_{\epsilon}] \right|_{\epsilon=0} & = 
\int_{I} \langle \vec{\kappa}, \partial_{s} \nabla_{s} \eta \rangle ds 
-\frac{1}{2} \int_{I} |\vec{\kappa}|^{2} \langle \tau, \partial_{s} \eta \rangle ds\\
& =[ \langle \nabla_s \eta, \vec{\kappa} \rangle  ]_{0}^{1} -[ \langle \nabla_s \vec{\kappa}, \eta \rangle  ]_{0}^{1} + \int_{I}\langle \eta, \partial_{s} \nabla_{s} \vec{\kappa} \rangle ds  \\
&\qquad  -\frac{1}{2} \int_{I}|\vec{\kappa}|^{2} \partial_{s}(\langle \tau,  \eta \rangle) ds +\frac{1}{2} \int_{I}|\vec{\kappa}|^{2} \langle \vec{\kappa},  \eta \rangle ds\\
& = [ \langle \nabla_s \eta, \vec{\kappa} \rangle  
-\langle \nabla_s \vec{\kappa}, \eta \rangle  
-\frac{1}{2} 
|\vec{\kappa}|^{2}\langle \tau, \eta \rangle  ]_{0}^{1} +
\int_{I} \langle \nabla_s^2\vec{\kappa}+\frac12 |\vec{\kappa}|^2 \vec{\kappa}, \eta \rangle  ds.
\end{align*}
The second part of the claim follows directly from the formulas of the first variation.
\end{proof}

\begin{lemma}[The energy decreases]\label{endecr}
Let $f:[0,T) \times \bar{I} \rightarrow \mathbb{R}^n$ be a sufficiently smooth solution of \eqref{eqh} satisfying \eqref{bch} for all $t$. Let the Willmore-Helfrich energy be defined as in \eqref{Wh}. Then, 
\begin{equation*}
\frac{d}{dt} W_{\lambda}(f) \leq 0 \, .
\end{equation*}
\end{lemma}
\begin{proof}
From the definition of the energy and Lemma \ref{lemform} formulas \eqref{e}, \eqref{e0}, \eqref{a}  we obtain
\begin{align*}
\frac{d}{dt} W_{\lambda}(f) & = \int_{I} \left( \langle \vec{\kappa}, \nabla_{t} \vec{\kappa} \rangle - \langle \zeta, \partial_{t} \vec{\kappa} \rangle \right) ds + \int_{I} \left( \frac12 |\vec{\kappa}|^2-\langle \zeta,\vec{\kappa} \rangle +\lambda \right) \partial_{t}(ds) \\ 
& =  \int_{I} \left( \langle \vec{\kappa}, \nabla_s^2 \vec{V} + \langle \vec{\kappa}, \vec{V} \rangle \vec{\kappa} \rangle - \langle \zeta, \partial_s \nabla_s \vec{V} + \langle \vec{\kappa}, \vec{V} \rangle \vec{\kappa}  \rangle \right) ds \\
&\quad  - \int_{I} \left( \frac12 |\vec{\kappa}|^2-\langle \zeta,\vec{\kappa} \rangle +\lambda \right) \langle \vec{\kappa}, \vec{V} \rangle  ds \\
& =  \int_{I}  \langle \vec{\kappa}, \nabla_s^2 \vec{V}  \rangle  ds
 - \int_{I} \langle \zeta, \partial_s \nabla_s \vec{V} \rangle ds + \int_{I}  \langle \frac12  \vec{\kappa} | \vec{\kappa}|^2 -\lambda \vec{\kappa}, \vec{V} \rangle  ds  ,
\end{align*}
and integrating by parts
\begin{align*}
\frac{d}{dt} W_{\lambda}(f) & = [ \langle \vec{\kappa} -\zeta, \nabla_s \vec{V}  \rangle ]_0^1
 -  \int_{I}  \langle \nabla_s \vec{\kappa}, \nabla_s \vec{V}  \rangle  ds  + \int_{I}  \langle \frac12  \vec{\kappa} | \vec{\kappa}|^2 -\lambda \vec{\kappa}, \vec{V} \rangle  ds \\
& =  -  [ \langle \nabla_s \vec{\kappa}, \vec{V}  \rangle ]_0^1  + \int_{I}  \langle \nabla_{s}^2 \vec{\kappa} + \frac12  \vec{\kappa} | \vec{\kappa}|^2 -\lambda \vec{\kappa}, \vec{V} \rangle  ds \\
& = - \int_{I}  | \vec{V} |^2  ds \leq 0 ,
\end{align*}
using the boundary conditions, the fact that $\vec{V}$ is zero at the boundary and the equation~\eqref{eqh}.
\end{proof}

%%%%%%%%%%%%%%%%%%%%%%%%%%%%%%%%%%%%%%%%%%%%%%%%%%%%%%%%%%%%%%%%%%%%%%%%%%%%%%%%%%%%%%%%%%%%%%
%%%%%%%%%%%%%%%%%%% Appendix Lin %%%%%%%%%%%%%%%%%%%%%%%%%%%%%%%%%%%%%%%%%%%%%%%%%%%%%%%%%%%%%
%%%%%%%%%%%%%%%%%%%%%%%%%%%%%%%%%%%%%%%%%%%%%%%%%%%%%%%%%%%%%%%%%%%%%%%%%%%%%%%%%%%%%%%%%%%%%%

\section{Proof of Lemma \ref{lemLin}}\label{AppLin}

%%%%%%%%%%%%%%%%%%%%%E1

\begin{proof}[Proof of \eqref{E1} in Lemma \ref{lemLin}]
For simplicity of notation let $\vec{\xi}^{l}=\nabla_s^{l}\vec{\kappa}$. We prove the claim by induction on $k$. For $k=1$ and any $l \in \mathbb{N}_{0}$ we have by~\eqref{f}
\begin{align}\nonumber
[ \nabla_{t} \nabla_{s}  -\nabla_{s} \nabla_{t}]  \vec{\xi}^{l} & =
\langle \vec{\kappa} , -\nabla_s^2 \vec{\kappa} -\frac12 |\vec{\kappa}|^2 \vec{\kappa} + \lambda \vec{\kappa} \rangle \nabla_{s} \vec{\xi}^{l} \\ \nonumber
& \quad  + \langle \vec{\kappa}, \vec{\xi}^{l} \rangle \left( -\nabla_s^3 \vec{\kappa} -\frac12 |\vec{\kappa}|^2 \nabla_s  \vec{\kappa} - \langle \vec{\kappa}, \nabla_s \vec{\kappa} \rangle \vec{\kappa}  + \lambda \nabla_s \vec{\kappa} \right) \\ \nonumber
& \quad - \langle \left( -\nabla_s^3 \vec{\kappa} -\frac12 |\vec{\kappa}|^2 \nabla_s  \vec{\kappa} - \langle \vec{\kappa}, \nabla_s \vec{\kappa} \rangle \vec{\kappa}  + \lambda \nabla_s \vec{\kappa} \right) , \vec{\xi}^{l} \rangle \vec{\kappa} \\ \nonumber
 & = P_{3}^{l+3,\max\{l+1,2\}}(\vec{\kappa}) +  P_{5}^{l+1,l+1}(\vec{\kappa}) +  \lambda P_{3}^{l+1,l+1}(\vec{\kappa}) \\ \nonumber
& \quad +  P_{3}^{l+3,\max\{l,3\}}(\vec{\kappa}) +  P_{5}^{l+1,\max \{ l,1 \}}(\vec{\kappa}) + \lambda  P_{3}^{l+1,\max \{ l,1 \}}(\vec{\kappa})\\ \label{riemcen}
& =  \sum_{\substack{[[a,b]] \leq [[l+1+2,3]]\\ c \leq \max\{l,2\}+1\\b\in[3,5], odd}} P^{a,c}_{b} (\vec{\kappa}) + \lambda
\sum_{\substack{[[a,b]] \leq [[l+1,3]]\\ c \leq l+1\\b=3}} P^{a,c}_{b} (\vec{\kappa}) 
\, .
\end{align}

We assume that the claim holds true for some $k\geq 1$ and any $l \in \mathbb{N}_0$. Then for any $ l \in \mathbb{N}_0$
\begin{align*}
[ &\nabla_{t} \nabla_{s}^{k+1} -\nabla_{s}^{k+1} \nabla_{t}]  \vec{\xi}^{l}  = \nabla_{s} \left[ (\nabla_{t} \nabla_{s}^{k}  -\nabla_{s}^{k} \nabla_{t}) \vec{\xi}^{l} \right] + [\nabla_{t} \nabla_{s} -\nabla_s \nabla_t] \nabla_s^{k} \vec{\xi}^{l} \\
& = \nabla_{s} \Big[
\sum_{\substack{[[a,b]] \leq [[l+k+2,3]]\\c\leq \max\{l,2\}+k\\b\in[3,5], odd}} P^{a,c}_{b} (\vec{\kappa})  + \lambda
\sum_{\substack{[[a,b]] \leq [[l+k,3]]\\ c \leq l+k\\b=3}} P^{a,c}_{b} (\vec{\kappa})
\Big]  %\\
%& \qquad 
+ [ \nabla_t \nabla_s -\nabla_s \nabla_t ]\vec{\xi}^{k+l} \\
& = \sum_{\substack{[[a,b]] \leq [[l+k+3,3]]\\c\leq \max\{l,2\}+k+1\\b\in[3,5], odd}} P^{a,c}_{b} (\vec{\kappa}) \,  + \lambda
\sum_{\substack{[[a,b]] \leq [[l+k+1,3]]\\ c \leq l+k+1\\b=3}} P^{a,c}_{b} (\vec{\kappa})
,
\end{align*}
using in the last step \eqref{riemcen} with $k+l$ instead of $l$.
\end{proof}

%%%%%%%%%%%%%%%%%%%%%%%%%%%%E2

\begin{proof}[Proof of \eqref{E2} in Lemma \ref{lemLin}]
We prove the claim by induction on $m$.
Since $P^{\mu,d}_{\nu}$ is linear combination of terms of the type 
$$ \langle \nabla_{s}^{i_{1}}\vec{\kappa}, \nabla_{s}^{i_{2}}\vec{\kappa} \rangle \ldots 
\langle \nabla_{s}^{i_{\nu-2}} \vec{\kappa}, \nabla_{s}^{i_{\nu-1}}\vec{\kappa} \rangle \nabla_{s}^{i_{\nu}}\vec{\kappa} $$
with $i_1+ \dots +i_{\nu}=\mu$ and $\max \{ i_{j} \} \leq d$, by Leibnitz's rule we need to understand terms of the kind
$$ \langle \nabla_{s}^{i_{1}}\vec{\kappa}, \nabla_{s}^{i_{2}}\vec{\kappa} \rangle \ldots \langle \nabla_t \nabla_{s}^{i_{j}} \vec{\kappa}, ..  \rangle \dots 
\langle \nabla_{s}^{i_{\nu-2}} \vec{\kappa}, \nabla_{s}^{i_{\nu-1}}\vec{\kappa} \rangle \nabla_{s}^{i_{\nu}}\vec{\kappa} $$
for $j \in \{1, \dots, \nu-1\}$ or 
\begin{equation}\label{bel1}
 \langle \nabla_{s}^{i_{1}}\vec{\kappa}, \nabla_{s}^{i_{2}}\vec{\kappa} \rangle \dots 
\langle \nabla_{s}^{i_{\nu-2}} \vec{\kappa}, \nabla_{s}^{i_{\nu-1}}\vec{\kappa} \rangle \nabla_{t} \nabla_{s}^{i_{\nu}}\vec{\kappa}
\end{equation}
with as before $i_1+ \dots +i_{\nu}=\mu$ and $\max \{ i_{k} \} \leq d$. If $i_{j}=0$, by \eqref{e} 
\begin{align}\nonumber
\nabla_{t} \nabla_s^{i_j} \vec{\kappa}= \nabla_{t} \vec{\kappa} & = -\nabla_s^4 \vec{\kappa} - \frac12 |\vec{\kappa}|^2 \nabla_s^2 \vec{\kappa}- 2 \langle \vec{\kappa}, \nabla_s \vec{\kappa} \rangle \nabla_s \vec{\kappa} - |\nabla_s \vec{\kappa}|^2 \vec{\kappa} \\ \nonumber
& \quad - 2 \langle \vec{\kappa}, \nabla_s^2 \vec{\kappa} \rangle \vec{\kappa} + \lambda \nabla_s^2 \vec{\kappa} - \frac12 |\vec{\kappa}|^4 \vec{\kappa} + \lambda |\vec{\kappa}|^2 \vec{\kappa} \\ \label{ee}
& = \sum_{\substack{[[a,b]] \leq [[4,1]]\\c\leq 4\\b\in [1,5], odd}} P^{a,c}_{b} (\vec{\kappa}) 
+ \lambda \sum_{\substack{[[a,b]] \leq [[2,1]]\\c\leq 2\\b\in [1,3], odd }} P^{a,c}_{b} (\vec{\kappa}) \, ,
\end{align}  
while if $i_{j} \geq 1$ we find using \eqref{E1}, \eqref{ee}
\begin{align*}
\nabla_t \nabla_{s}^{i_{j}} \vec{\kappa} & = \nabla_{s}^{i_{j}} \nabla_t  \vec{\kappa} + \sum_{\substack{[[a,b]] \leq [[i_j+2,3]]\\c\leq 2+i_1\\b \in [3,5], odd}} P^{a,c}_{b} (\vec{\kappa})  + \lambda
\sum_{\substack{[[a,b]] \leq [[i_j,3]]\\ c \leq i_1\\b=3}} P^{a,c}_{b} (\vec{\kappa}) \\
& =  \sum_{\substack{[[a,b]] \leq [[4+i_{j},1]]\\c\leq 4+i_{j}\\b\in [1,5], odd}} P^{a,c}_{b} (\vec{\kappa}) 
+ \lambda \sum_{\substack{[[a,b]] \leq [[2+i_{j},1]]\\c\leq 2+i_{j}\\b\in [1,3], odd }} P^{a,c}_{b} (\vec{\kappa}) \, ,
\end{align*}
and this formula is valid also for $i_{j}=0$. It follows that
\begin{align*}
&  \langle \nabla_{s}^{i_{1}}\vec{\kappa}, \nabla_{s}^{i_{2}}\vec{\kappa} \rangle \ldots \langle \nabla_t \nabla_{s}^{i_{j}} \vec{\kappa}, .. \rangle \dots 
\langle \nabla_{s}^{i_{\nu-2}} \vec{\kappa}, \nabla_{s}^{i_{\nu-1}}\vec{\kappa} \rangle \nabla_{s}^{i_{\nu}}\vec{\kappa} \\
& =  \sum_{\substack{[[a,b]] \leq [[\mu+4,\nu]]\\c\leq 4+d\\b \in [\nu,\nu+4], odd}} P^{a,c}_{b} (\vec{\kappa})  + \lambda
\sum_{\substack{[[a,b]] \leq [[\mu+2,\nu]]\\ c \leq d+2\\b \in [\nu,2+\nu]}} P^{a,c}_{b} (\vec{\kappa}) 
\end{align*}
for any $ j \in \{1, \dots, \nu -1\}$ and the same formula holds for the term in \eqref{bel1}. We get
\begin{align} 
\nabla_{t} P^{\mu,d}_{\nu} (\vec{\kappa}) &  = \sum_{\substack{[[a,b]] \leq [[\mu+4,\nu]]\\c\leq 4+d\\b \in [\nu,\nu+4], odd }} P^{a,c}_{b} (\vec{\kappa}) +\lambda \sum_{\substack{[[a,b]] \leq [[\mu+2,\nu]]\\c\leq 2+d\\b \in [\nu,2+\nu] }} P^{a,c}_{b} (\vec{\kappa})  \, \label{riemcen1},
\end{align}
that is \eqref{E2} for $m=1$.

Assuming that the claim holds true for some $m\geq 1$ and any $\nu \in \mathbb{N}$, $\nu$ odd, $\mu, d \in \mathbb{N}_0$ we find by \eqref{riemcen1}
\begin{align*}
& \nabla_{t}^{m+1} P^{\mu,d}_{\nu} (\vec{\kappa}) \\
& = \sum_{i=0}^{m} \lambda^{i} \sum_{\substack{[[a,b]] \leq [[4m + \mu -2 i,\nu]]\\c\leq 4 m -2i + d\\b \in [\nu,\nu+4m-2i], odd}}  \nabla_{t}  P^{a,c}_{b} (\vec{\kappa}) \\
& = \sum_{i=0}^{m} \lambda^{i} \sum_{\substack{[[a,b]] \leq [[4m + \mu -2 i,\nu]]\\c\leq 4 m -2i + d\\b \in [\nu, \nu+4m-2i], odd}}   \Big(\sum_{\substack{[[ \alpha, \beta]] \leq [[a+4,b]]\\ \gamma \leq 4+ c \\ \beta \in [b,b+4], odd }} P^{\alpha,\gamma}_{\beta} (\vec{\kappa}) +\lambda \sum_{\substack{[[\alpha,\beta]] \leq [[a+2,b]]\\ \gamma\leq 2+c\\\beta \in [b,2+b] }} P^{\alpha,\gamma}_{\beta} (\vec{\kappa})  \Big) \\
& = \sum_{i=0}^{m} \lambda^{i} \hspace{-.2cm} \sum_{\substack{[[a,b]] \leq [[4 (m+1) + \mu -2 i,\nu]]\\c\leq 4 (m+1) -2i + d\\b \in [\nu,\nu+4(m+1)-2i], odd}} \hspace{-.2cm}  P^{a,c}_{b} (\vec{\kappa}) 
+ \sum_{i=0}^{m} \lambda^{i+1}  \hspace{-.2cm} \sum_{\substack{[[a,b]] \leq
    [[4 (m+1) + \mu -2 (i+1),\nu]]\\c\leq 4 (m+1) -2(i+1) + d\\b \in [\nu,
    \nu+4(m+1)-2(i+1)], odd}}  \hspace{-.2cm} P^{a,c}_{b} (\vec{\kappa})  \, .
\end{align*}
The claim follows.
\end{proof}

%%%%%%%%%%%%%%%%%%%%%%%%%%%%%%%%%%%%%%%%%E2sum

\begin{proof}[Proof of \eqref{E2sum} in Lemma \ref{lemLin}]
Indeed, formula \eqref{riemcen1} implies that
\begin{align*}
& \nabla_{t} \sum_{\substack{[[a,b]]\leq [[A,B]]\\c\leq C\\b\in [N,M], odd}}P^{a,c}_{b} (\vec{\kappa}) \\
& = \sum_{\substack{[[a,b]] \leq [[A,B]]\\c\leq C\\b\in[N,M],odd}} \Big( \sum_{\substack{[[\alpha,\beta]] \leq [[a+4,b]]\\\gamma\leq 4+c \\\beta \in [b,b+4], odd }} P^{\alpha,\gamma}_{\beta} (\vec{\kappa}) +\lambda \sum_{\substack{[[\alpha,\beta]] \leq [[a+2,b]]\\\gamma\leq 2+c\\\beta \in [b,2+b] }} P^{\alpha,\gamma}_{\beta} (\vec{\kappa})  \Big)\\
& =  \sum_{\substack{[[a,b]] \leq [[A+4,B]]\\c\leq C+4\\b\in[N,M+4],odd}}  P^{a,c}_{b} (\vec{\kappa}) +\lambda  \sum_{\substack{[[a,b]] \leq [[A+2,B]]\\c\leq C+2\\b\in[N,M+2],odd}} P^{a,c}_{b} (\vec{\kappa}) \, .
\end{align*}
\end{proof}

%%%%%%%%%%%%%%%%%%%%%%%%%%%%%%%%%%%%%%%%E3

\begin{proof}[Proof of \eqref{E3} in Lemma \ref{lemLin}]
Equation \eqref{ee} gives us that
\begin{equation}\label{1mai1}
\nabla_{t} \vec{\kappa} = - \nabla_s^{4} \vec{\kappa} + \sum_{\substack{[[a,b]] \leq [[2,3]]\\c\leq 2\\b\in [3,5], odd}} P^{a,c}_{b} (\vec{\kappa}) + \lambda \sum_{\substack{[[a,b]] \leq [[2,1]]\\c\leq 2\\b\in [1,3], odd}} P^{a,c}_{b} (\vec{\kappa})  \, ,
\end{equation}
that is the claim for $m=1$. Assuming that \eqref{E3} holds for some $m \geq 1$, we get using \eqref{E1}, \eqref{E2sum}, \eqref{1mai1}
\begin{align*}
& \nabla^{m+1}_{t} \vec{\kappa} \\
& \quad = (-1)^m  \nabla_t \nabla_s^{4m}
\vec{\kappa}\\
& \qquad +  \nabla_t \Big(  \sum_{\substack{[[a,b]] \leq [[4m -2,3]]\\c\leq 4 m -2\\b\in[3,4m+1],odd}} P^{a,c}_{b} (\vec{\kappa}) + \sum_{i=1}^{m} \lambda^{i} \sum_{\substack{[[a,b]] \leq [[4m -2i,1]]\\c\leq 4 m-2i \\b\in[1,4m+1-2i],odd}} P^{a,c}_{b} (\vec{\kappa}) \Big)\\
& \quad = (-1)^m \Big( \nabla_s^{4m} \nabla_t \vec{\kappa} + \sum_{\substack{[[a,b]] \leq [[4m+2,3]]\\c\leq 2+4m\\b \in [3,5], odd}} P^{a,c}_{b} (\vec{\kappa})  + \lambda
\sum_{\substack{[[a,b]] \leq [[4m,3]]\\ c \leq 4m\\b=3}} P^{a,c}_{b} (\vec{\kappa})  \Big) \\
& \qquad  + \sum_{\substack{[[a,b]] \leq [[4(m+1)-2,3]]\\c\leq 4(m+1)-2\\b \in [3,4(m+1)+1], odd}} P^{a,c}_{b} (\vec{\kappa})  + \lambda
\sum_{\substack{[[a,b]] \leq [[4m,3]]\\ c \leq 4m\\b\in [3, 4m+3]}} P^{a,c}_{b} (\vec{\kappa}) \\
& \qquad + \sum_{i=1}^{m} \lambda^{i} \hspace{-.2cm} \sum_{\substack{[[a,b]] \leq [[4(m+1)
   -2i,1]]\\c\leq 4 (m+1)-2i \\b\in[1,4(m+1)+1-2i],odd}} \hspace{-.2cm} P^{a,c}_{b}
(\vec{\kappa}) + \sum_{i=1}^{m} \lambda^{i+1}  \hspace{-.2cm} \sum_{\substack{[[a,b]] \leq [[4(m+1) -2(i+1),1]]\\c\leq 4 (m+1)-2(i+1) \\b\in[1,4(m+1)+1-2(i+1)],odd}}\hspace{-.2cm} P^{a,c}_{b} (\vec{\kappa})\\
& \quad = (-1)^m  \nabla_s^{4m} \Big( - \nabla_s^{4} \vec{\kappa} + \sum_{\substack{[[a,b]] \leq [[2,3]]\\c\leq 2\\b\in [3,5], odd}} P^{a,c}_{b} (\vec{\kappa}) + \lambda \sum_{\substack{[[a,b]] \leq [[2,1]]\\c\leq 2\\b\in [1,3], odd}} P^{a,c}_{b} (\vec{\kappa}) \Big)\\
& \qquad + \sum_{\substack{[[a,b]] \leq [[4(m+1)-2,3]]\\c\leq 4(m+1)-2\\b \in [3,4(m+1)+1], odd}} P^{a,c}_{b} (\vec{\kappa}) + \sum_{i=1}^{m+1} \lambda^{i} \sum_{\substack{[[a,b]] \leq [[4(m+1) -2i,1]]\\c\leq 4 (m+1)-2i \\b\in[1,4(m+1)+1-2i],odd}} P^{a,c}_{b} (\vec{\kappa})
\end{align*}
from which the claim follows directly.
\end{proof}

%%%%%%%%%%%%%%%%%%%%%%%%%%%%%%%%%%%%%%%E4

\begin{proof}[Proof of \eqref{E4} in Lemma \ref{lemLin}]
We prove the formula by induction on $m$. For simplicity of notation let $\vec{\xi}^{l}=\nabla_s^{l}\vec{\kappa}$. Formula \eqref{E4} with $m=1$ follows  for any $k \in \mathbb{N}$, $l \in \mathbb{N}_0$ from \eqref{E1}.
Notice that this formula is ``weaker'' than \eqref{E1}. Assuming that the claim holds for some $m \geq 1$ and for any $l \in \mathbb{N}_0$ and $k\in \mathbb{N}$, we find using \eqref{E3}, the induction assumption (for $m$ and $m=1$), \eqref{E2sum}, \eqref{E1}, \eqref{1mai1} 
{\allowdisplaybreaks
\begin{align*}
 & \nabla_t^{m+1} \nabla_s^{k} \vec{\xi}^{l}  = \nabla_t \left(\nabla_t^{m} \nabla_s^{k+l} \vec{\kappa} \right) =\\ 
& = \nabla_{t} \Big( \nabla_s^{k+l} \nabla_{t}^{m} \vec{\kappa} +
\hspace{-.2cm} \sum_{\substack{[[a,b]] \leq [[4m+k+l -2,3]]\\c\leq 4 m
    +l+k-2\\b\in[3,4m+1],odd}} \hspace{-.2cm} P^{a,c}_{b} (\vec{\kappa}) +
\sum_{i=1}^{m} \lambda^{i} \hspace{-.2cm} \sum_{\substack{[[a,b]] \leq
    [[4m+k+l-2i,1]]\\c\leq 4 m +l+k-2i\\b\in[1,4m-2i+1],odd}} \hspace{-.2cm} P^{a,c}_{b} (\vec{\kappa})  \Big) \\
& = \nabla_{t}  \nabla_s^{k+l} \Big((-1)^m \nabla_s^{4m} \vec{\kappa} +
\hspace{-.2cm} \sum_{\substack{[[a,b]] \leq [[4m -2,3]]\\c\leq 4 m
    -2\\b\in[3,4m+1],odd}} \hspace{-.2cm} P^{a,c}_{b} (\vec{\kappa}) + \sum_{i=1}^{m} \lambda^{i}\hspace{-.2cm} \sum_{\substack{[[a,b]] \leq [[4m -2i,1]]\\c\leq 4 m-2i \\b\in[1,4m+1-2i],odd}}\hspace{-.2cm} P^{a,c}_{b} (\vec{\kappa}) \Big)\\
& \quad +\nabla_{t} \Big(  \hspace{-.2cm} \sum_{\substack{[[a,b]] \leq [[4m+k+l -2,3]]\\c\leq 4 m +l+k-2\\b\in[3,4m+1],odd}}\hspace{-.2cm} P^{a,c}_{b} (\vec{\kappa}) + \sum_{i=1}^{m} \lambda^{i}\hspace{-.2cm} \sum_{\substack{[[a,b]] \leq [[4m+k+l-2i,1]]\\c\leq 4 m +l+k-2i\\b\in[1,4m-2i+1],odd}}\hspace{-.2cm} P^{a,c}_{b} (\vec{\kappa})  \Big) \\
& = (-1)^m \nabla_{t} \nabla_s^{4m+k+l} \vec{\kappa} \\
& \quad +\nabla_{t} \Big( \hspace{-.2cm}  \sum_{\substack{[[a,b]] \leq [[4m+k+l -2,3]]\\c\leq 4 m +l+k-2\\b\in[3,4m+1],odd}}\hspace{-.2cm} P^{a,c}_{b} (\vec{\kappa}) + \sum_{i=1}^{m} \lambda^{i}\hspace{-.2cm} \sum_{\substack{[[a,b]] \leq [[4m+k+l-2i,1]]\\c\leq 4 m +l+k-2i\\b\in[1,4m-2i+1],odd}}\hspace{-.2cm} P^{a,c}_{b} (\vec{\kappa})  \Big) \\
& = (-1)^m  \nabla_s^{4m+k+l}\nabla_{t} \vec{\kappa} +\hspace{-.2cm}  \sum_{\substack{[[a,b]] \leq [[4(m+1)+k+l-2,3]]\\c\leq 4 (m+1) +l+k-2\\b\in[3,5],odd}}\hspace{-.2cm} P^{a,c}_{b} (\vec{\kappa}) + \lambda\hspace{-.2cm} \sum_{\substack{[[a,b]] \leq [[4m+k+l,3]]\\c\leq 4 m+l+k\\b=3,odd}}\hspace{-.2cm} P^{a,c}_{b} (\vec{\kappa})  \\
& \quad + \hspace{-.2cm} \sum_{\substack{[[a,b]] \leq [[4(m+1)+k+l
    -2,3]]\\c\leq 4 (m+1) +l+k-2\\b\in[3,4(m+1)+1],odd}} \hspace{-.2cm}
P^{a,c}_{b} (\vec{\kappa}) + \lambda \hspace{-.2cm} \sum_{\substack{[[a,b]] \leq [[4(m+1)+k+l -4,3]]\\c\leq 4 (m+1) +l+k-4\\b\in[3,4(m+1)+1-2],odd}}\hspace{-.2cm} P^{a,c}_{b} (\vec{\kappa})  \\
& \quad +  \sum_{i=1}^{m} \lambda^{i}\hspace{-.2cm} \sum_{\substack{[[a,b]] \leq [[4(m+1)+k+l-2i,1]]\\c\leq 4 (m+1) +l+k-2i\\b\in[1,4(m+1)-2i+1],odd}}\hspace{-.2cm} P^{a,c}_{b} (\vec{\kappa})  +  \sum_{i=1}^{m} \lambda^{i+1}\hspace{-.2cm} \sum_{\substack{[[a,b]] \leq [[4(m+1)+k+l-2(i+1),1]]\\c\leq 4 (m+1) +l+k-2(i+1)\\b\in[1,4(m+1)-2(i+1)+1],odd}}\hspace{-.2cm} P^{a,c}_{b} (\vec{\kappa}) \\
 & = (-1)^m  \nabla_s^{4m+k+l}( - \nabla_s^{4} \vec{\kappa} + \sum_{\substack{[[a,b]] \leq [[2,3]]\\c\leq 2\\b\in [3,5], odd}} P^{a,c}_{b} (\vec{\kappa}) + \lambda \sum_{\substack{[[a,b]] \leq [[2,1]]\\c\leq 2\\b\in [1,3], odd}} P^{a,c}_{b} (\vec{\kappa}))\\ 
& \quad +  \sum_{\substack{[[a,b]] \leq [[4(m+1)+k+l -2,3]]\\c\leq 4 (m+1) +l+k-2\\b\in[3,4(m+1)+1],odd}} P^{a,c}_{b} (\vec{\kappa})  +  \sum_{i=1}^{m+1} \lambda^{i} \sum_{\substack{[[a,b]] \leq [[4(m+1)+k+l-2i,1]]\\c\leq 4 (m+1) +l+k-2i\\b\in[1,4(m+1)-2i+1],odd}} P^{a,c}_{b} (\vec{\kappa})
\end{align*}
}
that yields
\begin{align}\nonumber
 \nabla_t^{m+1}  \nabla_s^{k}  \vec{\xi}^{l}  &  = (-1)^{m+1}  \nabla_s^{4(m+1)+k+l} \vec{\kappa} +\hspace{-.2cm}  \sum_{\substack{[[a,b]] \leq [[4(m+1)+k+l -2,3]]\\c\leq 4 (m+1) +l+k-2\\b\in[3,4(m+1)+1],odd}}\hspace{-.2cm} P^{a,c}_{b} (\vec{\kappa}) \\ \label{1mai2}
& \quad   +  \sum_{i=1}^{m+1} \lambda^{i}\hspace{-.2cm} \sum_{\substack{[[a,b]] \leq
    [[4(m+1)+k+l-2i,1]]\\c\leq 4 (m+1) +l+k-2i\\b\in[1,4(m+1)-2i+1],odd}}\hspace{-.2cm}
P^{a,c}_{b} (\vec{\kappa}) \, .
\end{align}
On the other hand, using \eqref{1mai2}
\begin{align*}
 \nabla_s^{k} \nabla_t^{m+1} \vec{\xi}^{l} & =  \nabla_s^{k} \nabla_t^{m+1} \nabla_s^{l} \vec{\kappa} \\
& = \nabla_s^{k} \Big(  (-1)^{m+1}  \nabla_s^{4(m+1)+l} \vec{\kappa} +  \sum_{\substack{[[a,b]] \leq [[4(m+1)+l -2,3]]\\c\leq 4 (m+1) +l-2\\b\in[3,4(m+1)+1],odd}} P^{a,c}_{b} (\vec{\kappa})  \\
& \qquad +  \sum_{i=1}^{m+1} \lambda^{i} \sum_{\substack{[[a,b]] \leq [[4(m+1)+l-2i,1]]\\c\leq 4 (m+1) +l-2i\\b\in[1,4(m+1)-2i+1],odd}} P^{a,c}_{b} (\vec{\kappa})   \Big) \\
& = (-1)^{m+1}  \nabla_s^{4(m+1)+k+l} \vec{\kappa} +  \sum_{\substack{[[a,b]] \leq [[4(m+1)+k+l -2,3]]\\c\leq 4 (m+1)+k +l-2\\b\in[3,4(m+1)+1],odd}} P^{a,c}_{b} (\vec{\kappa})  \\
& \qquad  +  \sum_{i=1}^{m+1} \lambda^{i} \sum_{\substack{[[a,b]] \leq [[4(m+1)+l+k-2i,1]]\\c\leq 4 (m+1)+k +l-2i\\b\in[1,4(m+1)-2i+1],odd}} P^{a,c}_{b} (\vec{\kappa}) \, .
\end{align*}
The claim follows combining the formula just obtained with \eqref{1mai2}.
\end{proof}

%%%%%%%%%%%%%%%%%%%%%%%%%%%%%%%%%%%%%%%%%%%%%%%%%%E5

\begin{proof}[Proof of \eqref{E5} in Lemma \ref{lemLin}]
Formula \eqref{E5} with $m=1$ is the equation that $f$ satisfies. Assuming that \eqref{E5} holds for some $m \geq 1$ we find with \eqref{E4}, \eqref{E2sum} and \eqref{1mai1}
{\allowdisplaybreaks 
\begin{align*}
& \nabla_t^{m+1} f \\
& \quad = \nabla_t \Big(  (-1)^{m} \nabla_s^{4m-2} \vec{\kappa} +
\hspace{-.2cm} \sum_{\substack{[[a,b]] \leq [[4m -4,3]]\\c\leq 4 m -4\\b\in[3, 4m-1],odd}}\hspace{-.3cm} P^{a,c}_{b} (\vec{\kappa}) + \sum_{i=1}^{m} \lambda^{i}\hspace{-.2cm} \sum_{\substack{[[a,b]] \leq [[4m -2-2i,1]]\\c\leq 4 m -2-2i\\b\in[1, 4m-1-2i],odd}}\hspace{-.3cm} P^{a,c}_{b} (\vec{\kappa})  \Big) \\
& \quad = (-1)^m \nabla_s^{4m-2} \nabla_t \vec{\kappa} +\hspace{-.2cm} \sum_{\substack{[[a,b]] \leq [[4m,3]]\\c\leq 4 m\\b\in[3, 5],odd}}\hspace{-.2cm} P^{a,c}_{b} (\vec{\kappa}) + \lambda\hspace{-.2cm} \sum_{\substack{[[a,b]] \leq [[4m,1]]\\c\leq 4 m \\b\in[1, 3],odd}}\hspace{-.2cm} P^{a,c}_{b} (\vec{\kappa}) \\
& \qquad +\hspace{-.2cm} \sum_{\substack{[[a,b]] \leq [[4m,3]]\\c\leq 4 m\\b\in[3, 4(m+1)-1],odd}}\hspace{-.2cm} P^{a,c}_{b} (\vec{\kappa}) + \lambda\hspace{-.2cm} \sum_{\substack{[[a,b]] \leq [[4m-2,3]]\\c\leq 4 m-2 \\b\in[3,4m+1],odd}}\hspace{-.2cm} P^{a,c}_{b} (\vec{\kappa})\\
& \qquad + \sum_{i=1}^{m} \lambda^{i}\hspace{-.2cm} \sum_{\substack{[[a,b]] \leq [[4(m+1) -2-2i,1]]\\c\leq 4 (m+1) -2-2i\\b\in[1, 4(m+1)-1-2i],odd}}\hspace{-.3cm} P^{a,c}_{b} (\vec{\kappa})+ \sum_{i=1}^{m} \lambda^{i+1}\hspace{-.2cm} \sum_{\substack{[[a,b]] \leq [[4(m+1)-2-2(i+1),1]]\\c\leq 4 (m+1) -2-2(i+1)\\b\in[1, 4(m+1)-1-2(i+1)],odd}}\hspace{-.3cm} P^{a,c}_{b} (\vec{\kappa}) \\
& \quad = (-1)^m \nabla_s^{4m-2} \Big(-\nabla_s^4 \vec{\kappa} +\hspace{-.2cm} \sum_{\substack{[[a,b]]\leq[[2,3]]\\c\leq 2\\b\in[3,5], odd}}\hspace{-.2cm} P_{b}^{a,c}(\vec{\kappa})  + \lambda\hspace{-.2cm} \sum_{\substack{[[a,b]]\leq[[2,1]]\\c\leq 2\\b\in[1,3],odd}}\hspace{-.2cm} P_{b}^{a,c}(\vec{\kappa}) \Big) \\
& \qquad  +\hspace{-.2cm} \sum_{\substack{[[a,b]] \leq [[4m,3]]\\c\leq 4 m\\b\in[3, 4(m+1)-1],odd}}\hspace{-.2cm} P^{a,c}_{b} (\vec{\kappa}) + \sum_{i=1}^{m+1} \lambda^{i}\hspace{-.2cm} \sum_{\substack{[[a,b]] \leq [[4(m+1) -2-2i,1]]\\c\leq 4 (m+1) -2-2i\\b\in[1, 4(m+1)-1-2i],odd}}\hspace{-.2cm} P^{a,c}_{b} (\vec{\kappa})\\
& \quad = (-1)^{m+1} \nabla_s^{4(m+1)-2} \vec{\kappa}  +\hspace{-.2cm}
\sum_{\substack{[[a,b]] \leq [[4m,3]]\\c\leq 4 m\\b\in[3,
    4(m+1)-1],odd}}\hspace{-.2cm} P^{a,c}_{b} (\vec{\kappa}) \\
& \qquad + \sum_{i=1}^{m+1} \lambda^{i}\hspace{-.2cm} \sum_{\substack{[[a,b]] \leq [[4(m+1) -2-2i,1]]\\c\leq 4 (m+1) -2-2i\\b\in[1, 4(m+1)-1-2i],odd}}\hspace{-.2cm} P^{a,c}_{b} (\vec{\kappa}) \, .
\end{align*}
}
\end{proof}

%%%%%%%%%%%%%%%%%%%%%%%%%%%%%%%%%%%%%%%%%%%%%%%%%%%%%%%%%%%%%%%%%%%%%%%%%%
%%%%%%%%%%%%% Appendix Interpolation %%%%%%%%%%%%%%%%%%%%%%%%%%%%%%%%%%%%%
%%%%%%%%%%%%%%%%%%%%%%%%%%%%%%%%%%%%%%%%%%%%%%%%%%%%%%%%%%%%%%%%%%%%%%%%%%

\section{Proof of Lemma \ref{leminter}}
\label{AppendixInter}
In the following we give some useful facts in order to prove Lemma \ref{leminter}. We use the notation presented previously and denote by $c$  a positive constant that may change from line to line.

Although the next result is well known, we report the exact statement, since it is used in several important steps and since 
it shows explicitly on what the constant depends.
\begin{lemma}\label{lemBGH}
Let $J \subset \R$ be a bounded open interval and $g:J \rightarrow
\mathbb{R}^n$, $g(x)$, be a sufficiently smooth function. Then 
\begin{equation*}
\| g\|_{C^{0}(\bar{J})} \leq c(n) \| \partial_x g\|_{L^1(J)} + \frac{c(n)}{|J|} \| g\|_{L^1(J)} \, .
\end{equation*}
If $n=1$, then $c(n)=1$.
\end{lemma}
\begin{proof}
Writing $g = (g^1, \dots, g^{n})$, $g^{i}:J \rightarrow \mathbb{R}$ 
for $i \in\{1,\dots, n\} $, 
the claim follows from \cite[Thm.~2.2]{BGH}.
\end{proof}

\begin{lemma}
\label{lemma6.1}
Let $J \subset \R$ be a bounded open interval and $g: J \to \R$, $g=g(x)$, be as regular as required. We have that for any $\epsilon \in (0,1)$
\begin{align}
\label{6.1}
\| g \| _{C(\bar{J})} &\leq \epsilon \| g_{x} \|_{L^{2}(J)}+\frac{c}{\epsilon} \| g \|_{L^{2}(J)} \, ,\\
\label{6.2}
\|g_{x} \|_{L^{2}(J)} & \leq \epsilon \| g \|_{W^{2,2}(J)} + \frac{1}{ \epsilon} \| g \|_{L^{2}(J)}  \, ,\\
\label{6.3}
\| g \| _{C(\bar{J})} &\leq \epsilon \| g \|_{W^{2,2}(J)} + \frac{c}{ \epsilon} \| g\|_{L^{2}(J)} \, ,
\end{align} 
with $c=c(J)$.
\end{lemma}
\begin{proof} Using Lemma \ref{lemBGH} we get for $x \in \bar{J}$
\begin{align*}
%|g(x)|^{2} \leq
 \| g^{2} \| _{C(\bar{J})} &  \leq \| \partial_{x} (g^{2})
\|_{L^{1}(J)} + \frac{1}{|J|}  \| g^{2} \|_{L^{1}(J)} 
& \leq \frac{1}{|J|}  \int_{J}|g|^{2} dx + 2 \int_{J}|g||g_{x}| dx
\end{align*}
and equation \eqref{6.1} follows using Young's inequality and $\epsilon <1$.
The second inequality is shown in \cite[Thm. 5.2]{Adams}, the third one follows from \eqref{6.1} and~\eqref{6.2}.
\end{proof}

Recall that as usual $f: [0,1] \rightarrow \mathbb{R}^n$ is a smooth regular curve, $I=(0,1)$. 
\begin{lemma}
\label{lemmafacile}
%For a generic vector field $\psi$ we have that $|\partial_{s} \psi|\geq |\partial_{s }|\psi||$.
For  normal vector field $\vec{\phi}$ we have that
$$|\partial_{s} |\vec{\phi}||\leq|\nabla_{s }\vec{\phi}| \mbox{ almost everywhere}.$$
\end{lemma}
\begin{proof}
For $\vec{\phi} \neq 0$ the claim easily follows from
$\partial_{s} |\vec{\phi}|= \langle \frac{\vec{\phi}}{|\vec{\phi}|}, \partial_{s } \vec{\phi} \rangle = \langle \frac{\vec{\phi}}{|\vec{\phi}|}, \nabla_{s } \vec{\phi} \rangle  $.
Otherwise consider for a positive $\delta$ the regularisation $\sqrt{\delta^{2}+ \langle \vec{\phi}, \vec{\phi} \rangle}$ and take the limit $\delta \searrow 0$ in the definition of weak derivative. 
\end{proof}
For a normal vector field $\vec{\phi}:\bar{I} \to \R^{n}$ recall %the notation 
$\| \vec{\phi}\|_{k,p} = \sum_{i=0}^{k} \| \nabla_s^{i} \vec{\phi}\|_{p}$ with
\begin{equation*} 
\| \nabla_s^{i} \vec{\phi} \|_{p} 
%= \mathcal{L}[f]^{i+1-1/p} \Big( \int_I |\nabla_s^{i} \vec{\phi}|^p ds \Big)^{1/p} 
= \mathcal{L}[f]^{i+1-1/p} \| \nabla_s^{i} \vec{\phi}\|_{L^{p}}\, ,
\end{equation*}
and keep in mind that these norms are scale invariant when $\vec{\phi}=\vec{\kappa}$ (otherwise the transformation  $f \mapsto \alpha f$ for $\alpha>0$ multiplies the norm by a factor $\alpha$).

%%%%%%%%%%%
\begin{lemma}
\label{indstep1}
Let $\vec{\phi}$ be a normal vector field. Then for any $\epsilon \in (0,1)$
$$\| \nabla_{s} \vec{\phi}\|_{2} \leq c\left(  \epsilon \| \vec{\phi}\|_{2,2}+ \frac{1}{\epsilon} \| \vec{\phi} \|_{2}\right). $$
\end{lemma}
\begin{proof}
Because of the scaling properties of the norm we may assume that $\mathcal{L}[f]=1$ so that $\| \cdot \|_{p}= \| \cdot \|_{L^{p}}$.
Moreover we consider the curve reparametrized according to arc-length (so that $|f_{x}|=1$, $dx =ds$, and we can use for instance Lemma \ref{lemma6.1}. For simplicity we take the boundary points of the domain (of length one since $\mathcal{L}[f]=1$) to be the points 0 and 1).
Now consider
\begin{align*}
\|\nabla_{s} \vec{\phi}\|_{2}^{2} & =\int_{0}^1\langle \nabla_{s}
\vec{\phi}(s), \nabla_{s} \vec{\phi} (s)\rangle ds %\\
& =- \int_{0}^1\langle \nabla^{2}_{s} \vec{\phi},  \vec{\phi} \rangle ds +[\langle \nabla_{s} \vec{\phi},  \vec{\phi} \rangle]_{0}^{1}=: I +II.
\end{align*}
Obviously 
$$|I| \leq \frac{\epsilon^2}{2} \| \vec{\phi}\|_{2,2}^{2}+ \frac{1}{2\epsilon^2} \| \vec{\phi} \|_{2}^{2} \, .$$
Moreover, by \eqref{6.1} and Lemma \ref{lemmafacile}
\begin{align*}
|II| &\leq 2\|\, |\nabla_{s} \vec{\phi}|\, \|_{L^{\infty}} \| \,|\vec{\phi} |\, \|_{L^{\infty}} \\
& \leq \big(\epsilon_1  \|\partial_s |\nabla_{s} \vec{\phi}| \|_{2} + \frac{c}{\epsilon_1} \| |\nabla_{s} \vec{\phi}| \|_{2} \big) \big(\epsilon_2  \|\partial_s | \vec{\phi}| \|_{2} + \frac{c}{\epsilon_2} \| |\vec{\phi}| \|_{2} \big)\\
& \leq \epsilon_1 \epsilon_2 \|\vec{\phi} \|_{2,2} \|\nabla_s  \vec{\phi} \|_{2} + c \frac{\epsilon_2}{\epsilon_1}  \|\nabla_s  \vec{\phi} \|_{2}^2 + c \frac{\epsilon_1}{\epsilon_2} \|\vec{\phi} \|_{2,2} \|\vec{\phi}\|_2 + \frac{c}{\epsilon_1\epsilon_2} \|\nabla_s  \vec{\phi} \|_{2} \| \vec{\phi} \|_{2} \, .
\end{align*}
Choosing $\epsilon_2=\epsilon_1/4c$ and by Young's inequality 
\begin{align*}
|II| & \leq c \epsilon_1^4 \|\vec{\phi} \|^2_{2,2} + \frac12 \|\nabla_s  \vec{\phi} \|_{2}^2   + \frac{c}{\epsilon_1^4} \| \vec{\phi} \|_{2}^2 \, .
\end{align*}

Putting the estimates together, with $\epsilon_1^2=\epsilon$ we find
\begin{align*}
\frac12 \|\nabla_{s} \vec{\phi}\|_{2}^{2} \leq \frac{\epsilon^2}{2} \| \vec{\phi}\|_{2,2}^{2}+ \frac{1}{2\epsilon^2} \| \vec{\phi} \|_{2}^{2} + c \epsilon^2 \| \vec{\phi} \|_{2,2}^2 + \frac{c}{\epsilon^2}  \| \vec{\phi} \|_{2}^2 ,
\end{align*}
from which the claim follows directly.

Note that due to the rescaling procedure the constant $c$ does not depend on the length of the curve.
\end{proof}

\begin{lemma}
Let $\vec{\phi}$ be a normal vector field and $\epsilon \in (0,1)$. Then $\forall\,\,  k \geq 2$ and all $0<i<k$ we have
\begin{align}
\label{eq6.4}
\|\nabla_{s}^{i} \vec{\phi} \|_{2} \leq c \left( \epsilon \| \vec{\phi} \|_{k,2} +
\epsilon^{\frac{i}{i-k}} \| \vec{\phi} \|_{2} \right),
\end{align}
with $c=c(i,k)$. In particular it follows that
\begin{align}
\label{eq6.5}
\|\nabla_{s}^{i} \vec{\phi} \|_{2} \leq c \| \vec{\phi} \|_{k,2}^{\frac{i}{k}}
\| \vec{\phi} \|_{2}^{\frac{k-i}{k}}.
\end{align}
\end{lemma}
\begin{proof}
We may assume that $\mathcal{L}[f]=1$. Equation \eqref{eq6.5} follows from
\eqref{eq6.4} by choosing $\epsilon$ so that the two terms in the right-hand
side of \eqref{eq6.4} are equal, i.e. by imposing
$\epsilon \| \vec{\phi}\|_{k,2}=\epsilon^{\frac{i}{i-k}} \| \vec{\phi} \|_{2} $
from which we derive that
$$ \epsilon =\left(\frac{\| \vec{\phi}\|_{k,2}}{\| \vec{\phi}\|_{2}}\right)^{\frac{i-k}{k}} <1.$$
It remains to show the first claim. Lemma \ref{indstep1} gives \eqref{eq6.4} for the case $k=2$ and $i=1$. 
A suitable induction argument yields the stated result. 
More precisely: our \underline{induction assumption} (A)  can be stated as
\begin{align}
\label{(A)}
\exists \, \mathbf{k}>2 : \, \forall k \leq \mathbf{k}, \forall \, i: 0 <i < k
: \;  \|\nabla_{s}^{i} \vec{\phi} \|_{2}  \leq c \left( \epsilon \| \vec{\phi} \|_{k,2} +
\epsilon^{\frac{i}{i-k}} \| \vec{\phi} \|_{2} \right).
\;
\tag{A}
\end{align}
We prove the estimate for $k =\mathbf{k}+1$ and all $i: 0< i< \mathbf{k}+1$.

First of all note that \eqref{(A)} implies that
\begin{align}
\label{eq6.6}
\| \vec{\phi} \|_{k,2}  \leq 2 \| \nabla_s^k \vec{\phi} \|_{2} +
c \| \vec{\phi} \|_{2},  \qquad \text{ for all } k \leq \mathbf{k}.
\end{align}
(Indeed we have
 \begin{align*}
 \|\vec{\phi} \|_{k,2} &= \|\nabla_s^k \vec{\phi} \|_{2} + \sum_{i=1}^{k-1}\| \nabla_s^i\vec{\phi} \|_{2} + \|\vec{\phi} \|_{2} \\ &
 \leq \|\nabla_s^k \vec{\phi} \|_{2} + \sum_{i=1}^{k-1}
 c\left( \epsilon_i \| \vec{\phi} \|_{k,2} +
\epsilon_i^{\frac{i}{i-k}} \| \vec{\phi} \|_{2} \right) +\| \vec{\phi} \|_{2},
 \end{align*}
 and choosing $\epsilon_i$ such that $c \sum_{i=1}^{k-1}
  \epsilon_i =\frac{1}{2}$ we derive immediately \eqref{eq6.6}.)
  
To prove the induction step we distinguish two cases:

\smallskip
\underline{Induction step, case $i=\mathbf{k}$:}
Using  \eqref{(A)} we obtain
\begin{align*}
\| \nabla_s^{\mathbf{k}} \vec{\phi} \|_2 = \| \nabla_s^{\mathbf{k}-1} ( \nabla_s \vec{\phi} )\|_2 & \leq c( \epsilon_1 \| \nabla_s \vec{\phi} \|_{\mathbf{k},2} +
\epsilon_1^{1-\mathbf{k}} \| \nabla_s \vec{\phi} \|_{2} ) \\
& \leq c ( \epsilon_1 \|  \vec{\phi} \|_{\mathbf{k}+1,2} +
\epsilon_1^{1-\mathbf{k}} \| \nabla_s \vec{\phi} \|_{2} ) 
\end{align*}
for  $\epsilon_1 \in (0,1)$.
Next by $\eqref{(A)}$  and  \eqref{eq6.6} we get
\begin{align*}
\| \nabla_s \vec{\phi} \|_2 \leq c (\epsilon_2 \| \vec{\phi} \|_{\mathbf{k},2} +
\epsilon_2^{\frac{1}{1-\mathbf{k}}} \|\vec{\phi} \|_{2} )
\leq c \epsilon_2 (2 \| \nabla_s^{\mathbf{k}} \vec{\phi} \|_{2} +
c \| \vec{\phi} \|_{2})  +c \epsilon_2^{\frac{1}{1-\mathbf{k}}} \|\vec{\phi} \|_{2}.
\end{align*}
Putting the above two estimates together we find
\begin{align*}
\| \nabla_s^{\mathbf{k}} \vec{\phi} \|_2  \leq c \epsilon_1 \|  \vec{\phi} \|_{\mathbf{k}+1,2} + c\epsilon_1^{1-\mathbf{k}}  \epsilon_2 \|  \nabla_s^{\mathbf{k}}\vec{\phi} \|_{2} + c\epsilon_1^{1-\mathbf{k}}(\epsilon_2  +\epsilon_2^{\frac{1}{1-\mathbf{k}}}) \| \vec{\phi} \|_2.
\end{align*}
Choosing $\epsilon_2 <1$ so that $ c\epsilon_1^{1-\mathbf{k}}  \epsilon_2 =\frac{1}{2}$ we infer
\begin{align*}
\| \nabla_s^{\mathbf{k}} \vec{\phi} \|_2  \leq c \epsilon_1 \|  \vec{\phi} \|_{\mathbf{k}+1,2} +  c(1  +\frac{1}{\epsilon_1^{\mathbf{k}}}) \| \vec{\phi} \|_2.
\end{align*}
Using now the fact that $1 \leq \frac{1}{\epsilon_1}$ (hence $ 1 \leq \frac{1}{\epsilon_1^{\mathbf{k}}}$) we obtain
\begin{align*}
\| \nabla_s^{\mathbf{k}} \vec{\phi} \|_2  \leq c (\epsilon_1 \|  \vec{\phi} \|_{\mathbf{k}+1,2} +  \frac{1}{\epsilon_1^{\mathbf{k}}} \| \vec{\phi} \|_2)
\end{align*}
and the claim follows.

\smallskip
\underline{Induction step, case $0<i<\mathbf{k}$:}  
Using twice \eqref{(A)} we get for $\epsilon, \epsilon_{j} \in (0,1)$
\begin{align*}
\|\nabla_{s}^{i} \vec{\phi} \|_{2} & \leq c \left( \epsilon \| \vec{\phi} \|_{\mathbf{k},2} +
\epsilon^{\frac{i}{i-\mathbf{k}}} \| \vec{\phi} \|_{2} \right)\\
&=c \left(  \epsilon \| \nabla_s^{\mathbf{k}} \vec{\phi} \|_2
+ \epsilon \sum_{j=1}^{\mathbf{k}-1} \| \nabla_s^{j} \vec{\phi} \|_2
+ \epsilon^{\frac{i}{i-\mathbf{k}}} \| \vec{\phi} \|_{2} + \epsilon 
\| \vec{\phi} \|_{2} \right)\\
& \leq c \left(  \epsilon \| \nabla_s^{\mathbf{k}} \vec{\phi} \|_2
+ c \epsilon \sum_{j=1}^{\mathbf{k}-1} ( \epsilon_j \| \vec{\phi}\|_{\mathbf{k},2} + \epsilon_j^{\frac{j}{j-\mathbf{k}}} \| \vec{\phi}\|_2)
+ \epsilon^{\frac{i}{i-\mathbf{k}}} \| \vec{\phi} \|_{2}
\right) \, ,%\qquad \text{ since } \epsilon <\epsilon^{\frac{i}{i-\mathbf{k}}} \\
% & \leq c\left(  \epsilon \| \nabla_s^{\mathbf{k}} \vec{\phi} \|_2 +2 c
 % \epsilon  \| \nabla_s^{\mathbf{k}} \vec{\phi} \|_2 \sum_{j=1}^{\mathbf{k}-1}
 % \epsilon_j+ c\epsilon  \| \vec{\phi} \ |_2
 % \sum_{j=1}^{\mathbf{k}-1}\epsilon_j + c \epsilon \| \vec{\phi} \|_2 \sum_{j=1}^{\mathbf{k}-1}\epsilon_j^{\frac{j}{j-\mathbf{k}}}   +\epsilon^{\frac{i}{i-\mathbf{k}}} \| \vec{\phi} \|_{2}
% \right)  \, ,
\end{align*}
since $ \epsilon <\epsilon^{\frac{i}{i-\mathbf{k}}}$. Using \eqref{eq6.6} we find
\begin{align*}
\|\nabla_{s}^{i} \vec{\phi} \|_{2} & \leq c\left(  \epsilon \| \nabla_s^{\mathbf{k}} \vec{\phi} \|_2 +2 c
  \epsilon  \| \nabla_s^{\mathbf{k}} \vec{\phi} \|_2 \sum_{j=1}^{\mathbf{k}-1}
  \epsilon_j \right. \\
& \quad \left. + c\epsilon  \| \vec{\phi} \ |_2
  \sum_{j=1}^{\mathbf{k}-1}\epsilon_j + c \epsilon \| \vec{\phi} \|_2 \sum_{j=1}^{\mathbf{k}-1}\epsilon_j^{\frac{j}{j-\mathbf{k}}}   +\epsilon^{\frac{i}{i-\mathbf{k}}} \| \vec{\phi} \|_{2}
\right)  \, .
\end{align*}
% where we have used \eqref{eq6.6} for the last inequality.
Choosing $\epsilon_j$ so that $\sum_{j=1}^{\mathbf{k}-1}\epsilon_j
=\frac{1}{2}$ and using $ \epsilon <\epsilon^{\frac{i}{i-\mathbf{k}}}$ we get 
\begin{align*}
\|\nabla_{s}^{i} \vec{\phi} \|_{2} & \leq c (\epsilon \| \nabla_s^{\mathbf{k}} \vec{\phi} \|_2 + \epsilon^{\frac{i}{i-\mathbf{k}}} \| \vec{\phi} \|_{2}).
\end{align*}
Using the estimate obtained above for the case $i=\mathbf{k}$ we deduce
\begin{align*}
\|\nabla_{s}^{i} \vec{\phi} \|_{2} & \leq c \left(\epsilon (\epsilon_1 \|  \vec{\phi} \|_{\mathbf{k}+1,2} + \frac{1}{\epsilon_1^{\mathbf{k}}}\| \vec{\phi} \|_2) + \epsilon^{\frac{i}{i-\mathbf{k}}} \| \vec{\phi} \|_{2} \right).
\end{align*}
By choosing $\epsilon_1= \epsilon^{\frac{1}{\mathbf{k}-i}} <1$ 
we get
\begin{align*}
\|\nabla_{s}^{i} \vec{\phi} \|_{2} & \leq c (\epsilon^{\frac{\mathbf{k}+1-i}{\mathbf{k}-i}} \|  \vec{\phi} \|_{\mathbf{k}+1,2} + \epsilon^{\frac{i}{i-\mathbf{k}}} \| \vec{\phi} \|_{2})= c\left( \tilde{\epsilon} \| \vec{\phi} \|_{\mathbf{k}+1,2} +
\tilde{\epsilon}^{\frac{i}{i-(\mathbf{k}+1)}} \| \vec{\phi} \|_{2} \right),
\end{align*}
where $\tilde{\epsilon}= \epsilon^{\frac{\mathbf{k}+1-i}{\mathbf{k}-i}}$, and the claim follows.
\end{proof}

\begin{lemma}
Let $\vec{\phi}$ be a normal vector field. Then for $p \geq 2$ and for all $k \geq 1$ and $0 \leq i <k$ we have that
\begin{equation}
\label{eq6.7}
\|\nabla_{s}^{i} \vec{\phi} \|_{p}  \leq c \|\nabla_{s}^{i} \vec{\phi} \|_{k-i,2}^{\frac{1}{k-i}(\frac{1}{2}-\frac{1}{p})}  \|\nabla_{s}^{i} \vec{\phi} \|_{2}^{1-\frac{1}{k-i}(\frac{1}{2}-\frac{1}{p})} 
\end{equation}
where c=c(p,n,k).
\end{lemma}

\begin{proof}
We may assume that $\mathcal{L}[f]=1$ and that the curve is parametrized with
respect to arc-length. We distinguish two cases.

\underline{Case $k-i=1$:}  
Using \cite[Thm.5.8]{Adams} we get (for a constant $c=c(p,n,i,k)$) 
\begin{align*}
\|\nabla_{s}^{i} \vec{\phi} \|_{p} =\| \, |\nabla_{s}^{i} \vec{\phi}| \,
\|_{p}  & \leq c \| \,|\nabla_{s}^{i} \vec{\phi}
|\,\|_{W^{1,2}}^{\frac{1}{2}-\frac{1}{p}} \|\nabla_{s}^{i} \vec{\phi}
\|_{L^{2}}^{1-(\frac{1}{2}-\frac{1}{p})} \\
& \leq c\|\nabla_{s}^{i} \vec{\phi} \|_{1,2}^{\frac{1}{2}-\frac{1}{p}}  \|\nabla_{s}^{i} \vec{\phi} \|_{2}^{1-(\frac{1}{2}-\frac{1}{p})}, 
\end{align*}
where we have used Lemma \ref{lemmafacile} for the last inequality.

\underline{Case $ k-i>1$:} 
Using the previous step we infer
\begin{align*}
\|\nabla_{s}^{i} \vec{\phi} \|_{p} &\leq c \|\nabla_{s}^{i} \vec{\phi} \|_{1,2}^{\frac{1}{2}-\frac{1}{p}}  \|\nabla_{s}^{i} \vec{\phi} \|_{2}^{1-(\frac{1}{2}-\frac{1}{p})} \\
&\leq c
\|\nabla_{s}^{i} \vec{\phi} \|_{2} + c \|\nabla_{s}^{i+1} \vec{\phi} \|_{2}^{\frac{1}{2}-\frac{1}{p}}  \|\nabla_{s}^{i} \vec{\phi} \|_{2}^{1-(\frac{1}{2}-\frac{1}{p})}.
\end{align*}
Since $k-i>1$ we can use \eqref{eq6.5} and get
\begin{align*}
\|\nabla_{s}^{i+1} \vec{\phi} \|_{2} =\|\nabla_{s}(\nabla_{s}^{i} \vec{\phi}) \|_{2}
\leq c \|\nabla_{s}^{i} \vec{\phi} \|_{k-i,2}^\frac{1}{k-i} \|\nabla_{s}^{i} \vec{\phi} \|_{2}^{1- \frac{1}{k-i}},
\end{align*}
where we have chosen the constant independent of $i$.
Putting these two estimates together we obtain
\begin{align*}
\|\nabla_{s}^{i}& \vec{\phi} \|_{p} \leq c \|\nabla_{s}^{i} \vec{\phi} \|_{2} + c\|\nabla_{s}^{i} \vec{\phi} \|_{k-i,2}^{\frac{1}{k-i}(\frac{1}{2}-\frac{1}{p})}  \|\nabla_{s}^{i} \vec{\phi} \|_{2}^{(1-\frac{1}{k-i})(\frac{1}{2}-\frac{1}{p})} \|\nabla_{s}^{i} \vec{\phi} \|_{2}^{1-(\frac{1}{2}-\frac{1}{p})}\\
&=c \|\nabla_{s}^{i} \vec{\phi} \|_{2} + c\|\nabla_{s}^{i} \vec{\phi} \|_{k-i,2}^{\frac{1}{k-i}(\frac{1}{2}-\frac{1}{p})}  \|\nabla_{s}^{i} \vec{\phi} \|_{2}^{1-\frac{1}{k-i}(\frac{1}{2}-\frac{1}{p})} \\
&=c\|\nabla_{s}^{i} \vec{\phi} \|_{2}^{\frac{1}{k-i}(\frac{1}{2}-\frac{1}{p})}
\|\nabla_{s}^{i} \vec{\phi} \|_{2}^{1-\frac{1}{k-i}(\frac{1}{2}-\frac{1}{p})}
%\\
%& \quad  
+ c\|\nabla_{s}^{i} \vec{\phi} \|_{k-i,2}^{\frac{1}{k-i}(\frac{1}{2}-\frac{1}{p})}  \|\nabla_{s}^{i} \vec{\phi} \|_{2}^{1-\frac{1}{k-i}(\frac{1}{2}-\frac{1}{p})}  \\
& \leq 
c \|\nabla_{s}^{i} \vec{\phi} \|_{k-i,2}^{\frac{1}{k-i}(\frac{1}{2}-\frac{1}{p})}  \|\nabla_{s}^{i} \vec{\phi} \|_{2}^{1-\frac{1}{k-i}(\frac{1}{2}-\frac{1}{p})} ,
\end{align*}
and the claim follows.
\end{proof}

Now we have all tools at disposal to prove  Lemma \ref{leminter}. 

\noindent \textbf{Lemma \ref{leminter}}
\emph{
Let $f: I \rightarrow \mathbb{R}^n$ be a smooth regular curve. Then for all $k \in \mathbb{N}$, $p \geq 2$ and $0 \leq i<k$ we have
\begin{equation*}
\| \nabla_s^{i} \vec{\kappa}\|_{p} \leq C \|\vec{\kappa}\|_{2}^{1-\alpha} \|\vec{\kappa}\|_{k,2}^{\alpha} \, ,  
\end{equation*}
with $\alpha= (i+\frac12 -\frac{1}{p})/k$ and $C=C(n,k,p)$.}

\begin{proof}
The case $k=1$, $i=0$ is direct consequence of \eqref{eq6.7}. 
For $k \geq 2$, $0 \leq i <k$ we get using again \eqref{eq6.7} that
\begin{align*}
\|\nabla_{s}^{i} \vec{\kappa} \|_{p} &\leq c \|\nabla_{s}^{i} \vec{\kappa} \|_{k-i,2}^{\frac{1}{k-i}(\frac{1}{2}-\frac{1}{p})}  \|\nabla_{s}^{i} \vec{\kappa} \|_{2}^{1-\frac{1}{k-i}(\frac{1}{2}-\frac{1}{p})} \leq c \| \vec{\kappa} \|_{k,2}^{\frac{1}{k-i}(\frac{1}{2}-\frac{1}{p})}  \|\nabla_{s}^{i} \vec{\kappa} \|_{2}^{1-\frac{1}{k-i}(\frac{1}{2}-\frac{1}{p})}.
\end{align*}
But from \eqref{eq6.5} we know that
$
%\begin{align*}
 \|\nabla_{s}^{i} \vec{\kappa} \|_{2} \leq c  \|\vec{\kappa} \|_{k,2}^{\frac{i}{k}}   \| \vec{\kappa} \|_{2}^{\frac{k-i}{k}},
%\end{align*}
$
so that we obtain
\begin{align*}
\|\nabla_{s}^{i} \vec{\kappa} \|_{p} &\leq c \|\vec{\kappa} \|_{k,2}^{\frac{i}{k} -\frac{1}{k-i}(\frac{1}{2}-\frac{1}{p}) \frac{i}{k}} 
\|\vec{\kappa} \|_{k,2}^{ \frac{1}{k-i}(\frac{1}{2}-\frac{1}{p})}  
 \| \vec{\kappa} \|_{2}^{\frac{k-i}{k} -\frac{1}{k}(\frac{1}{2}-\frac{1}{p})}\\
 &=c \|\vec{\kappa} \|_{k,2}^{\frac{i+ \frac{1}{2}-\frac{1}{p}}{k}}
  \| \vec{\kappa} \|_{2}^{1-\frac{i+ \frac{1}{2}-\frac{1}{p}}{k} },
\end{align*}
and the claim follows.
\end{proof}

%%%%%%%%%%%%%%%%%%%%%%%%%%%%%%%%%%%%%%%%%%%%%%%%%%%%%%%%%%%%%%%%%%%%%%%%%%
%%%%%%%%%%%%% Appendix Compatibility %%%%%%%%%%%%%%%%%%%%%%%%%%%%%%%%%%%%%
%%%%%%%%%%%%%%%%%%%%%%%%%%%%%%%%%%%%%%%%%%%%%%%%%%%%%%%%%%%%%%%%%%%%%%%%%%

\section{Compatibility conditions}
\label{AppendixCompa}

In order to have smoothness of the solution of the parabolic problem up to  time $t=0$, the initial data has to satisfy some compatibility conditions at the boundary. These make sure that at the initial time the information given by the boundary conditions agree with those provided by the equation.

First of all we need  to introduce some notation. Let $L$ denote the quasilinear differential operator of fourth order such that $\partial_t f =Lf$ as in \eqref{eqh}. Similarly, for $i \in \mathbb{N}$ let $L^{(i)}$ denote the quasilinear differential operator of order $4i$ such that 
$$\partial_t^i f = L^{(i)} f \, .$$
Let $Q^{(0)}$ denote the following quasilinear second order operator
$$Q^{(0)}(f)=\vec{\kappa}+\langle \zeta, \tau \rangle \tau \, ,$$
with $\zeta \in \mathbb{R}^n$ fixed. For $i \in \mathbb{N}$ let $Q^{(i)}$ be the quasilinear differential operator of order $2+4i$ such that
$$Q^{(i)} = \partial_{t}^{i} Q^{(0)} \, .$$

Then the compatibility condition requires that
\begin{equation}
\label{compcond}
L^{(i)} f_0 =0 \, \mbox{ and } \, Q^{(i)} f_0=0 \quad \mbox{ at } \, x\in \{0,1\} \, \mbox{ for all }i \in \mathbb{N}\, . 
\end{equation}

The existence of  initial data satisfying \eqref{compcond} can be easily proved. We give a couple of examples: if $\zeta=0$ then  the line connecting  $f_0(0)=f_{-}$ and $f_0(1)=f_+$ (with any perturbation $\varphi \in C_0^\infty (\bar{I})$ to it) is a good candidate. If $\zeta \neq 0$  one can take a smooth curve $f_0$ such that in a small neighborhood of the boundary points $f_0$ is  a straight line with tangent equal to $\zeta/|\zeta|$. In this way we obtain that $\vec{\kappa}$ and all its derivatives disappear in proximity of the boundary and the compatibility conditions can be easily verified.

%%%%%%%%%%%%%%%%%%%%%%%%%%%%%%%%%%%%%%%%%%%%%%%%%%%%%%%%%%%%%%%%%%%%%%%%%%%%%%%

%-----------------BIBLIOGRAPHY---------------------------
% for the list of references:
% if you type \nocite* all references contained in references.bib will be listed,
% otherwise only those which are referred to with the \cite{} command appear

% to make references ``active'' remember to:
% >latex file.tex (twice) 
% >bibtex file
% >latex fle.tex

%\nocite*
\bibliography{ref}
\bibliographystyle{acm}
%---------------------------------------------------------

%\newpage

\small

\noindent \textit{Anna Dall'Acqua}, Max Planck Institute for Mathematics in the Sciences, Inselstra\ss e~22,  04103 Leipzig, Germany and Otto-von-Guericke Universit\"at Magdeburg, Universit\"atsplatz 2, 39106 Magdeburg,  \texttt{acqua@mis.mpg.de}

\bigskip 
\noindent \textit{Paola Pozzi}, Universit\"at Duisburg-Essen, Forsthausweg 2, 47057 Duisburg, Germany, \texttt{ paola.pozzi@uni-due.de}

%\noindent Anna Dall'Acqua\\
%\indent Max Planck Institute for Mathematics in the Sciences\\
%\indent Inselstra\ss e 22\\
%\indent D - 04103 Leipzig, Germany\\
%\indent and Otto-von-Guericke Universitaet Magdeburg\\
%\indent Universit\"atsplatz 2\\
%\indent 38106 Magdeburg \\
%\indent \texttt{acqua@mis.mpg.de}\medskip

%\noindent Paola Pozzi\\
%\indent Universit\"at Duisburg-Essen\\ 
%\indent Forsthausweg 2\\
%\indent D-47057 Duisburg, Germany\\
%\indent \texttt{ paola.pozzi@uni-due.de}

%%%%%%%%%%%%%%%%%%%%%%%%%%%%%
\end{document}